\documentclass[12pt]{article}

\usepackage[dvipsnames]{xcolor}

\usepackage[margin={2cm,1.5cm}]{geometry} 

\usepackage{graphicx}
\usepackage{float}
\usepackage{subcaption}

\usepackage{listings}

\usepackage{multirow}
\usepackage{multicol}
\usepackage{arydshln} 

\usepackage[backend=bibtex,firstinits=true,sortcites,style=trad-plain]{biblatex}
\usepackage{microtype}

\usepackage{tikz}
\usepackage{pgfplots}
\usetikzlibrary{spy}

\lstdefinestyle{Python}
{language=Python}

\usepackage{enumerate}
\usepackage{amsmath,amsthm,amsfonts,bm}

\theoremstyle{plain}
\newtheorem{theorem}{Theorem}[section]
\newtheorem{proposition}[theorem]{Proposition}
\newtheorem{lemma}[theorem]{Lemma}
\newtheorem{corollary}[theorem]{Corollary}
\newtheorem{remark}[theorem]{Remark}

\theoremstyle{definition}

\newtheorem{hyp}{Hypothesis}

\usepackage{mathtools}
\newcommand{\R}{\mathbb{R}}

\newcommand{\bv}{\overline{v}}
\newcommand{\bu}{\overline{u}}
\newcommand{\bmu}{\overline{\mu}}
\newcommand{\bn}{\overline{n}}
\newcommand{\buu}{\overline{\mathbf{u}}}
\newcommand{\bvv}{\overline{\mathbf{v}}}
\newcommand{\bw}{\overline{w}}
\newcommand{\bp}{\overline{p}}
\newcommand{\bphi}{\overline{\phi}}

\renewcommand{\div}{\nabla\cdot}

\newcommand{\nn}{\mathbf{n}}
\newcommand{\vv}{\mathbf{v}}
\newcommand{\uu}{\mathbf{u}}

\newcommand{\up}{\Pi_1^h u}

\newcommand{\uKs}{u_{K^*}}
\newcommand{\uL}{u_{L}}

\newcommand{\nKs}{n_{K^*}}

\newcommand{\wK}{w_{K}}

\newcommand{\wL}{w_{L}}
\newcommand{\phiK}{\phi_{K}}
\newcommand{\phiL}{\phi_{L}}

\newcommand{\muK}{\mu_{K}}
\newcommand{\muL}{\mu_{L}}
\newcommand{\mup}{\Pi_0\mu}

\newcommand{\Vh}{\mathcal{V}_h}
\newcommand{\Wh}{\mathcal{W}_h}
\newcommand{\Ph}{\mathcal{P}_h}

\newcommand{\nabland}{\nabla_{\nn_e}^0}

\usepackage[colorlinks=true]{hyperref}
\hypersetup{
linkcolor=BrickRed
,citecolor=Green
,filecolor=Mulberry
,urlcolor=NavyBlue
,menucolor=BrickRed
,runcolor=Mulberry
,linkbordercolor=BrickRed
,citebordercolor=Green
,filebordercolor=Mulberry
,urlbordercolor=NavyBlue
,menubordercolor=BrickRed
,runbordercolor=Mulberry
}

\def\escalar#1#2{\left(#1,#2\right)}

\def\escalarL#1#2{\escalar{#1}{#2}}
\def\escalarLd#1#2{\escalar{#1}{#2}}
\def\escalarML#1#2{\escalar{#1}{#2}_h}
\def\dualH#1#2{\langle#1,#2\rangle}

\def\T{\mathcal{T}}
\def\E{\mathcal{E}}

\def\N{\mathbb{N}}
\def\P{\mathbb{P}}
\def\X{\mathcal{X}}

\def\Ehi{\mathcal{E}_h^\text{i}}

\def\Pd{\mathbb{P}^{\text{disc}}}
\def\Pc{\mathbb{P}^{\text{cont}}}

\def\norma#1{\left\|#1\right\|}
\def\normaL#1{\norma{#1}_{L^2(\Omega)}}

\def\salto#1{\left[\!\left[#1\right]\!\right]}
\def\media#1{\left\{\!\!\left\{#1\right\}\!\!\right\}}

\def\shd#1#2#3#4{s_h^{\eta}(#1,#2,#3,#4)}

\def\ch#1#2#3{c_h(#1,#2,#3)}

\def\aupw#1#2#3{a_h^{\text{upw}}(#1;#2,#3)}
\def\bupw#1#2#3{b_h^{\text{upw}}(#1;#2,#3)}

\def\sign#1{\text{sign}(#1)}

\def\Hdiv{\mathbf{H}_0(\text{div}, \Omega)}

\title{\textbf{A new upwind DG scheme for a Cahn--Hilliard--Darcy model of tumor growth}}
\makeatletter
\def\@fnsymbol#1{\ensuremath{\ifcase#1\or *\or\dagger\or \ddagger\or \mathsection \or\mathparagraph\or *\or **\or \dagger\dagger \or \ddagger\ddagger \else\@ctrerr\fi}}
\makeatother

\author{Daniel Acosta-Soba\thanks{Departamento de Matemáticas, Universidad de Cádiz, Puerto Real, 11510 Cádiz, Spain -- Email: \texttt{daniel.acosta@uca.es} -- Corresponding author}~,
~Francisco Guillén-González\thanks{Departamento de Ecuaciones Diferenciales y Análisis Numérico \& IMUS, Universidad de Sevilla, 41012 Seville, Spain -- Email: \texttt{guillen@us.es}}~,
~J. Rafael Rodríguez-Galván\thanks{Departamento de Matemáticas, Universidad de Cádiz, Puerto Real, 11510 Cádiz, Spain -- Email: \texttt{rafael.rodriguez@uca.es}}}

\begin{document}

\maketitle

\begin{abstract}
In this work, we develop a new numerical scheme for a Cahn--Hilliard--Darcy model that describes tumor growth in a fluid-saturated porous medium. First, we derive a physically consistent model from the general framework proposed in \cite{garcke2016cahn} that guarantees mass conservation and pointwise bounds on the phase-field and nutrient variables, with a decreasing energy law. The resulting model couples the evolution of tumor cells via a Cahn--Hilliard equation with a diffusion equation for the nutrients through chemotactic interactions and extends the model in \cite{acosta2023structure} by introducing the effect of a surrounding fluid described by Darcy's law. Subsequently, we propose a fully discrete scheme that combines an upwind discontinuous Galerkin method in space and a convex splitting strategy in time, which inherits the fundamental properties of the continuous model: mass conservation, pointwise bounds and a discrete energy law that mimics the one of the continuous model up to some residual terms. Our theoretical analysis is accompanied by numerical experiments that demonstrate the robustness of the proposed scheme and show the influence of the surrounding fluid on the tumor evolution.
\end{abstract}

\paragraph{Keywords:} pointwise bounds; energy law; mass conservation; inf-sup condition; upwind discontinuous Galerkin.

\section{Introduction}

Diffuse-interface models have recently been appointed as a successful alternative for tumor growth modeling. In this sense, considerable effort has been made in describing a general framework for the correct modeling and calibration of the diffuse-interface models of tumor growth, oriented to their possible physical application as summarized, for example, in \cite{hawkins2013bayesian,fritz2023tumor,oden2010general,oden2016toward} and the references therein. So far, there are several examples of success in this direction as different approaches have been capable of providing accurate enough results to be compared with real clinical data. Among these celebrated models we can find the work by Pozzi et al. \cite{pozzi2022t} where a Cahn--Hilliard equation is coupled with a Keller--Segel system; the works by Agosti et al. \cite{agosti2018computational,agosti2018personalized,agosti2020learning} where a Cahn-Hilliard equation for the tumor with nonsymmetric degenerate mobility (as the one shown in \cite{acosta2023structure}) is coupled with a diffusion-reaction equation for the nutrients; or the works of Lima et al. \cite{lima2017selection,lima2016selection} where phase-field models are compared against reaction-diffusion models regarding data prediction.

The complexity of these models vary depending on their constitutive assumptions and intrinsic limitations, but most of them are based on multicomponent mixture theory, where the phase-field variable is assumed to be a pointwise-bounded tumor volume fraction, and which accounts for the mass, momentum and energy balances for of each of the constituents. One can trace back these kind of thermodynamically consistent mixture models to the work of Wise et al., \cite{wise2008three} and the references therein. As a consequence of this pioneering work, many other models have arose taking into account different kind of processes and proposing simplifications. For instance, we can find \cite{frieboes2010three} where the model in \cite{wise2008three} is extended to describe angiogenesis and tumor invasion, the work by \cite{hawkins-daarud_numerical_2012} where the nutrients are included as a component of the mixture, or the more recent model in \cite{garcke2021phase} where mechanical effects are also taken into account, just to mention a few.

Among the existing literature, one can find different works which have tried to model the tumor tissue immersed in a fluid that transports the mixture of cells and nutrients. On the one hand, as initially proposed in \cite{wise2008three}, some authors have derived models relying on the Cahn--Hilliard--Darcy equations, \cite{garcke2016cahn, garcke2018multiphase,garcke2018cahn,frigeri2018multi}, where the tissue is assumed to behave as a porous medium. On the other hand, some alternatives have arose for the cases where the tissue cannot be modeled as porous medium, for instance, a Cahn--Hilliard--Brinkman model \cite{colli2023cahn,ebenbeck2021cahn,ebenbeck2019analysis} and, very recently, a Cahn--Hilliard--Navier--Stokes system \cite{elbar2024analysis}. In this sense, even more general models have been proposed where the previous approaches have been generalized to satisfy the Darcy--Forchheimer--Brinkman law, \cite{fritz2019unsteady}, or to introduce viscoelastic effects, \cite{garcke2022viscoelastic}.

However, increasing the complexity of the model leads to  more challenging mathematical problem.
As a consequence, not many
works have been able to provide a successful, unconditionally physically meaningful approximation of these kind of models involving tumor-nutrient interactions and fluid flows. 
In this regard, the works \cite{ebenbeck2021cahn,garcke2018multiphase,garcke2016cahn} propose a finite element bound-preserving discretization that involves solving a discrete variational inequality following the ideas in \cite{barrett1999finite} whereas \cite{garcke2022viscoelastic} introduce a finite element approximation that mimic some entropies of the model at the discrete level. On the contrary, in \cite{elbar2024analysis}, a combination of a suitable phase-field variable transformation, the time-discrete SAV approach and an upwind finite volume spatial discretization is used to preserve the pointwise bounds and the energy-stability in their approximations, under a CFL condition that may be difficult to check in advance.

More broadly, numerical methods for Cahn--Hilliard--Darcy systems have aimed to preserve mass and discrete energy laws. Examples include second-order energy-stable finite difference schemes in time \cite{Chen2019CHHS}, SAV-based finite element schemes \cite{Gao2020CHHS} and mass-conservative, decoupled, energy-stable schemes \cite{HanWang2016CHD,GaoHan2024CHD}.
However, preserving physically admissible pointwise bounds remains challenging. Gao et al. \cite{GaoHanWang2024BoundPreserving} addressed this issue with a second-order, mass-conservative, unconditionally energy-stable and bound-preserving finite element method for the quasi-incompressible CHD system with a Flory--Huggins potential, using a convex-concave splitting and nonlinear stabilization. In this sense, to the best of our knowledge, there are no works  on numerical schemes preserving all these properties with a Ginzburg-Landau potential and a degenerate mobility function.

In this work, we address the problem of developing a physically meaningful approximation for a tumor system coupled with a fluid equation. First, in Section~\ref{sec:derived_model} we derive a new mass-conservative, pointwise-bounded and energy-stable Cahn--Hilliard--Darcy model from the more general model introduced by Garcke et al. in \cite{garcke2016cahn} under certain constitutive assumptions. For this purpose, we extend the ideas the we previously introduced in \cite{acosta2023structure} where a non-symmetric degenerate mobility and proliferation functions, in combination with a Ginzburg-Landau potential, were used to modify the tumor model in \cite{hawkins-daarud_numerical_2012}. Afterwards, in Section~\ref{sec:numerical_approximation}, using the ideas in \cite{acosta2023structure,acosta2025property} we develop an approximation of the model, based on an upwind discontinuous Galerkin discretization, that is mass-conservative, pointwise bounded and satisfies a discrete energy law that mimics the one of the continuous system up to some residuals terms. Finally, in Section~\ref{sec:numerical_experiments}, we present some numerical experiments that show the good performance of the proposed scheme in terms of preserving the physical properties of the model. In this sense, we show how the residual terms in the discrete energy law vanish when the mesh is refined and we carry out some tests to estimate the order of convergence in time and space. In particular, we study computationally how Darcy's velocity affects the evolution of the tumor and the nutrients in comparison to the tests shown in \cite{acosta2023structure} in absence of surrounding fluid.

The main novelties of this work are summarized as follows:
\begin{itemize}
	\item A new Cahn--Hilliard--Darcy model of tumor growth involving degenerate mobility and proliferation functions and a Ginzburg-Landau potential is presented. The resulting model, derived from the general framework presented in \cite{garcke2016cahn}, generalizes the one in \cite{acosta2023structure} by introducing the effect of fluid flow and presents the following physical properties: it is mass-conservative, pointwise-bounded and energy-stable. Afterwards, we propose a fully discrete numerical scheme that preserves the aforementioned fundamental properties up to some residual terms that appear in the discrete energy law. 
	\item We show that, although the approximation of Darcy's velocity variable is discontinuous (unlike what was assumed in previous works \cite{acosta-soba_CH_2022,acosta2025property}), we can ensure that the pointwise bounds are preserved for the transported discrete variables. In particular, the argument shown in Theorem~\ref{thm:discrete_maximum_principle} makes use of the continuity of the velocity field in the direction of the unit normal vector on the edges of the mesh.
	\item Conversely to what is done in \cite{acosta2025property}, we do not introduce any stabilization in the proposed numerical approximation of Darcy's velocity to prevent the artificial diffusion from slowing down the velocity variable. In fact, due to the absence of a viscous term and time derivative in Darcy's equation, this velocity variable is less regular and more sensitive to instantaneous variations than the velocity modeled by the Navier-Stokes equations. As a result, we obtain a weakly energy-stable scheme in the sense that there appears a residual in the energy law of the scheme \eqref{ley_energia_discreta} that is expected to vanish as the mesh is refined. An extensive discussion on the residuals of the energy law and possible stabilization terms is carried out in Remark~\ref{rmk:residuals_energy}. In addition, we provide numerical experiments supporting the claim that the residuals of the energy law vanish as the mesh size $h$ tends to 0 with order, at least, $O(h^2)$.
	\item Existence of solution of the proposed numerical scheme is studied in Theorem~\ref{thm:existencia_solucion_DG-UPW} by making use of the inf-sup condition satisfied by the pair of discrete spaces chosen for the velocity and pressure variables of Darcy's system. 
\end{itemize}

\section{Cahn--Hilliard--Darcy model}
\label{sec:derived_model}

In this section we derive a particular continuous Cahn--Hilliard--Darcy model from the general equations proposed in \cite{garcke2016cahn} using the ideas that we presented in \cite{acosta2023structure} to modify the tumor model introduced in \cite{hawkins-daarud_numerical_2012}. Moreover, we present the physical properties of the resulting model.

\subsection{General model}
H. Garcke et al. introduced in \cite[Section 2.3 (2.25)]{garcke2016cahn} the following diffuse interface model of tumor growth defined in an open,  bounded domain $\Omega\subset \R^d$, $d=2,3$, with a smooth, polygonal boundary $\partial\Omega$ and outwards normal vector $\mathbf{n}$, and over a time interval $(0,T)$, with $T>0$:
\begin{subequations}
	\label{problema:changed_model}
	\begin{align}
		\label{eq:changed_model_v}
		\vv&=-K(\nabla p-\mu_u\nabla u-\partial_n N(u,n)\nabla n) \quad&\text{in }\Omega\times (0,T), \\
		\label{eq:changed_model_p}
		\nabla\cdot\vv&=\Gamma_{\vv} \quad&\text{in }\Omega\times (0,T),\\
		\label{eq:changed_model_u}
		\partial_t u+\nabla\cdot(u\vv)&=\nabla\cdot\left(M_u(u)\nabla\mu_u\right)+\frac{\Gamma_u+\Gamma_{\vv}}{2} \quad&\text{in }\Omega\times (0,T),\\
		\label{eq:changed_model_muu}
		\mu_u&=AF'(u)-B\Delta u+\partial_u N(u,n)\quad&\text{in }\Omega\times (0,T),\\
		\label{eq:changed_model_n}
		\partial_t n+\nabla\cdot(n\vv)&=\nabla\cdot\left(M_n(u)\nabla \partial_nN(u,n)\right)-\mathcal{S} \quad&\text{in }\Omega\times (0,T),\\
		\label{eq:changed_model_bc_u}
		\nabla u\cdot \mathbf{n}&=\left( M_u(u)\nabla \mu_u\right)\cdot \mathbf{n}=0 \quad &\text{on }\partial\Omega\times (0,T),\\
		\label{eq:changed_model_bc_n}
		\left( M_n(u)\nabla \partial_n N(u,n)\right)\cdot \mathbf{n}&=c(n_\infty-n) \quad &\text{on }\partial\Omega\times (0,T),\\
		u(0)&=u_0,\quad n(0)=n_0\quad&\text{in }\Omega.
	\end{align}
\end{subequations}

Here, $\vv$ is the volume-averaged velocity field; $p$ is the pressure; $u$ is a phase-field variable that represents the the volume fraction of tumor cells, where the region $\{x\in\Omega\colon u(x)=1\}$ represents the unmixed tumor and $\{x\in\Omega\colon u(x)=0\}$, the pure healthy area; $n$ is the concentration of chemicals that supply nutrients to the tumor; and $\mu_u$ is the chemical potential of $u$.

Initial data are $u_0\in H^1(\Omega)$ and $n_0\in L^2(\Omega)$ 
 and the boundary condition \eqref{eq:changed_model_bc_n} with  the given supply at the boundary $n_\infty$, and the coefficient $c\ge 0$, allows the entrance/exit of nutrients through the boundary. In particular, if $c=0$ we obtain the zero flux boundary condition and,  as $c\to\infty$ we may formally approach the Dirichlet boundary condition $n=n_\infty$ on $\partial\Omega$.

Moreover, we assume that $F(u)$ is the potential of the phase-field equation, typically the Ginzburg-Landau double well potential, i.e $F(u)=\frac{1}{4}u^2(1-u)^2$, although other choices are possible (see \cite{garcke2016cahn}). Also, $N(u,n)$ models the contribution to the energy of the system of the interaction between the tumor tissue and the nutrients due to different phenomena such as chemotaxis (tumor cells are attracted by nutrients) or active transport of the nutrients (mechanism by which the nutrients are attracted by the tumors).

The terms $\Gamma_{\vv}$ and $\Gamma_u$ are related to the constant densities of the healthy and the tumor tissues, $\rho_{1}$ and $\rho_{2}$, respectively, as follows (see \cite{garcke2016cahn})
$$
\Gamma_{\vv}=\rho_{1}^{-1}\Gamma_{1}+\rho_{2}^{-1}\Gamma_{2},\quad \Gamma_u=\rho_{2}^{-1}\Gamma_{2}-\rho_{1}^{-1}\Gamma_{1}.
$$
Here, $\Gamma_{1}$ and $\Gamma_2$ are the source terms of the mass balance equations for each of the components of the mixture, healthy and tumor cells, respectively. On the other hand, the term $\mathcal{S}$ in \eqref{eq:changed_model_n} is a source/sink term for the nutrients. In addition, the constant $K>0$ in \eqref{eq:changed_model_v} is the permeability coefficient of the tissue.

Finally, note that one boundary condition remains to be imposed in \eqref{problema:changed_model} to ensure its well-posedness. In this sense, if we introduce the natural Dirichlet boundary condition in the normal direction
\begin{equation}\label{bc-vn}
\vv\cdot\nn=g \quad \text{ on }\partial\Omega,
\end{equation}
this condition must satisfy the compatibility restriction 
\begin{equation}
\label{comp_cond}	
\int_{\partial\Omega} g=\int_\Omega\Gamma_{\vv},
\end{equation}
due to \eqref{eq:changed_model_p}. 

\begin{remark}
	Notice that we can regard the equations \eqref{eq:changed_model_v}--\eqref{eq:changed_model_p} as an elliptic problem for the pressure variable. In fact, if we take the divergence of \eqref{eq:changed_model_v} we obtain, using \eqref{eq:changed_model_p}, that
	$$
	\Delta p = \nabla\cdot(\mu_u\nabla u)+\nabla\cdot(\partial_n N(u,n)\nabla n) + \frac{\Gamma_\vv}{K}\quad \text{in }\Omega\times(0,T).
	$$
\end{remark}

\subsection{Constitutive assumptions}

From now on we will make several considerations that will lead us to a generalized version of the model in \cite{acosta2023structure}. We refer the reader to \cite{garcke2016cahn} to explore other possibilities.

First, in order to simplify the model for the numerical analysis (see Remark~\ref{rmk:incompressible}), we consider the case where the fluid is incompressible, i.e., $\nabla\cdot\vv=0$.

Then, following the ideas of \cite{acosta2023structure}, we define the positive and negative parts of a function $v$ as $(v)_\oplus=\max\{0,v\}$ and $(v)_\ominus=\max\{0,-v\}$, respectively, and we introduce the family of degenerate and normalized mobilities
\begin{equation}
	\label{def:mobility}
	M(v)\coloneqq h_{p,q}(v),
\end{equation}
for certain $p,q\ge 1$ where
$$
h_{p,q}(v)\coloneqq K_{p,q}v_\oplus^p(1-v)_\oplus^q=
\begin{cases}
	K_{p,q}v^p(1-v)^q, & v\in[0,1],\\
	0, &\text{elsewhere},
\end{cases}
$$
with $K_{p,q}>0$ a constant so that $\max_{v\in\R}h_{p,q}(v)=1$.
We consider the tumor mobility as $M_u(u)=h_{p,q}(u)$ with $p, q\ge 1$. Moreover, in order to constrain the nutrient variable to the interval $[0,1]$ so that it represents a volume fraction of nutrients (see Proposition~\ref{prop:pointwise_bounds_continuous}), we consider the nutrients mobility as 
$M_n(n)=h_{p',q'}(n)$ with $p', q'\ge 1$.

\begin{remark}
	\label{rmk:mobility}
	Notice that, $p > q$ in $M_u(u)=h_{p,q}(u)$ implies that the mobility is more relevant when $u$ is closer to 1, i.e., in the tumor region, while $p< q$ implies that the mobility is more relevant when $u$ is closer to 0, i.e., in the healthy tissue region. The standard choice $p=q=1$ leads to a symmetric mobility function which gives more importance to the mobility in the intermediate state $u\approx 0.5$. A similar behavior can be observed for the nutrients mobility. See the Figure~\ref{fig:mobility} for a comparison of the different choices of the mobility function.

	\begin{figure}
		\centering
		\begin{tikzpicture}
			\pgfplotsset{
			/pgfplots/xlabel near ticks/.style={
				/pgfplots/every axis x label/.style={
					at={(ticklabel cs:0.5)},anchor=near ticklabel
				}
			},
			/pgfplots/ylabel near ticks/.style={
				/pgfplots/every axis y label/.style={
					at={(ticklabel cs:0.5)},rotate=90,anchor=near ticklabel}
				}
		}
			\begin{axis}[ticklabel style={
				/pgf/number format/fixed,
				/pgf/number format/precision=5,
				font=\tiny,
			}, ylabel near ticks,
			xlabel near ticks,samples=50,grid=major,xlabel={$v$},ylabel={$h_{p,q}(v)$},
			domain=0:1,scale=1.0, legend pos = outer north east]
			\addplot[blue, ultra thick]{4*x*(1-x)};
			\addplot[purple, ultra thick, dashed]{46656/3125*x^5*(1-x)};
			\addplot[orange, ultra thick, dotted]{256/27*x*(1-x)^3};
			\legend{$h_{1,1}(v)$, $h_{5,1}(v)$, $h_{1,3}(v)$}
			\end{axis}
		\end{tikzpicture}
		\caption{Comparison of the function $h_{p,q}(v)$ for different choices of $p$ and $q$.}
		\label{fig:mobility}
	\end{figure}
\end{remark}

Moreover, following the previous work \cite{hawkins-daarud_numerical_2012} we take
$$
N(u,n)=\frac{1}{2\delta}n^2-\chi_0un
$$
for certain $\chi_0\ge 0$ and $\delta>0$. To abbreviate the notation, we define the chemical potential of nutrients as 
\begin{equation}
	\label{def:mu_n}
\mu_n\coloneqq \partial_n N(u,n)=\frac{1}{\delta} n-\chi_0 u.
\end{equation}

In addition, we define the proliferation function depending on both cells and nutrients as
\begin{equation}
	\label{def:proliferation}
	P(u,n)\coloneqq h_{r,s}(u)n_\oplus,
\end{equation}
for certain $r,s\in\N$ that can be chosen to give more importance to the proliferation in the tumor region or in the healthy tissue region as mentioned in Remark~\ref{rmk:mobility}. Then, we take $P_0>0$ as a constant that scales the proliferation function and we define the source terms $\Gamma_2$ and $\mathcal{S}$ as
\begin{align}
\mathcal{S}&=\delta P_0P(u,n)(\mu_n-\mu_u)_\oplus,\\
\Gamma_2&=\rho_2\mathcal{S},
\end{align}
which leads to $\Gamma_u/2=\mathcal{S}$.

In particular, we will assume that $\delta$ is a small positive parameter. In this case, we consider that the active transport of the nutrients towards the tumor is barely negligible with respect to the diffusion of the nutrients in the medium $\Omega$, as $\mu_n\approx \frac{1}{\delta}n$ if $\delta \approx 0$. Conversely, since $\partial_u N(u,n)=-\chi_0 n$, the attraction of the tumor cells towards the nutrients can still be relevant as $\chi_0$ increases. This way, as specified in \cite[Remark 1]{acosta2023structure}, formally, if $\delta\to 0$:
$$
\mathcal{S}=\delta P_0P(u,n)(\mu_n-\mu_u)_\oplus=P_0h_{r,s}(u)n_\oplus(n-\delta\chi_0 u-\delta\mu_u)_\oplus\approx P_0P(u,n)(n_\oplus)^2,
$$
which implies that the reaction terms in equations \eqref{eq:changed_model_u} and \eqref{eq:changed_model_n} are assumed to grow with the square of the nutrients volume fraction. This assumption allows the choice of the source terms to be physically reasonable while they lead to the energy law shown in Proposition~\ref{prop:energy_law_continuous}.

As a result of this choice of the degenerate mobility and proliferation functions, the resulting model will ensure that, if $u_0, n_0\in[0,1]$ in $\Omega$, both variables $u(t)$ and $n(t)$ are restricted to the interval $[0,1]$ for a.e. $t\in(0,T)$ (see Proposition~\ref{prop:pointwise_bounds_continuous}) so that they represent the volume fraction of tumors and nutrients, respectively.

Finally, we choose $A=1$, $B=\varepsilon^2$ where $\varepsilon>0$ characterizes the thickness of the diffuse interface region (\cite[Section 3]{garcke2016cahn}) and we introduce $C_u, C_n>0$ as multiplicative factors that scale the mobility terms of the phase-field and nutrient equations, respectively.

This way,
we can arrive at the following model, which is a extension of the work in \cite{acosta2023structure}:
\begin{subequations}
	\label{problema:modified_model}
	\begin{align}
		\label{eq:modified_model_v}
		\vv&=-K(\nabla p-\mu_u\nabla u-\mu_n\nabla n) \quad&\text{in }\Omega\times (0,T), \\
		\label{eq:modified_model_p}
		\nabla\cdot\vv&=0 \quad&\text{in }\Omega\times (0,T),\\
		\label{eq:modified_model_u}
		\partial_t u+\nabla\cdot(u\vv)&=C_u\nabla\cdot\left(M_u(u)\nabla\mu_u\right)+\delta P_0P(u,n)(\mu_n-\mu_u)_\oplus \quad&\text{in }\Omega\times (0,T),\\
		\label{eq:modified_model_muu}
		\mu_u&=F'(u)-\varepsilon^2\Delta u-\chi_0 n\quad&\text{in }\Omega\times (0,T),\\
		\label{eq:modified_model_n}
		\partial_t n+\nabla\cdot(n\vv)&=C_n\nabla\cdot\left(M_n(n)\nabla\mu_n\right)-\delta P_0P(u,n)(\mu_n-\mu_u)_\oplus \quad&\text{in }\Omega\times (0,T),\\
		\label{eq:modified_model_bc}
		\vv\cdot\nn&=\nabla u\cdot \mathbf{n}=\left( M_n(n)\nabla \mu_n\right)\cdot \mathbf{n}=\left( M_u(u)\nabla \mu_u\right)\cdot \mathbf{n}=0 \quad &\text{on }\partial\Omega\times (0,T),\\
		u(0)&=u_0,\quad n(0)=n_0\quad&\text{in }\Omega,
	\end{align}
\end{subequations}

Notice that we have imposed the ``slip" boundary condition $\vv\cdot \nn=0$ on $\partial\Omega$ in \eqref{eq:modified_model_bc}.

\begin{remark}
	\label{rmk:incompressible}
	The assumption $\nabla\cdot\vv=0$ is a simplification of the original model \eqref{problema:changed_model} that can be non realistic in some biological scenarios but that help us simplify the analysis of the model. In this sense, we refer the reader to \cite[Section 2.4]{garcke2016cahn} for different choices on $\Gamma_1$ and $\Gamma_2$ that leads to $\nabla\cdot\vv\neq 0$.
\end{remark}

\subsection{Variational formulation and properties}

To define the weak formulation of \eqref{problema:modified_model}, we introduce the spaces
$$
\Hdiv =\left\{ \vv\in L^2(\Omega)^d:  \ \nabla\cdot \vv\in L^2(\Omega),\ \vv\cdot\nn=0\ \text{on }\partial\Omega \right\},
$$
$$
L^2_0(\Omega)=\left\{ p\in L^2(\Omega):\ \int_\Omega p =0 \right\}.
$$

Now, we multiply each equation in \eqref{problema:modified_model} by a test function and we integrate by parts the transport and mobility terms in \eqref{eq:changed_model_u} and \eqref{eq:changed_model_n} and the term $-\varepsilon^2\Delta u$ in \eqref{eq:changed_model_muu}. Thus, we arrive at the following variational problem: find $(\vv,p,u,\mu_u,n)$ such that $\vv\in L^2(0,T;\Hdiv)$, 
$p\in L^2(0,T;L^2_0(\Omega))$, 
$u \in L^\infty(0,T;H^1(\Omega))$, $n\in L^2(0,T;H^1(\Omega))$ 
and $\partial_t u,\partial_t n\in L^2(0,T;H^1(\Omega)')$;
and $\mu_u\in L^2(0,T; H^1(\Omega))$,
which satisfies the following variational problem for a.e. $t\in(0,T)$
\begin{subequations}
	\label{problema:mod_model_gamma-0}
	\begin{align}
		\label{eq:mod_model_form_var_v}
		\escalarL{\vv(t)}{\bvv}&=K\left(\escalarL{p(t)}{\nabla\cdot \bvv}\right.\nonumber\\&\left.\quad+\escalarL{\mu_u(t)\nabla u(t)+\mu_n(t)\nabla n(t)}{\bvv}\right), &&\forall\overline{\vv}\in \Hdiv, \\
		\label{eq:mod_model_form_var_p}
		\escalarL{\nabla\cdot\vv(t)}{\bp}&=0, &&\forall\overline{p}\in L^2_0(\Omega),\\
		\label{eq:mod_model_form_var_u}
		\dualH{\partial_t u(t)}{\overline{u}}-\escalarL{u(t)\vv(t)}{\nabla\bu}&=-C_u\escalarLd{M_u(u(t))\nabla\mu_u(t)}{\nabla\overline{u}}\notag\\&\quad+\delta P_0 \escalarL{P(u(t),n(t))(\mu_n(t)-\mu_u(t))_\oplus }{\overline{u}},&&\forall\overline{u}\in H^1(\Omega),\\
		\label{eq:mod_model_form_var_muu}
		\escalarL{\mu_u(t)}{\overline{\mu}_u}&=\varepsilon^2 \escalarLd{\nabla u(t)}{\nabla\overline{\mu}_u}+\escalarL{F'(u(t))-\chi_0 n(t)}{\overline{\mu}_u},&&\forall\overline{\mu}_u\in H^1(\Omega),\\
		\label{eq:mod_model_form_var_n}
		\dualH{\partial_t n(t)}{\overline{n}}-\escalarL{n(t)\vv(t)}{\nabla\bn}&=-C_n\escalarLd{M_n(n(t))\nabla\mu_n(t)}{\nabla\overline{n}}\notag\\&\quad-\delta P_0 \escalarL{P(u(t),n(t))(\mu_n(t)-\mu_u(t))_\oplus }{\overline{n}},&&\forall\overline{n}\in H^1(\Omega),
	\end{align}
\end{subequations}
where 
\begin{equation}
	\label{eq:mod_model_form_var_mun}
	\mu_n(t)=\frac{1}{\delta} n(t) -\chi_0 u(t),
\end{equation}
$u(0)=u_0$, $n(0)=n_0$ and $\escalarL{\cdot}{\cdot}$, $\dualH{\cdot}{\cdot}$ denote the usual scalar product in $L^2(\Omega)$ and the dual product over $H^1(\Omega)$ (or $H^1(\Omega)^d$, as there is no ambiguity), respectively. Note that this variational formulation allows us to impose the boundary conditions \eqref{eq:modified_model_bc} implicitly, except the slip condition $\vv\cdot \nn=0$ which has been imposed by the choice of the space $\Hdiv$ for the velocity field. Moreover, we have imposed $\int_\Omega p=0$ in the choice of the space $L^2_0(\Omega)$ for the pressure variable so that the pressure is uniquely defined.

In particular, this variational formulation allows us to prove the main properties of the model, which are summarized in the following results, and that will be replicated in the numerical approximation of the system in Section~\ref{sec:numerical_approximation}.

\begin{lemma}
	\label{lemma:divergence_free_L2}
	For any function $w\in L^2(\Omega)$, it holds that $\escalarL{\nabla\cdot\vv(t)}{w}=0$ for a.e. $t\in(0,T)$.
\end{lemma}
\begin{proof}
	Notice that, by using \eqref{eq:mod_model_form_var_p} and integrating by parts, we have
	$$0=\escalarL{\nabla\cdot\vv(t)}{w-\int_\Omega w(x) dx} =\escalarL{\nabla\cdot\vv(t)}{w}-\int_\Omega w(x) dx\int_\Omega\nabla\cdot\vv(x,t) dx=\escalarL{\nabla\cdot\vv(t)}{w}.$$
\end{proof}

\begin{proposition}[Mass conservation]
	\label{prop:mass_conservation_continous}
	Let $(\vv,p,u,\mu_u,n)$ be a solution of the problem \eqref{problema:mod_model_gamma-0}. Then, this solution conserves the total mass of tumor cells plus nutrients in the sense of $$\frac{d}{dt}\int_\Omega (u(x,t)+n(x,t))dx=0.$$
\end{proposition}
\begin{proof}
	It is enough to take $\overline{u}=\overline{n}=1$ in \eqref{eq:mod_model_form_var_u} and \eqref{eq:mod_model_form_var_n} and add the resulting expressions.
\end{proof}

\begin{proposition}[Pointwise bounds]
	\label{prop:pointwise_bounds_continuous}
	Let $(\vv,p,u,\mu_u,n)$ be a solution of the problem \eqref{problema:mod_model_gamma-0}. Then, this solution satisfies that $u(t), n(t) \in [0,1]$ in $\Omega$ for a.e. $t\in(0,T)$ provided that $u_0,n_0\in[0,1]$ in $\Omega$.
\end{proposition}
\begin{proof}
	Let us show that $n(t)\in[0,1]$ for a.e. $t\in(0,T)$ provided that $n_0\in[0,1]$ in $\Omega$. The proof for $u(t)$ is analogous.

	First, to show $n(t)\ge 0$, notice that $n_\ominus\in L^2(0,T, H^1(\Omega))$ and take $\overline{n}=n(t)_\ominus$ in \eqref{eq:mod_model_form_var_n}. We arrive at
	$$
	\int_\Omega \partial_t n(x,t)n(x,t)_\ominus dx - \int_\Omega n(x,t)\vv(x,t)\cdot\nabla n(x,t)_\ominus dx = 0.
	$$
	Now, using 
	Lemma~\ref{lemma:divergence_free_L2}, we have
	\begin{align*}
	\int_\Omega n(x,t)\vv(x,t)\cdot\nabla n(x,t)_\ominus =& -\int_\Omega n(x,t)_\ominus\vv(x,t)\cdot\nabla n(x,t)_\ominus = -\frac{1}{2}\int_\Omega \vv(x,t)\cdot\nabla (n(x,t)_\ominus)^2=0.
	\end{align*}

	Therefore,
	$$
	0 = \int_\Omega \partial_t n(x,t)n(x,t)_\ominus dx = -\frac{1}{2}\frac{d}{dt}\int_\Omega (n(x,t)_\ominus)^2dx,
	$$
	which implies $\normaL{n(t)_\ominus}=\normaL{n(0)_\ominus}=0$. Hence, $n(t)\ge 0$ for a.e. $t\in(0,T)$.

	Now, to show that $n(t)\le 1$, take $\overline{n}=(1-n(t))_\ominus$ in \eqref{eq:mod_model_form_var_n}. We arrive at
	$$
	\int_\Omega \partial_t n(x,t)(1-n(x,t))_\ominus dx - \int_\Omega n(x,t)\vv(x,t)\cdot\nabla (1-n(x,t))_\ominus dx \le 0,
	$$
	where, using again Lemma~\ref{lemma:divergence_free_L2},
	\begin{align*}
	\int_\Omega n(x,t)\vv(x,t)\cdot\nabla (1-n(x,t))_\ominus dx =& -\int_\Omega(1-n(x,t))_\ominus\nabla n(x,t)\cdot\vv(x,t)dx
	\\=& - \frac{1}{2}\int_\Omega \vv(x,t)\cdot\nabla ((1-n(x,t))_\ominus)^2dx
	=0.
	\end{align*}

	Thus,
	$$
	0 \ge \int_\Omega \partial_t n(x,t)(1-n(x,t))_\ominus dx = -\int_\Omega \partial_t (1-n(x,t))(1-n(x,t))_\ominus dx = \frac{1}{2}\frac{d}{dt}\int_\Omega ((1-n(x,t))_\ominus)^2dx,
	$$
	which implies $\normaL{(1-n(t))_\ominus}\le\normaL{(1-n(0))_\ominus}=0$. Hence, $n(t)\le 1$ for a.e. $t\in(0,T)$.
\end{proof}

\begin{remark}
	\label{rmk:pointwise_bounds}
	Notice that the proof of Proposition~\ref{prop:pointwise_bounds_continuous} relies on the following:
	\begin{itemize}
	\item The mobility functions $M_u(u)$ and $M_n(n)$ degenerate when $u$ and $n$ reach the values $0$ or $1$.
	\item The proliferation function $P(u,n)$ degenerates when $u$ reaches the value $0$ or $1$ and when $n$ reaches the value $0$. Moreover, the source term in the nutrient equation \eqref{eq:mod_model_form_var_n} is non-positive, which prevents $n$ from reaching values greater than $1$.
	\item The velocity field $\vv$ is divergence-free, which allows us to cancel the transport terms in the proof of pointwise bounds. Notice that, in general, if $\nabla\cdot \vv\ge 0$, the same result would hold following similar arguments. 
	Conversely, we believe that, in general, if $\nabla\cdot \vv< 0$, the upper bound may not be satisfied as it occurs with the standard convection equation (see, for instance, \cite[Chapter 3]{pironneau1989finite}).
	\end{itemize}
	As such, the sign of $\nabla \cdot \vv$ and the choice of the mobility and proliferation functions are crucial to ensure the pointwise bounds of the solution. Consequently, based on this proof, it is not clear whether different consitutive assumptions on the divergence of the velocity field that allow that $\nabla \cdot \vv$ becomes negative or other non-degenerate choices of the mobility and proliferation functions (for instance, in combination with different double well potentials such as the Flory--Huggins potential) would lead to a solution that satisfies the pointwise bounds.
\end{remark}
	
\begin{proposition}[Energy law]
	\label{prop:energy_law_continuous}
	Let $(\vv,p,u,\mu_u,n)$ be a solution of the problem \eqref{problema:mod_model_gamma-0}.
	Then, it satisfies the following energy law
	\begin{align}
		\label{ley_energia_continua}
		\frac{d E(u(t),n(t))}{dt}&
		+C_u\int_\Omega M_u(u(x,t))|\nabla\mu_u(x,t)|^2dx
		+C_n\int_\Omega M_n(n(x,t))|\nabla\mu_n(x,t)|^2 dx \notag
		\\&
		+\delta P_0\int_\Omega P(u(x,t),n(x,t))(\mu_u(x,t)-\mu_n(x,t))_\oplus ^2dx+\frac{1}{K}\int_\Omega\vert\vv(x,t)\vert^2dx
		\notag\\&=0
	\end{align}
	where the energy functional is defined by 
	\begin{align}
		\label{energia}
		E(u,n)&\coloneqq\int_\Omega\left(\frac{\varepsilon^2}{2}|\nabla u|^2+ F(u)-\chi_0 u\, n 
		+\frac{1}{2\delta}n^2\right).
	\end{align}
	Therefore, the solution is energy stable in the sense $$\frac{d}{dt}E(u(t),n(t))\le 0.$$
\end{proposition}
\begin{proof}
	First, notice that, using Lemma~\ref{lemma:divergence_free_L2} and integrating by parts, we get $\escalarL{\mu_u\nabla u}{\vv}=\escalar{\mu_u}{\nabla\cdot(u\vv)}=-\escalar{u\vv}{\nabla\mu_u}$. Similarly, $\escalarL{\mu_n\nabla n}{\vv}=-\escalar{n\vv}{\nabla\mu_n}$.
	
	Now, take $\bvv=\frac{1}{K}\vv(t)$ in \eqref{eq:mod_model_form_var_v},$\bp=p(t)$, $\bu=\mu_u(t)$, $\bmu_u=\partial_t u(t)$, $\bn=\mu_n(t)$ in \eqref{eq:mod_model_form_var_v}--\eqref{eq:mod_model_form_var_n} and test \eqref{eq:mod_model_form_var_mun} by $\partial_t n(t)$.
	Adding the resulting expressions we arrive at
	\begin{align*}
		0=&\varepsilon^2\escalarLd{\nabla u(t)}{\nabla(\partial_t u(t))} + \escalarL{F'(u(t))}{\partial_t u(t)} -\chi_0\left[\escalarL{n(t)}{\partial_tu(t)} + \escalarL{u(t)}{\partial_t n(t)}\right]+\frac{1}{\delta}\escalarL{n(t)}{\partial_t n(t)}\\&+C_u\int_\Omega M_u(u(x,t))\vert\nabla\mu_u(x,t)\vert^2dx+C_n\int_\Omega M_n(n(x,t))\vert\nabla\mu_n(x,t)\vert^2dx\\& +\delta P_0\int_\Omega P(u(x,t),n(x,t))(\mu_u(x,t)-\mu_n(x,t))_\oplus (\mu_u(x,t)-\mu_n(x,t))dx+\frac{1}{K}\int_\Omega\vert\vv(x,t)\vert^2 dx.
	\end{align*}
	Therefore, it is straightforward to check that \eqref{ley_energia_continua} holds.
\end{proof}

Finally, we recall the following inf-sup condition that will be used in the numerical analysis of the model.
\begin{lemma}[\cite{boffi2013mixed,ern_theory_2010}]
	There is a constant $\beta>0$ such that, for every $p\in L^2_0(\Omega)$, it holds that
	\begin{equation}
		\label{eq:inf-sup_continuous}
		\beta \|p\|_{L^2(\Omega)} \le \sup_{\vv\in \Hdiv\setminus\{0\}} \frac{(p,\nabla\cdot \vv)}{\|\vv\|_{\Hdiv}}.
	\end{equation}
\end{lemma}

\section{Numerical approximation}
\label{sec:numerical_approximation}

We discretize time by partitioning $[0,T]$ into equally-spaced intervals $0=t_0<t_1<\cdots<t_N=T$ with step size $\Delta t=t_{m+1}-t_m$. For a scalar function $w=w(t)$ on $[0,T]$, we denote $w^m\simeq w(t_m)$ and use $\delta_t w^{m+1}=(w^{m+1}-w^m)/\Delta t$ as the discrete time derivative. To handle the double well potential $F(u)$ efficiently, we employ a convex splitting (see, for instance, \cite{eyre_1998_unconditionally, guillen-gonzalez_linear_2013,acosta-soba_CH_2022}, for more details) that decomposes it as $F(u)=F_i(u)+F_e(u)$ for $u\in[0,1]$, where $F_i(u)=3 u^2/8$ is the convex part (treated implicitly) and $F_e(u)=\frac{1}{4}u^4-\frac{1}{2}u^3-\frac{1}{8}u^2$ is the concave part (treated explicitly). This decomposition yields
\begin{equation}
	\label{convex-split}
f(u^{m+1},u^m)=F_i'(u^{m+1})+F_e'(u^m)=\frac{1}{4}\left(3 u^{m+1}+4 (u^m)^3-6(u^m)^2-u^m\right).
\end{equation}

For the spatial discretization, we employ a shape-regular conforming mesh $\T_h=\{K\}_{K\in \T_h}$ in the sense of Ciarlet \cite[Sections 3.1 and 3.2]{ciarlet2002finite} of size $h>0$ (maximum diameter of the elements in $\T_h$) over $\Omega$. For the sake of readability, we will abuse the notation and denote by $\Omega$ the domain that is being approximated by the mesh, instead of $\Omega_h$. The set of all mesh edges is partitioned into interior edges $\E_h^{\text{i}}$ and boundary edges $\E_h^{\text{b}}$. Each interior edge $e\in\E_h^{\text{i}}$ shared by elements $K$ and $L$ is equipped with a unit normal vector $\nn_e$ pointing outward from $K$ into $L$, while the unit normal vector on the boundary edges is oriented outward from $\Omega$. We impose the following condition on the mesh:
\begin{hyp}
	\label{hyp:mesh}
	The line between the barycenters of any adjacent elements $K$ and $L$ is orthogonal to the interface $e=K\cap L\in\E_h^{\text{i}}$.
\end{hyp}
Some examples of meshes satisfying this Hypothesis~\ref{hyp:mesh} can be found in \cite{acosta-soba_KS_2022} and in Figure~\ref{fig:test-2_initial_cond_interface}. Note that this hypothesis is easy to fulfill in two or three dimensions using meshes based on squares or cubes. In particular, in the case of $d=2$, the mesh can be refined by dividing each square into triangles along the diagonals. However, this kind of meshes is restrictive and may not be suitable to approximate complex geometries or curved boundaries. In this sense, we are working on how to relax this hypothesis in a future work.

On each edge, we define the \textit{average} $\media{\cdot}$ and the \textit{jump} $\salto{\cdot}$ of a scalar function $w$ on an edge $e\in\E_h$ in the standard manner:
\begin{equation*}
		\media{w}\coloneqq
		\begin{cases}
			\dfrac{\wK+\wL}{2}&\text{if } e\in\E_h^{\text{i}}\\
			\wK&\text{if }e\in\E_h^{\text{b}}
		\end{cases},
		\qquad
		\salto{w}\coloneqq
		\begin{cases}
			\wK-\wL&\text{if } e\in\E_h^{\text{i}}\\
			\wK&\text{if }e\in\E_h^{\text{b}}
		\end{cases}.
\end{equation*}

For the finite element spaces, we use $\Pd_k(\T_h)$ and $\Pc_k(\T_h)$ to denote spaces of degree-$k$ piecewise-polynomial functions that are discontinuous and continuous, respectively, over the edges of the mesh, $\E_h$ (see \cite{di_pietro_ern_2012} for more details on discontinuous Galerkin methods). We further define two operators: a $L^2$-projection $\Pi_0\colon L^1(\Omega) \rightarrow \Pd_0(\T_h)$ and a mass-lumped regularization $\Pi^h_1\colon L^1(\Omega)\rightarrow \Pc_1(\T_h)$ of a function $g\in L^1(\Omega)$ as the function satisfying
\begin{align}
	\label{eq:esquema_DG_Pi0}
	\escalarL{g}{\overline{w}}&=
	\escalarL{\Pi_0 g}{\overline{w}},&\forall\,\overline{w}\in \Pd_0(\T_h),
	\\
	\label{eq:esquema_DG_Pih1}
\escalarL{g}{\overline{\phi}}&=\escalarML{\Pi^h_1 g}{\overline{\phi}},&\forall\,\overline{\phi}\in \Pc_1(\T_h),
\end{align}
where $\escalarML{\cdot}{\cdot}$ is the mass-lumping scalar product in $\Pc_1(\T_h)$. 
In fact, $(\Pi_0 g)|_K =( \int_K g)/|K|$  for all $K\in \T_h $, 
and 
$(\Pi^h_1 g)(a_j)=(\sum_{K\in \mathcal{T}_{a_j}} \int_K g\ \varphi_j) /(\sum_{K\in \mathcal{T}_{a_j}}|K|/(d+1) )$
 for all vertices $a_j$ with $\varphi_j $ the canonical basis of $\Pc_1(\T_h)$ and $\mathcal{T}_{a_j} := \{ K \in \mathcal{T}_h : a_j \in \overline{K} \}$ the set of elements sharing the vertex $a_j$.

We define
$\ch{w}{\mu}{\bvv}$ as a centered discretization of the term $-\escalarL{\mu\nabla w}{\bvv}=\escalarL{w}{\nabla\cdot(\mu\bvv)}$ in \eqref{eq:mod_model_form_var_v}
defined as
\begin{equation}
	\label{centered_discretization}
	\ch{w}{\mu}{\bvv}\coloneqq 
	 \int_\Omega\nabla\cdot(\mu\bvv) w
	-\sum_{e\in\Ehi}\int_e (\bvv\cdot\nn_e)\media{w}\salto{\mu},
\end{equation}
where a consistent term involving the jumps of $w$ on the interior edges of the mesh $\T_h$ have been dropped.

Moreover, we consider two different upwind formulas, the classical upwind 
\begin{align}
	\label{def:aupw}
	\aupw{\vv}{\phi}{\bphi}&\coloneqq \sum_{e\in\Ehi, e=K\cap L}\int_e\left( (\vv\cdot\nn_e)_\oplus\phi_K - (\vv\cdot\nn_e)_\ominus\phi_L\right)\salto{\bphi}
\end{align}
whose properties were discussed in \cite{acosta-soba_CH_2022}, and 
$$\bupw{\nabla_{\nn}^0\mu }{M(\phi)}{\bphi},$$
 which follows the ideas introduced in \cite{acosta-soba_KS_2022,acosta2023structure} and is detailed in the next lines.

In particular, following the ideas in \cite{acosta-soba_CH_2022,mazen_saad_2014}, in order to preserve the pointwise bounds using an upwind approximation of the convective term
$$
-\escalarL{M(\phi)\nabla\mu}{\nabla\bw},\quad \bw\in H^1(\Omega),
$$
we split the mobility function into its increasing and its decreasing part as follows:
$$
M^\uparrow(\phi)=
\begin{cases}
	M(\phi),& \phi\le \phi^*,\\
	M(\phi^*),& \phi>\phi^*,
\end{cases}
\quad
M^\downarrow(\phi)=
\begin{cases}
	0,& \phi\le \phi^*,\\
	M(\phi)-M(\phi^*),& \phi>\phi^*,
\end{cases}
$$
where  $\phi^*\in\R$ is the point where the maximum of $M(\phi)$ is attained, which can be obtained
by simple algebraic computations. Note that 
 $M(\phi)=M^\uparrow(\phi)+M^\downarrow(\phi)$.

Now, we define
\begin{multline}
	\label{def:bupw}
	\bupw{\mu}{M(\phi)}{\bw}\coloneqq\\ \sum_{e\in\E_h^i,e=K\cap L}\int_e\left(\left(-\nabland\mu\right)_\oplus \left(M^\uparrow(\phiK) + M^\downarrow(\phiL)\right)_\oplus-(-\nabland\mu)_\ominus \left(M^\uparrow(\phiL) + M^\downarrow(\phiK)\right)_\oplus\right)\salto{\bw}
\end{multline}
with
\begin{equation}
\label{eq:approx_gradn}
\nabland\mu =\frac{-\salto{\mu}}{\mathcal{D}_e(\T_h)}= \frac{\muL-\muK}{\mathcal{D}_e(\T_h)},
\end{equation}
a reconstruction of the normal gradient using $\P_0(\T_h)$ functions for every $e\in\E_h^i$ with $e=K\cap L$ (see \cite{acosta-soba_KS_2022} for more details). We have denoted $\mathcal{D}_e(\T_h)$ the distance between the barycenters of the elements $K$ and $L$ of the mesh $\T_h$ that share $e\in\Ehi$.
This way, we can rewrite \eqref{def:bupw} as
\begin{multline}
\label{def:aupw_barycenter}
	\bupw{\mu}{M(\phi)}{\bw}\coloneqq \\\sum_{e\in\E_h^i,e=K\cap L}\frac{1}{\mathcal{D}_e(\T_h)}\int_e\left(\salto{\mu}_\oplus \left(M^\uparrow(\phiK) + M^\downarrow(\phiL)\right)_\oplus-\salto{\mu}_\ominus \left(M^\uparrow(\phiL) + M^\downarrow(\phiK)\right)_\oplus\right)\salto{\bw}.
\end{multline}
Notice that, as in \cite{acosta2023structure}, we have also truncated the mobility $M(w)$
to avoid negative approximations of $M(w)$ that may lead to a loss of energy stability. This truncation is not explicitly written in $\bupw{\mu}{M(w)}{\bw}$ to simplify the notation.

Using the previous definitions and considering $\Vh$ and $\Ph$ discrete spaces for the velocity and pressure variables, we propose the following fully discrete scheme for the model \eqref{problema:modified_model}: given $u^m,n^m\in\Pd_0(\T_h)$, 
 find $\vv^{m+1}\in\Vh$ with $\vv^{m+1}\cdot \nn=0$ on $\partial\Omega$, $p^{m+1}\in\Ph$ with $\int_\Omega p^{m+1}=0$, $u^{m+1}, n^{m+1}\in \Pd_0(\T_h)$ and $\mu_u^{m+1} \in \Pc_1(\T_h)$, such that
\begin{subequations}
	\label{esquema_DG_mod_model}
	\begin{align}
		\label{eq:esquema_DG_mod_model_v}
		\frac{1}{K}\escalarL{\vv^{m+1}}{\bvv}
		&-\escalarL{p^{m+1}}{\nabla \cdot \bvv}
		+\ch{u^{m+1}}{\Pi_0\mu_u^{m+1}}{\bvv}+\ch{n^{m+1}}{\mu_n^{m+1}}{\bvv}=0,
		\\
		\label{eq:esquema_DG_mod_model_p}
		\escalarL{\nabla\cdot \vv^{m+1}}{\bp}&=0,\\
		\label{eq:esquema_DG_mod_model_u}
		\escalarL{\delta_tu^{m+1}}{\overline{u}}&
		=-\aupw{\vv^{m+1}}{u^{m+1}}{\bu}-C_u\bupw{\Pi_0\mu_u^{m+1}}{M_u(u^{m+1})}{\overline{u}}&\notag\\&\quad+\delta P_0\escalarL{P(u^{m+1},n^{m+1})(\mu_n^{m+1}-\Pi_0\mu_u^{m+1})_\oplus }{\overline{u}}\\
		\label{eq:esquema_DG_mod_model_muu}
		\escalarML{\mu_u^{m+1}}{\overline{\mu}_u}&=\varepsilon^2 \escalarLd{\nabla \up^{m+1}}{\nabla\overline{\mu}_u}+ \escalarL{f(\up^{m+1},\up^m)}{\overline{\mu}_u}-\chi_0\escalarL{ n^{m+1}}{\overline{\mu}_u},\\
		\label{eq:esquema_DG_mod_model_n}
		\escalarL{\delta_t n^{m+1}}{\overline{n}}&=-\aupw{\vv^{m+1}}{n^{m+1}}{\bn}-C_n\bupw{\mu^{m+1}_n}{M_n(n^{m+1})}{\overline{n}}&\notag\\&\quad-\delta P_0 \escalarL{P(u^{m+1},n^{m+1})(\mu_n^{m+1}-\Pi_0\mu_u^{m+1})_\oplus }{\overline{n}},
	\end{align}
\end{subequations}
for every $\bvv\in\Vh$ with $\bvv\cdot \nn=0$, $\bp\in\Ph$ with $\int_\Omega \bp =0$, $\bu\in\Pd_0(\T_h)$, 
$\overline{\mu}_u\in\Pc_1(\T_h)$ and $\bn\in\Pd_0(\T_h)$, where $u^0=\up_0,n^0=\Pi_0n_0$,
\begin{equation}
	\label{eq:esquema_DG_mod_model_mun}
	\mu_n^{m+1}=\frac{1}{\delta}n^{m+1}  -\chi_0 \Pi_0(\up^{m}).
\end{equation}

We take $(\Vh,\Ph)$ a compatible ``Darcy inf-sup" pair of finite-dimensional spaces in the sense that, following the inf-sup condition from the continuous case \eqref{eq:inf-sup_continuous}, there is $\beta>0$ such that for every $p\in\Ph$ it holds
\begin{equation}
	\label{eq:inf-sup_discrete}
	\beta \|p\|_{L^2(\Omega)} \le \sup_{\vv\in \Vh\setminus\{0\}} \frac{(p,\nabla\cdot \vv)}{\|\vv\|_{\Hdiv}},
\end{equation}
with $\Vh\subset \Hdiv$, $\Ph\subset L^2(\Omega)$ and $\div\Vh\subset\Ph$. Moreover, we assume that $(\Vh,\Ph)$ satisfy that the normal components of the velocity are continuous on $\Ehi$ and $\Pd_0(\T_h) \subset \Ph $. For instance, one can choose the following:
\begin{itemize}
	\item $\Vh=RT_k(\T_h)$ and $\Ph=\Pd_k(\T_h)$, for $k\ge 0$,
	where $RT_k(\T_h)$ denotes the Raviart-Thomas space of order $k$ (see \cite{boffi2013mixed,raviart1977mixed,brezzi2004piecewise} for more details).
	\item $\Vh=BDM_k(\T_h)$ and $\Ph=\Pd_{k-1}(\T_h)$, for $k\ge 1$, where $BDM_k(\T_h)$ denotes the Brezzi-Douglas-Marini space of order $k$ (see \cite{boffi2013mixed,brezzi1985two,brezzi2004piecewise} for more details).
\end{itemize}

In fact, either of these choices guarantee the local incompressibility of $\vv^{m+1}$ in the following sense:
\begin{equation}
	\label{local_incompressibility}
	\sum_{e\in\Ehi}\int_{e} (\vv^{m+1}\cdot\nn_e) \salto{\bp}=0,\quad\forall \,\bp\in\Pd_0(\T_h),
\end{equation}
which can be derived integrating by parts in \eqref{eq:esquema_DG_mod_model_p}.
This constraint will allow us to preserve the point-wise bounds of $u^{m+1}$ and $n^{m+1}$, see Theorem \ref{thm:discrete_maximum_principle} below.

Notice that, for any choice of this pair $(\Vh,\Ph)$, the error bounds are expected to be determined by the lowest accuracy approximation
of the phase-field and nutrient variables by $\Pd_0(\T_h)$ (look at Table~\ref{tab:space_order}).

In this sense, we have chosen the approximation of the phase-field and nutrient variables by piecewise constant functions to preserve the pointwise bounds of $u^{m+1}$ and $n^{m+1}$, see Theorem \ref{thm:discrete_maximum_principle} below. Conversely, the approximation of the chemical potential $\mu_u^{m+1}$ by piecewise linear functions, in addition to the projection $\up^{m+1}$, allows us to approximate the term $\Delta u$ in \eqref{eq:esquema_DG_mod_model_muu} while reducing the computational cost that discontinuous methods would incur. Finally, the projections of $\mu_u^{m+1}$ and $\mu_n^{m+1}$ on the space of piecewise constant functions is needed to guarantee the energy law shown in Theorem \ref{thm:energia_esquema} below.

To start the algorithm we take $u^0=\Pi_0u_0$, $n^0=\Pi_0n_0$ where $u_0, n_0$ are the continuous initial data, which  satisfy $u_0,n_0\in[0,1]$. Notice that, one also has $u^0,n^0\in[0,1]$.

Note that,  $\mu_n^{m+1}\in \Pd_0(\T_h)$, 
\begin{equation} \label{Pi_1}
(\Pi^h_1 u^m)(a_j)=\frac{\sum_{L\in \mathcal{T}_{a_j}} |L| u^m_L} {\sum_{L\in \mathcal{T}_{a_j}}|L| },
\quad
\forall\, a_j,
\end{equation}
and
\begin{equation} \label{Pi_0Pi_1}
\Pi_0(\up^{m})|_K=\frac1{d+1}\frac1K \sum_{a_j\in K}(\Pi^h_1 u^m)(a_j) ,
\quad
 \forall\, K\in \T_h.
\end{equation} 

Notice that we have introduced the regularization of $u^{m+1}$, $\up^{m+1}$ to preserve the diffusion term in \eqref{eq:esquema_DG_mod_model_u}.
In fact, this regularized variable will be regarded as our approximation of the tumor cells volume fraction as, according to the results in Subsection \ref{sec:properties_DG_scheme}, it preserves the pointwise bounds and satisfies a discrete energy law. Moreover, in order to preserve the pointwise bounds and the dissipation of the energy, we consider mass lumping in the term $\escalarML{\mu_u^{m+1}}{\overline{\mu}}$.

\begin{remark}
	The homogeneous Neumann boundary conditions on the fluxes of $u^m$ and $n^m$ have been implicitly imposed in the definition of $\bupw{\cdot}{\cdot}{\cdot}$, see \eqref{def:bupw}. In addition, the boundary condition $\nabla \up^{m}\cdot\nn=0$ on $\partial\Omega\times(0,T)$ is imposed implicitly by the term $\escalarL{\nabla \up^{m}}{\nabla\overline{\mu}_u}$ in \eqref{eq:esquema_DG_mod_model_u}.
	Finally, the boundary condition $\vv^{m+1}\cdot\nn=0$ on $\partial\Omega\times(0,T)$ is strongly imposed in \eqref{esquema_DG_mod_model}.
\end{remark}

\begin{remark}
 In practice, the $0$-mean value constraint on the pressure of the discrete formulation \eqref{esquema_DG_mod_model} will be removed, solving the singular system and the constraint is imposed by post-processing.
\end{remark}

\begin{remark}
	The scheme \eqref{esquema_DG_mod_model} is nonlinear so we will have to use an iterative procedure, such as Newton's method, to approach its solution.
\end{remark}

\subsection{Properties of the fully discrete scheme}
\label{sec:properties_DG_scheme}

\begin{proposition}[Mass conservation]
	The scheme \eqref{esquema_DG_mod_model} conserves the total mass of cells plus nutrients in the following sense: for all $m\ge 0$,
	\begin{align*}
	\int_\Omega (u^{m+1}+n^{m+1})=\int_\Omega (u^m+n^m)\quad\text{and}\quad\int_\Omega (\up^{m+1}+n^{m+1})=\int_\Omega (\up^m+n^m).
	\end{align*}
\end{proposition}
\begin{proof}
	Just need to take $\overline{u}=1$ in \eqref{eq:esquema_DG_mod_model_u} and $\overline{n}=1$ in \eqref{eq:esquema_DG_mod_model_n} and add both expressions to obtain:
	\begin{align*}
		\int_\Omega (u^{m+1}+n^{m+1})=\int_\Omega (u^m+n^m).
	\end{align*}
	Moreover, due to the definition of the regularization $\Pi_1^h$, we have that $\int_\Omega u^{m+1}=\int_\Omega \up^{m+1}$ and $\int_\Omega u^{m}=\int_\Omega \up^{m}$, what yields
	\begin{align*}
		\int_\Omega (\up^{m+1}+n^{m+1})=\int_\Omega (\up^m+n^m).
	\end{align*}
\end{proof}

\begin{theorem}[Pointwise bounds]
	\label{thm:discrete_maximum_principle}
	
	Let $(\vv^{m+1}, u^{m+1}, \mu_u^{m+1}, n^{m+1})$ be a solution of the scheme \eqref{esquema_DG_mod_model}, then $u^{m+1}, n^{m+1}\in[0,1]$ in $\overline\Omega$ provided $u^{m},n^{m}\in[0,1]$ in $\overline\Omega$.
\end{theorem}
\begin{proof}
	Firstly, we prove
	that $u^{m+1}, n^{m+1}\ge 0$ in $\Omega$.

	To prove that $u^{m+1}\ge 0$ we may take the following $\Pd_0(\T_h)$ test function
	\begin{align*}
		\overline{u}^*=
		\begin{cases}
			(\uKs^{m+1})_\ominus &\text{in }K^*,\\
			0&\text{out of }K^*,
		\end{cases}
	\end{align*}
	where $K^*$ is an element of $\T_h$ such that  $\uKs^{m+1}=\min_{K\in\T_h}u_{K}^{m+1}$. Then, by definition of $P(u,n)$ in \eqref{def:proliferation}, one has 
	$$
	\delta P_0\escalarL{P(u^{m+1},n^{m+1}) (\mu_n^{m+1}-\mu_u^{m+1})_\oplus }{\overline{u}^*}=0,
	$$
	hence equation \eqref{eq:esquema_DG_mod_model_u} becomes
	\begin{equation}
		\label{test_u_0_vdz}
		\vert K^*\vert\delta_t \uKs^{m+1}(\uKs^{m+1})_\ominus = -\aupw{\vv^{m+1}}{u^{m+1}}{\bu^*} -C_u\bupw{\Pi_0\mu_u^{m+1}}{M_u(u^{m+1}) }{\bu^*}.
	\end{equation}
	
	Now, since $\uL^{m+1}\ge \uKs^{m+1}$ and the local incompressibility of $\vv^{m+1}$ given in \eqref{local_incompressibility} holds, we can assure that
	$$
	\aupw{\vv^{m+1}}{u^{m+1}}{\bu^*} \le 0.
	$$
	Also, since
	$$M_u^\uparrow(\uL^{m+1}) \ge M_u^\uparrow(\uKs^{m+1})\quad
	\hbox{and} \quad
	M_u^\downarrow(\uL^{m+1}) \le M_u^\downarrow(\uKs^{m+1}),
	$$
	using that the positive part is an increasing function, we obtain
	$$
	\bupw{\Pi_0\mu_u^{m+1}}{M_u(u^{m+1})}{\bu^*}\le 0.
	$$
	These inequalities yield $\vert K^*\vert\delta_t \uKs^{m+1}(\uKs^{m+1})_\ominus \ge0$.
	
	Consequently,
	$$
	0
	\le |K^*|(\delta_t \uKs^{m+1})(\uKs^{m+1})_\ominus 
	=
	-\frac{|K^*|}{\Delta t}\left((\uKs^{m+1})_\ominus ^2+\uKs^m(\uKs^{m+1})_\ominus \right)
	\le 0,
	$$
	which implies, since $\uKs^m\ge 0$, that $(\uKs^{m+1})_\ominus =0$. Hence $u^{m+1}\ge0$ in $\Omega$.
	
	Similarly, taking the following $\Pd_0(\T_h)$ test
	function in \eqref{eq:esquema_DG_mod_model_n},
	\begin{align*}
		\overline{n}^*=
		\begin{cases}
			(\nKs^{m+1})_\ominus &\text{in }K^*,\\
			0&\text{out of }K^*,
		\end{cases}
	\end{align*}
	 where $K^*$ is an element of $\T_h$ such that  $\nKs^{m+1}=\min_{K\in\T_h}n_{K}^{m+1}$ we get $n^{m+1}\ge 0$ in $\Omega$.
	
	\
	
	Secondly, we prove that $u^{m+1},n^{m+1}\le 1$ in $\Omega$.
	
	To prove that $u^{m+1}\le 1$, taking the following test function in \eqref{eq:esquema_DG_mod_model_u},
	\begin{align*}
		\overline{u}^*=
		\begin{cases}
			(\uKs^{m+1}-1)_\oplus &\text{in }K^*,\\
			0&\text{out of }K^*,
		\end{cases}
	\end{align*}
	where $K^*$ is an element of $\T_h$ such that  $\uKs^{m+1}=\max_{K\in\T_h}u_{K}^{m+1}$ and using similar arguments than above, we arrive at
	$$
	|K^*|\delta_t \uKs^{m+1}(\uKs^{m+1}-1)_\oplus \le0.
	$$
	Therefore, it is satisfied that
	\begin{align*}
		0&\ge
		|K^*|\delta_t \uKs^{m+1}(\uKs^{m+1}-1)_\oplus 
		=\frac{|K^*|}{\Delta t}\left((\uKs^{m+1}-1)+(1-\uKs^m)\right)(\uKs^{m+1}-1)_\oplus \\
		&=
		\frac{|K^*|}{\Delta t}\left((\uKs^{m+1}-1)^2_\oplus +(1-\uKs^m)(\uKs^{m+1}-1)_\oplus \right)
		\ge0,
	\end{align*}
	what yields $(\uKs^{m+1}-1)_\oplus =0$ and, therefore, $u^{m+1}\le1$ in $\Omega$.
	
	Finally, taking the test function in \eqref{eq:esquema_DG_mod_model_n} 
		\begin{align*}
		\overline{n}^*=
		\begin{cases}
			(\nKs^{m+1}-1)_\oplus &\text{in }K^*,\\
			0&\text{out of }K^*,
		\end{cases}
	\end{align*}
	 where $K^*$ is an element of $\T_h$ such that  $\nKs^{m+1}=\max_{K\in\T_h}n_{K}^{m+1}$ we obtain, similarly, that $n^{m+1}\le 1$ in $\Omega$.
\end{proof}

The following result is a direct consequence of the previous Theorem \ref{thm:discrete_maximum_principle} and the equality \eqref{Pi_1} of the regularization $\Pi_1^h$.
\begin{corollary}
	\label{cor:principio_del_maximo_w_DG}
	It satisfies $\up^{m+1}\in[0,1]$ in $\overline\Omega$ provided $u^{m+1}\in[0,1]$ in $\overline\Omega$.
\end{corollary}

The following Lemma is a technical result that we are going to use when computing the discrete energy law. 
\begin{lemma}
	\label{lemma:discrete_energy}
	The following expressions hold
	\begin{align}
		\label{upwind_stabilization_energy_u}
		\aupw{\vv^ {m+1}}{u^{m+1}}{\Pi_0\mu_u^{m+1}}
		+\ch{u^{m+1}}{&\Pi_0\mu_u^{m+1}}{\vv^{m+1}}=\nonumber
		\\ &=\frac{1}{2}\sum_{e\in\Ehi}\int_e|\vv^{m+1}\cdot\nn_e|\salto{u^{m+1}}\salto{\Pi_0\mu_u^{m+1}}
		\eqqcolon \tau^{m+1}_u,
		\\
		\label{upwind_stabilization_energy_n}
		\aupw{\vv^ {m+1}}{n^{m+1}}{\mu_n^{m+1}}
		+\ch{n^{m+1}}{&\mu_n^{m+1}}{\vv^{m+1}}=\nonumber
		\\ &=\frac{1}{2}\sum_{e\in\Ehi}\int_e|\vv^{m+1}\cdot\nn_e|\salto{n^{m+1}}\salto{\mu_n^{m+1}}
		\eqqcolon \tau^{m+1}_n.
	\end{align}
\end{lemma}
\begin{proof}
	Using \eqref{eq:esquema_DG_mod_model_energia_p}, we have that
	$$
	\ch{u^{m+1}}{\Pi_0\mu_u^{m+1}}{\vv^{m+1}} = \sum_{e\in\Ehi}\int_e(\vv^{m+1}\cdot\nn_e)\media{u^{m+1}}\salto{\Pi_0\mu_u^{m+1}}.
	$$
	and, therefore,
	$$
	\aupw{\vv^ {m+1}}{u^{m+1}}{\Pi_0\mu_u^{m+1}}=-\ch{u^{m+1}}{\Pi_0\mu_u^{m+1}}{\vv^{m+1}}+\frac{1}{2}\sum_{e\in\Ehi}\int_e|\vv^{m+1}\cdot\nn_e|\salto{u^{m+1}}\salto{\Pi_0\mu_u^{m+1}},
	$$
	which completes the proof for \eqref{upwind_stabilization_energy_u}. The proof of \eqref{upwind_stabilization_energy_n} is analogous.
\end{proof}

\begin{theorem}[Unconditional energy law]
	\label{thm:energia_esquema}
	If $u^m, n^m\in[0,1]$ in $\overline\Omega$, any solution
	of the scheme \eqref{esquema_DG_mod_model} satisfies the following \textbf{discrete energy law}, independently of the time step $\Delta t$ and mesh size $h$: \begin{align}
		\label{ley_energia_discreta}
		\delta_t E(\up^{m+1},n^{m+1})&+C_u\bupw{\mup_u^{m+1}}{M_u(u^{m+1})}{\Pi_0\mu_u^{m+1}}+C_n\bupw{\mu_n^{m+1}}{M_n(n^{m+1})}{\mu_n^{m+1}}\notag\\&
		+\delta P_0\int_\Omega P(u^{m+1},n^{m+1}) (\mu_n^{m+1}-\mup_u^{m+1})_\oplus ^2+\frac{1}{K}\int_\Omega|\vv^{m+1}|^2\nonumber\\&
		+\frac{\Delta t \, \varepsilon^2}{2}\int_\Omega|\delta_t\nabla \up^{m+1}|^2
		+\frac{\Delta t}{2\delta}\int_\Omega \vert\delta_t n^{m+1}\vert^2\le -(\tau^{m+1}_u+\tau^{m+1}_n),
	\end{align}
	where the energy $E(u,n)$ is defined in \eqref{energia}, $\tau_u^{m+1}$ and $\tau_n^{m+1}$ are defined in \eqref{upwind_stabilization_energy_u} and \eqref{upwind_stabilization_energy_n}, respectively, and $$\bupw{\mup_u^{m+1}}{M_u(u^{m+1})}{\Pi_0\mu_u^{m+1}}, \bupw{\mu_n^{m+1}}{M_n(n^{m+1})}{\mu_n^{m+1}}\ge 0.$$
\end{theorem}
\begin{proof}
	By taking $\bvv=\vv^{m+1}$, $\bp=p^{m+1}$,
	$\overline{u}=\Pi_0\mu_u^{m+1}$,  $\overline{\mu}_u=\delta_t \up^{m+1}$, $\overline{n}=\mu_n^{m+1}$ in \eqref{eq:esquema_DG_mod_model_v}--\eqref{eq:esquema_DG_mod_model_n} and testing \eqref{eq:esquema_DG_mod_model_mun} by $\delta_t n^{m+1}$ we arrive at
	\begin{subequations}
		\label{esquema_DG_mod_model_energia}
		\begin{align}
			\label{eq:esquema_DG_mod_model_energia_v}
			\frac{1}{K}\escalarL{\vv^{m+1}}{\vv^{m+1}}
			&-\escalarL{p^{m+1}}{\nabla\cdot\vv^{m+1}}
			+\ch{u^{m+1}}{\Pi_0\mu_u^{m+1}}{\vv^{m+1}}+\ch{n^{m+1}}{\mu_n^{m+1}}{\vv^{m+1}}=0,
			\\
			\label{eq:esquema_DG_mod_model_energia_p}
			\escalarL{\nabla\cdot\vv^{m+1}}{p^{m+1}}&=0,\\
			\label{eq:esquema_DG_mod_model_energia_u}
			\escalarL{\delta_t u^{m+1}}{\mup_u^{m+1}}&+\aupw{\vv^{m+1}}{u^{m+1}}{\Pi_0\mu_u^{m+1}}+C_u\bupw{\mup_u^{m+1}}{M_u(u^{m+1})}{\Pi_0\mu_u^{m+1}}&\notag\\&=\delta P_0 \escalarL{P(u^{m+1},n^{m+1}) (\mu_n^{m+1}-\mup_u^{m+1})_\oplus }{\mup_u^{m+1}},\\
			\label{eq:esquema_DG_mod_model_energia_muu}
			\escalarL{\mu_u^{m+1}}{\delta_t \up^{m+1}}&=\varepsilon^2 \escalarLd{\nabla \up^{m+1}}{\delta_t\nabla \up^{m+1}}+ \escalarL{f(\up^{m+1},\up^m)}{\delta_t \up^{m+1}}\nonumber\\&\quad-\chi_0\escalarL{ n^{m+1}}{\delta_t \up^{m+1}},\\
			\label{eq:esquema_DG_mod_model_energia_n}
			\escalarL{\delta_t n^{m+1}}{\mu_n^{m+1}}&+\aupw{\vv^{m+1}}{n^{m+1}}{\mu_n^{m+1}}+C_n\bupw{\mu_n^{m+1}}{M_n(n^{m+1})}{\mu_n^{m+1}}&\notag\\&=-\delta P_0 \escalarL{P(u^{m+1},n^{m+1}) (\mu_n^{m+1}-\mup_u^{m+1})_\oplus }{\mu_n^{m+1}},\\
			\label{eq:esquema_DG_mod_model_energia_mun}
			\escalarL{\mu_n^{m+1}}{\delta_t n^{m+1}}&=\frac{1}{\delta}\escalarL{n^{m+1}}{\delta_t n^{m+1}}  -\chi_0 \escalarL{\up^{m}}{\delta_t n^{m+1}}.
		\end{align}
	\end{subequations}
	
	Observe that, by \eqref{eq:esquema_DG_Pi0}--\eqref{eq:esquema_DG_Pih1},
	\begin{align*}
		\escalarL{\delta_t \up^{m+1}}{\mu_u^{m+1}}&=\escalarL{\delta_t u^{m+1}}{\mu_u^{m+1}},\\
		\escalarL{\delta_t u^{m+1}}{\mu_u^{m+1}}&=\escalarL{\delta_t u^{m+1}}{\Pi_0 \mu_u^{m+1}},
	\end{align*}
	hence in particular 
	$$(\delta_t u^{m+1}, \Pi_0 \mu_u^{m+1})=(\delta_t \up^{m+1}, \mu_u^{m+1}).$$

	Then, by adding \eqref{eq:esquema_DG_mod_model_energia_v}--\eqref{eq:esquema_DG_mod_model_energia_mun} and taking into account Lemma~\ref{lemma:discrete_energy}, we obtain
	\begin{align*}
		\frac{1}{K}\int_\Omega|\vv^{m+1}|^2 &+C_u\bupw{\mup_u^{m+1}}{M_u(u^{m+1})}{\Pi_0\mu_u^{m+1}}
		+C_n\bupw{\mu_n^{m+1}}{M_n(n^{m+1})}{\mu_n^{m+1}}
		\\&+\varepsilon^2\escalarLd{\nabla \up^{m+1}}{\delta_t\nabla \up^{m+1}} +\escalarL{f(\up^{m+1},\up^m)}{\delta_t \up^{m+1}}\\
		&+\delta P_0\escalarL{P(u^{m+1},n^{m+1}) (\mu_n^{m+1}-\mup_u^{m+1})_\oplus }{\mu_n^{m+1}-\mup_u^{m+1}}\\&+\frac{1}{\delta}\escalarL{n^{m+1}}{\delta_t n^{m+1}}-\chi_0\escalarL{ n^{m+1}}{\delta_t \up^{m+1}}-\chi_0 \escalarL{\up^{m}}{\delta_t n^{m+1}}= -(\tau^{m+1}_u+\tau^{m+1}_n).
	\end{align*}
	Taking into account that
	\begin{align*}
	\varepsilon^2\escalarLd{\nabla \up^{m+1}}{\delta_t \nabla \up^{m+1}}
	&=\frac{\varepsilon^2}{2}\delta_t\int_\Omega |\nabla \up^{m+1}|^2 
	+\frac{\Delta t \,\varepsilon^2}{2}\int_\Omega|\delta_t\nabla \up^{m+1}|^2,\\
	\frac{1}{\delta}\escalarL{n^{m+1}}{\delta_t n^{m+1}}&=\frac{1}{2\delta}\delta_t\int_\Omega |n^{m+1}|^2 +\frac{\Delta t}{2\delta}\int_\Omega|\delta_t n^{m+1}|^2,\\
	\chi_0\delta_t\int_\Omega u^{m+1}n^{m+1}&=\chi_0\escalarL{ n^{m}}{\delta_t \up^{m+1}}+\chi_0 \escalarL{\up^{m+1}}{\delta_t n^{m+1}},\\
	\int_\Omega P(u^{m+1}, n^{m+1}) (\mu_n^{m+1}-\mup_u^{m+1})_\oplus ^2&=\escalarL{P(u^{m+1},n^{m+1}) (\mu_n^{m+1}-\mup_u^{m+1})_\oplus }{\mu_n^{m+1}-\mup_u^{m+1}},
	\end{align*}
	and by adding and subtracting $\delta_t\int_\Omega F(\up^{m+1})$,
	we get the following equality
	\begin{align*}
		\delta_t E(\up^{m+1},n^{m+1})&+C_u\bupw{\mup_u^{m+1}}{M_u(u^{m+1})}{\mup_u^{m+1}}+C_n\bupw{\mu_n^{m+1}}{M_n(n^{m+1})}{\mu_n^{m+1}}\notag
		\\&+\delta P_0\int_\Omega P(u^{m+1},n^{m+1}) (\mu_n^{m+1}-\mup_u^{m+1})_\oplus ^2\notag+\frac{1}{K}\int_\Omega|\vv^{m+1}|^2\\&+\frac{\Delta t\,\varepsilon^2}{2}\int_\Omega|\delta_t\nabla \up^{m+1}|^2+\frac{\Delta t}{2\delta}\int_\Omega \vert\delta_t n^{m+1}\vert^2\notag\\&=\delta_t\int_\Omega F(\up^{m+1})-\escalarL{f(\up^{m+1},\up^m)}{\delta_t \up^{m+1}} -(\tau^{m+1}_u+\tau^{m+1}_n).
	\end{align*}

	Thus, using Corollary~\ref{cor:principio_del_maximo_w_DG} we know that $\up^{m+1}\in[0,1]$, and by the convex-splitting approximation \eqref{convex-split} (see \cite{eyre_1998_unconditionally,guillen-gonzalez_linear_2013,acosta-soba_CH_2022}), one has that
	$$
	\int_\Omega\delta_t F(\up^{m+1}) - \escalarL{f(\up^{m+1},\up^m)}{\delta_t( \up^{m+1})}\le0,
	$$
	which implies \eqref{ley_energia_discreta}.

	Finally, it is straightforward to see that $$\bupw{\mup_u^{m+1}}{M_u(u^{m+1})}{\Pi_0\mu_u^{m+1}}, \bupw{\mu_n^{m+1}}{M_n(n^{m+1})}{\mu_n^{m+1}}\ge 0,$$ what completes the proof.

\end{proof}

\begin{remark}
	\label{rmk:residuals_energy}
	Notice that the discrete energy law \eqref{ley_energia_discreta} does not guarantee, in general, a nonincreasing behavior of the discrete energy due to the residual term $-(\tau_u^{m+1}+\tau_n^{m+1})$. However, formally, if the solution of the system \eqref{problema:mod_model_gamma-0} is smooth enough, one may expect that the right-hand side of \eqref{ley_energia_discreta} tends to zero as $h\to 0$. In particular, according to \cite[Lemma 1.59]{di_pietro_ern_2012}, if the exact solution $u(t_{m+1})$ belongs to $H^s(\Omega)$, then $\int_e\salto{u^{m+1}}^2\to 0$ with order $O(h^{2s-1})$. In this sense, one may expect that, if $u,\mu_u,n$ are sufficiently smooth, the right-hand side of \eqref{ley_energia_discreta} tends to zero with order $O(h^{2s-1})$ as $h\to 0$. In fact, this has been supported by numerical test shown in Section~\ref{sec:numerical_experiments}, in particular, in Table~\ref{tab:residuals_order} we show that the residuals $\tau_u^{m+1}$ and $\tau_n^{m+1}$ seem to converge to zero with order $O(h^2)$ and $O(h^{5/2})$, respectively. Thus, we can expect our scheme to be \textbf{weakly energy stable} in the sense that the discrete energy is nonincreasing for a small enough mesh size $h$.

	One important observation is that, ensuring a nonincreasing discrete energy, $E(\up^{m+1},n^{m+1})\le E(\up^{m},n^{m})$ for all $m\ge 0$, can be achieved by introducing stabilization terms in te equation of Darcy's flow \eqref{eq:esquema_DG_mod_model_v} following the ideas of \cite{acosta2025property}. However, we have observed computationally that this approach does not work in practice, since it introduces a strong dissipation source that significantly slows down the velocity of the fluid, making it almost vanish. We think this may be due to the fact that the velocity in the case of Darcy's equation is less regular and sensitive to instantaneous variations than the velocity modeled by the Navier-Stokes equations due to the absence of a viscous term and time derivative, respectively.

	In this sense, we have tried to stabilize the scheme, as in our previous work for a Cahn--Hilliard--Navier--Stokes system \cite{acosta2025property}, by adding $\sigma_u(h)\shd{\vv^{m+1}}{u^{m+1}}{\Pi_0\mu_u^{m+1}}{\bv}$ and $\sigma_n(h)\shd{\vv^{m+1}}{n^{m+1}}{\mu_n^{m+1}}{\bv}$ to the left hand side of the equation of Darcy's flow \eqref{eq:esquema_DG_mod_model_v} where
\begin{equation}
\label{def:sh2}
	\shd{\uu}{\phi}{\mu}{\buu}\coloneqq
	\begin{cases}
		-\frac{1}{2}\sum_{e\in\Ehi}\int_e(\buu\cdot\nn_e)\, \sign{\uu\cdot\nn_e}\salto{\phi}\salto{\mu} &\text{if } 	\eta=0,\\
		-\frac{1}{2}\sum_{e\in\Ehi}\int_e(\buu\cdot\nn_e)\, \frac{\uu\cdot\nn_e}{|\uu\cdot\nn_e|+\eta}\salto{\phi}\salto{\mu} &\text{if } \eta>0,
	\end{cases}
\end{equation}
$\eta\ge0$ is a regularization parameter and $\sigma_u(h), \sigma_n(h)>0$ are parameters to control de influence of the stabilization terms. This stabilization leads to a similar discrete energy law to \eqref{ley_energia_discreta} with
\begin{align*}
	\tau_u^{m+1}&=\frac{1}{2}\sum_{e\in\Ehi}\int_e\frac{(1-\sigma_u(h))|\vv^{m+1}\cdot\nn_e|+\eta}{|\vv^{m+1}\cdot\nn_e|+\eta}\,|\vv^{m+1}\cdot\nn_e|\salto{u^{m+1}}\salto{\Pi_0\mu_u^{m+1}},\\
	\tau_n^{m+1}&=\frac{1}{2}\sum_{e\in\Ehi}\int_e\frac{(1-\sigma_n(h))|\vv^{m+1}\cdot\nn_e|+\eta}{|\vv^{m+1}\cdot\nn_e|+\eta}\,|\vv^{m+1}\cdot\nn_e|\salto{n^{m+1}}\salto{\mu_n^{m+1}},
\end{align*}
which become $\tau_u=\tau_n=0$ (that implies $E(\up^{m+1},n^{m+1})\le E(\up^{m},n^{m})$) for $\sigma_u(h)=\sigma_n(h)=1$ and $\eta=0$. However, in practice, this choice leads to convergence issues for the iterative methods used to approximate the solution of the nonlinear system \eqref{esquema_DG_mod_model} (Newton's method, fixed-point methods, etc.) due to the $\sign{\cdot}$ function involved. In addition, although choosing $\sigma_u(h), \sigma_n(h)>0$ and $\eta>0$ help the convergence of the iterative methods, the artificial diffusion introduced slows down the velocity of the fluid and may lead to a loss of accuracy in the approximations.

Furthermore, we have explored other possibilities for the definition of $\shd{\cdot}{\cdot}{\cdot}$ including the use of stabilization terms that help making equation~\eqref{eq:esquema_DG_mod_model_v} more coercive such as
\begin{equation}
	\shd{\uu}{\phi}{\mu}{\buu}\coloneqq
	\begin{cases}
		-\frac{1}{2}\sum_{e\in\Ehi}\int_e(\buu\cdot\nn_e)\, \sign{\uu\cdot\nn_e}\salto{\phi}^2 &\text{if } 	\eta=0,\\
		-\frac{1}{2}\sum_{e\in\Ehi}\int_e(\buu\cdot\nn_e)\, \frac{\uu\cdot\nn_e}{|\uu\cdot\nn_e|+\eta}\salto{\phi}^2 &\text{if } \eta>0,
	\end{cases}
\end{equation}
for which we have obtained, using the epsilon Young inequality, a discrete energy law similar to \eqref{ley_energia_discreta} with
\begin{align*}
	\tau_u^{m+1}=&\frac{1}{2}\sum_{e\in\Ehi}\int_e\frac{\eta}{|\vv^{m+1}\cdot\nn_e|+\eta}\,|\vv^{m+1}\cdot\nn_e|\salto{u^{m+1}}\salto{\Pi_0\mu_u^{m+1}} \\&- \frac{1}{\sigma_u(h)}\frac{1}{2}\sum_{e\in\Ehi}\int_e\frac{|\vv^{m+1}\cdot\nn_e|}{|\vv^{m+1}\cdot\nn_e|+\eta}\,|\vv^{m+1}\cdot\nn_e|\salto{\Pi_0\mu_u^{m+1}}^2,\\
	\tau_n^{m+1}=&\frac{1}{2}\sum_{e\in\Ehi}\int_e\frac{\eta}{|\vv^{m+1}\cdot\nn_e|+\eta}\,|\vv^{m+1}\cdot\nn_e|\salto{n^{m+1}}\salto{\mu_n^{m+1}} \\&- \frac{1}{\sigma_n(h)}\frac{1}{2}\sum_{e\in\Ehi}\int_e\frac{|\vv^{m+1}\cdot\nn_e|}{|\vv^{m+1}\cdot\nn_e|+\eta}\,|\vv^{m+1}\cdot\nn_e|\salto{\mu_n^{m+1}}^2.
\end{align*}
With this choice, although we cannot guarantee a nonincreasing behavior of the discrete energy, the convergence of the iterative methods is improved and, when $\eta=0$, the resulting residual terms are controlled by the factors $\sigma_u(h)^{-1}$ and $\sigma_n(h)^{-1}$. However, as in the previous case, the dissipation introduced by the stabilization term is still too high and the accuracy of the approximations is lost as the velocity of the fluid obtained is close to 0.

Thus, based on the fact that the right-hand side of \eqref{ley_energia_discreta} is expected to tend to zero as $h\to 0$ for smooth enough solutions, we have decided to use the scheme \eqref{esquema_DG_mod_model} without any stabilization term, which is the one that yields the best approximation of the velocity of the fluid and, therefore, the best accuracy in the approximations of \eqref{esquema_DG_mod_model}.
\end{remark}

Now, we focus on the existence of solution to the scheme \eqref{esquema_DG_mod_model}.
For this, we consider the following well-known result.
\begin{theorem}[Leray-Schauder fixed point theorem]
	\label{thm:Leray-Schauder}
	Let $\X$ be a Banach space and let $T\colon\X\longrightarrow\X$ be a continuous and compact operator. If the set $$\{x\in\X\colon x=\alpha \,T(x)\quad\text{for some } 0\le\alpha\le1\}$$ is uniformly bounded (with respect to $\alpha$), then $T$ has  at least one fixed point.
\end{theorem}

\begin{theorem}[Existence]
	\label{thm:existencia_solucion_DG-UPW}
	 If $u^m, n^m\in[0,1]$ in $\overline\Omega$, then there is at least one solution of the scheme
	\eqref{esquema_DG_mod_model}.
\end{theorem}
\begin{proof}
	Given two functions $z_u, z_n\in\Pd_0(\T_h)$ with $0\le z_u, z_n \le 1$, we define the map
	$$T\colon \Wh\coloneqq \Vh\times\Ph\times\Pd_0\times \Pc_1\times\Pd_0\longrightarrow\Wh$$ such that $$T(\widehat{\vv},\widehat{p},\widehat{u},\widehat{\mu}_u,\widehat{n})=(\vv,p,u,\mu_u,n)\in\Wh$$ is the unique solution of the linear (and decoupled, computing first $\mu_n$, next $n$, $u$ and the Darcy mixed formulation $(\vv,p)$, and finally $\mu_u$) scheme:
	
	\begin{subequations}
		\label{esquema_lineal_Leray-Schauder_DG_upw_Eyre}
		\begin{align}
			\label{eq:esquema_lineal_Leray-Schauder_DG_upw_Eyre_v}
			\frac{1}{K}\escalarL{\vv}{\bvv}
			&-\escalarL{p}{\nabla \cdot \bvv}
			+\ch{\widehat u}{\Pi_0\widehat\mu_u}{\bvv}+\ch{\widehat n}{\mu_n}{\bvv}=0,
			&&\forall\bvv\in\Vh,
			\\
			\label{eq:esquema_lineal_Leray-Schauder_DG_upw_Eyre_p}
			\escalarL{\nabla\cdot \vv}{\bp}=&0, &&\forall\bp\in\Ph,\\
			\label{eq:esquema_lineal_Leray-Schauder_DG_upw_Eyre_u}
			\frac{1}{\Delta t}\escalarL{u-z_u}{\overline{u}}
			=&-\aupw{\widehat\vv}{\widehat u}{\bu}-C_u\bupw{\Pi_0\widehat{\mu}_u}{M_u(\widehat{u})}{\overline{u}}\nonumber\\&+\delta P_0 \escalarL{P(\widehat{u},\widehat{n}) ( \mu_n-\Pi_0\widehat \mu_u)_\oplus }{\overline{u}},&&\forall\overline{u}\in \Pd_0(\T_h),\\
			\label{eq:esquema_lineal_Leray-Schauder_DG_upw_Eyre_mu_u}
			\escalarML{\mu_u}{\overline{\mu}_u}=&\varepsilon^2 \escalarL{\nabla \up}{\nabla\overline{\mu}_u}+ \escalarL{f(\up,\Pi_1^h z_u)}{\overline{\mu}_u}-\chi_0\escalarL{n}{\overline{\mu}_u},&&\forall\overline{\mu}_u\in \Pc_1(\T_h),\\
			\label{eq:esquema_lineal_Leray-Schauder_DG_upw_Eyre_n}
			\frac{1}{\Delta t}\escalarL{ n-z_n}{\overline{n}}=&-\aupw{\widehat\vv}{\widehat n}{\bn}-C_n\bupw{\mu_n}{M_n(\widehat n)}{\overline{n}}\nonumber\\&-\delta P_0\escalarL{P(\widehat u, \widehat n) (\mu_n-\Pi_0\widehat\mu_u)_\oplus }{\overline{n}},&&\forall\overline{n}\in \Pd_0(\T_h),
		\end{align}
	\end{subequations}
	where $\mu_n $ is already known as
	\begin{equation}
		\label{eq:esquema_lineal_Leray-Schauder_DG_upw_Eyre_mu_n}
			\mu_n=\frac{1}{\delta}\widehat n -\chi_0 \Pi_0(\Pi_h^1 z_u).
	\end{equation}
	
	It is straightforward to check that, for any given $(\widehat{\vv},\widehat{u},\widehat\mu_u,\widehat{n})\in \Wh$, there is a unique solution $(u,n)\in \Pd_0(\T_h)\times \Pd_0(\T_h)$ of \eqref{eq:esquema_lineal_Leray-Schauder_DG_upw_Eyre_u} and \eqref{eq:esquema_lineal_Leray-Schauder_DG_upw_Eyre_n}. Also, due to the inf-sup condition \eqref{eq:inf-sup_discrete} satisfied by the pair of spaces $\Vh\times\Ph$ there is a unique solution $(\vv,p)\in \Vh\times\Ph$ of \eqref{eq:esquema_lineal_Leray-Schauder_DG_upw_Eyre_v} and \eqref{eq:esquema_lineal_Leray-Schauder_DG_upw_Eyre_p} (see \cite{boffi2013mixed,ern_theory_2010}). Finally, there is a unique solution $\mu_u\in\Pc_1(\T_h)$ such that \eqref{eq:esquema_lineal_Leray-Schauder_DG_upw_Eyre_mu_u} holds. Therefore, the operator $T$ is well defined.

	Now we will prove that  $T$ is under the hypotheses of the Leray-Schauder fixed-point theorem \ref{thm:Leray-Schauder}.
	
	First, we check that $T$ is continuous. Let $\{(\widehat\vv_j,\widehat p_j,\widehat{u}_j,\widehat{\mu}_{u_j},\widehat{n}_j)\}_{j\in\N}\subset\Wh$ be a sequence such that $\lim_{j\to\infty}(\widehat\vv_j,\widehat{p}_j,\widehat{u}_j,\widehat{\mu}_{u_j},\widehat{n}_j)=(\widehat{\vv},\widehat{p},\widehat{u},\widehat{\mu}_u,\widehat{n})$. Taking into account that all norms are equivalent in $\Wh$ since it is a finite-dimensional space and that $\Omega$ is bounded, the convergence can be considered to be uniform. Then, since all the functions appearing in \eqref{esquema_lineal_Leray-Schauder_DG_upw_Eyre} are continuous in $\Omega$, 
	taking limits when $j\to \infty$ in \eqref{esquema_lineal_Leray-Schauder_DG_upw_Eyre} (with $\widehat{\vv}\coloneqq\widehat{\vv}_j$, $\widehat{u}\coloneqq\widehat{u}_j$, $\widehat{\mu}_u\coloneqq\widehat{\mu}_{u_j}$, $\widehat{n}\coloneqq\widehat{n}_j$)
	and using the fact that the solution of \eqref{esquema_lineal_Leray-Schauder_DG_upw_Eyre} is unique, we get that $$\lim_{j\to \infty} T(\widehat\vv_j, \widehat p_j,\widehat u_j,\widehat\mu_j, \widehat n_j)=T(\widehat\vv, \widehat p, \widehat u,\widehat\mu, \widehat n)=T\left(\lim_{j\to \infty}(\widehat\vv_j,\widehat p_j, \widehat u_j,\widehat\mu_j, \widehat n_j)\right),$$ hence $T$ is continuous. Therefore, $T$ is also compact since $\Wh$ has finite dimension.
	
	Finally, let us prove that the set 
	$$
		B=\{(\vv,p,u,\mu_u,n)\in\Wh\colon 
		(\vv,p,u,\mu_u,n)=\alpha\, T(\vv,p,u,\mu_u,n)\text{ for some } 0\le\alpha\le1\}
	$$
	is bounded (independently of $\alpha$). The case $\alpha=0$ is trivial so we will assume that $\alpha\in(0,1]$.
	
	If $(u,\mu_u,n)\in B$, then $(u,\mu_u,n)\in\Pd_0(\T_h)\times \Pc_1(\T_h)\times\Pd_0(\T_h)$ is the solution of
	\begin{subequations}
		\label{esquema_DG_upw_Eyre_Leray-Schauder}
		\begin{align}
			\label{eq:esquema_DG_upw_Eyre_Leray-Schauder_v}
			\frac{1}{K}\escalarL{\vv}{\bvv}
			&-\escalarL{p}{\nabla \cdot \bvv}
			+\alpha\, \ch{u}{\Pi_0\mu_u}{\bvv}
			+\alpha\, \ch{n}{\mu_n}{\bvv}=0,
			&&\forall\bvv\in\Vh,
			\\
			\label{eq:esquema_DG_upw_Eyre_Leray-Schauder_p}
			\escalarL{\nabla\cdot \vv}{\bp}=&0, &&\forall\bp\in\Ph,\\
			\label{eq:esquema_DG_upw_Eyre_Leray-Schauder_u}
			\frac{1}{\Delta t}\escalarL{u-\alpha z_u}{\overline{u}}
			=&-\alpha\aupw{\vv}{u}{\bu}-\alpha C_u\bupw{\Pi_0\mu_u}{M_u(u)}{\overline{u}}\nonumber\\&+\alpha\delta P_0 \escalarL{P(u,n) ( \mu_n-\Pi_0 \mu_u)_\oplus }{\overline{u}},&&\forall\overline{u}\in \Pd_0(\T_h),\\
			\label{eq:esquema_DG_upw_Eyre_Leray-Schauder_mu_u}
			\escalarML{\mu_u}{\overline{\mu}_u}=&\varepsilon^2 \escalarL{\nabla \up}{\nabla\overline{\mu}_u}+ \escalarL{f(\up,\Pi_1^h z_u)}{\overline{\mu}_u}-\chi_0\escalarL{n}{\overline{\mu}_u},&&\forall\overline{\mu}_u\in \Pc_1(\T_h),\\
			\label{eq:esquema_DG_upw_Eyre_Leray-Schauder_n}
			\frac{1}{\Delta t}\escalarL{ n-\alpha z_n}{\overline{n}}=&-\alpha\aupw{\vv}{ n}{\bn}-\alpha C_n\bupw{\mu_n}{M_n(n)}{\overline{n}}\nonumber\\&-\alpha\delta P_0\escalarL{P(u, n) (\mu_n-\Pi_0\mu_u)_\oplus }{\overline{n}},&&\forall\overline{n}\in \Pd_0(\T_h),
		\end{align}
	\end{subequations}
	where
	\begin{equation}
		\label{eq:esquema_DG_upw_Eyre_Leray-Schauder_mu_n}
			\mu_n=\frac{1}{\delta}\widehat n -\chi_0 \Pi_0(\Pi_h^1 z_u).
	\end{equation}
	
	Since $0\le z_u, z_n\le 1$, it can be proved that $0\le u, n\le 1$ using the same arguments as in Theorem~\ref{thm:discrete_maximum_principle}. Hence, $\norma{u}_{L^\infty(\Omega)},\norma{n}_{L^\infty(\Omega)}<+ \infty$.
	
	Now, we will check that $\mu_u$ is bounded. Testing 
	\eqref{eq:esquema_DG_upw_Eyre_Leray-Schauder_mu_u} with $\overline{\mu}_u=\mu_u$ we obtain that
	$$
		\normaL{\mu_u}^2
\le \varepsilon^2\norma{\up}_{H^1(\Omega)}\norma{\mu_u}_{H^1(\Omega)}
+\normaL{f(\up,\Pi_1^hz_u)}\normaL{\mu_u}
+\chi_0 \normaL{n}\normaL{\mu_u}.
$$
	Since all the norms are equivalent in the finite-dimensional space $\Pc_1(\T_h)$,  there is $K_1\ge 0$ such that
	$$\normaL{\mu_u}\le \varepsilon^2K_1\norma{\up}_{H^1(\Omega)}+\normaL{f(\up,\Pi_1^hz_u)}+\norma{n}_{L^2(\Omega)}.$$
	Consequently, since $\normaL{f(\up,\Pi_1^hz_u)}$ is bounded due to $0\le \up,\Pi_1^h z_u\le 1$ and $\norma{\up}_{L^1(\Omega)}$ and $\norma{n}_{L^1(\Omega)}$ are also bounded, we conclude that $\normaL{\mu_u}$ is bounded.

	Then, by testing \eqref{eq:esquema_DG_upw_Eyre_Leray-Schauder_v} and \eqref{eq:esquema_DG_upw_Eyre_Leray-Schauder_p} with $\bvv=\vv$ and $\bp=p$, and using the bounds for $u,n,\mu_u$, we get that $\normaL{\vv}^2\le K \norma{\vv}_{L^\infty(\Omega)}$, where $K$ is a constant independent of $\alpha$. Hence, using the equivalence of norms, $\vv$ is also bounded. Consequently, using the inf-sup discrete condition \eqref{eq:inf-sup_discrete} for Darcy, we get that $p$ is also bounded.

	Since $\Wh$ is a finite-dimensional space where all the norms are equivalent, we have proved that $B$ is bounded.
	
	Thus, using the Leray-Schauder fixed point theorem \ref{thm:Leray-Schauder}, there is a solution $(\vv,p,u,\mu_u,n)$ of the scheme
	\eqref{esquema_DG_mod_model}.
\end{proof}

In the case with the non-regularized stabilization ($\eta=0$), the existence of a solution of the scheme \eqref{esquema_DG_mod_model} is not straightforward due to the discontinuity of the $\sign{\cdot}$ function.

\section{Numerical experiments}
\label{sec:numerical_experiments}
Now, we will present several numerical experiments that match the results presented in the previous section. In particular, we are going to repeat the tests in \cite[Section 4.2]{acosta2023structure}, where we show the irregular growth of a tumor due to the irregular distribution of the nutrients over the domain, but considering the effects of the Darcy flow. Therefore, we will keep the same domain $\Omega=[-10,10]^2$ and define the same initial conditions (shown in Figure \ref{fig:test-2_initial_cond})
\begin{align*}
	u_0 &= \frac{1}{2}\left[\tanh\left(\frac{1.75 - \sqrt{x^2 + y^2}}{\sqrt{2}\varepsilon}\right) + 1\right],\\
	n_0 &= \frac{1}{2}(1-u_0)
	+ \frac{1}{4}\left[\tanh\left(\frac{1 - \sqrt{(x - 2.45)^2 + (y - 1.45)^2}}{\sqrt{2}\varepsilon}\right)\right.\\
	&\quad
	\left.+ \tanh\left(\frac{1.75 - \sqrt{(x + 3.75)^2 + (y - 1)^2}}{\sqrt{2}\varepsilon}\right)
	+ \tanh\left(\frac{2.5 - \sqrt{x^2 + (y + 5)^2}}{\sqrt{2}\varepsilon}\right) + 3\right],
\end{align*}
as in \cite{acosta2023structure}. These initial conditions do not intend to represent a realistic tumor growth scenario, but they are used to illustrate the behavior of the numerical scheme. In fact, one can regard this situation as a traversal cut of an idealized spherical tumor surrounded by three blood vessels full of nutrients. 

\begin{figure}[h]
	\centering
	\begin{tabular}{cc}
		\hspace*{-1.1cm}$\boldsymbol{u_0}$ & \hspace*{-1.1cm}$\boldsymbol{n_0}$\\
		\includegraphics[scale=0.24]{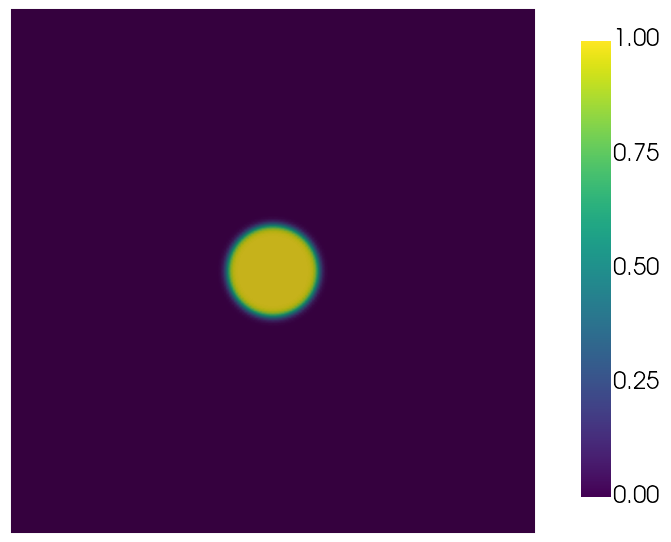} &
		\includegraphics[scale=0.24]{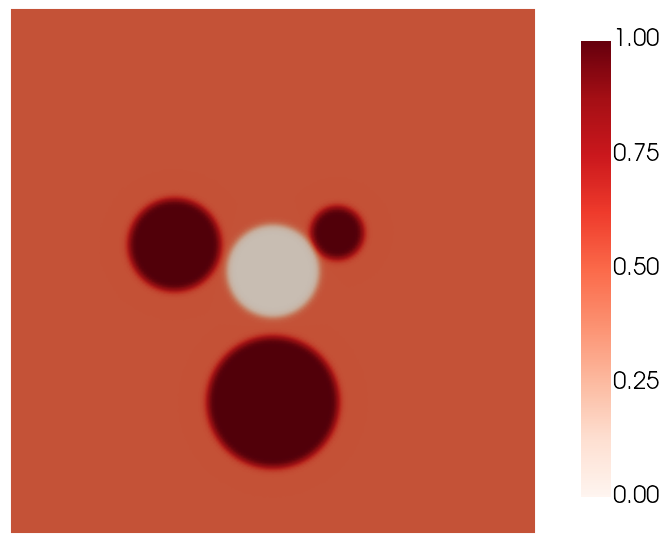}
	\end{tabular}
	\caption{Initial conditions ($u_0$ left, $n_0$ right).}
	\label{fig:test-2_initial_cond}
\end{figure}

\begin{figure}[h]
	\centering
	\begin{tabular}{cc}
		\includegraphics[scale=0.4]{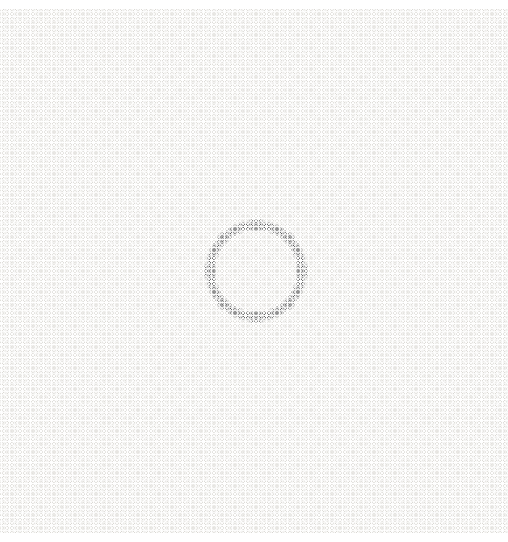}&
		\begin{tikzpicture}[baseline=-3.8cm, spy using outlines={
            circle,                   
            lens={width=2.5cm, height=2.5cm}, 
            magnification=2,             
            connect spies,              
            every spy on/.style={draw=blue, ultra thick},
            every spy in/.style={draw=blue, ultra thick}  
        }]
    		\node (imagen) {\includegraphics[scale=0.4]{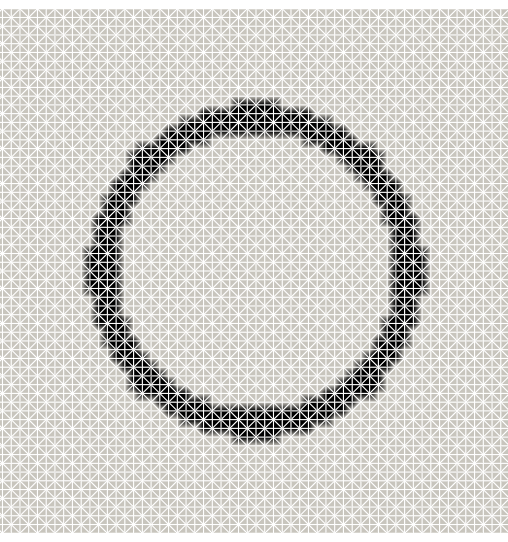}};
        	\spy[blue, ultra thick] on (1.5,-1.5) in node [minimum width=2cm, minimum height=2cm] at (2.5, -2.5);
		\end{tikzpicture}
	\end{tabular}
	\caption{Diffuse interface profile of $u_0$. Full view (left) and zoom (right).}
	\label{fig:test-2_initial_cond_interface}
\end{figure}

The domain $\Omega$ has been discretized using a uniform triangular mesh satisfying Hypothesis~\ref{hyp:mesh} with $h\approx 0.14$. This mesh allows to capture the details of the diffuse interface allocating 360 degrees of freedom to represent the interface (understood as the values of the phase field variable $u$ which range in the interval $(0.1,0.9)$) of the initial condition $u_0$. The diffuse interface profile of the initial condition $u_0$ and the mesh used are shown in Figure~\ref{fig:test-2_initial_cond_interface}.

Regarding the discrete ``Darcy inf-sup'' pair, we will consider $\Vh=BDM_1(\T_h)$ and $\Ph=\Pd_{0}(\T_h)$.

Moreover, to approximate the solution of \eqref{esquema_DG_mod_model}, we will use Newton's method with relative and absolute tolerances of $10^{-6}$. Moreover, for each of the linear iterations within Newton's method, we use the MUMPS LU solver \cite{mumps2001} without preconditioner. All the tests shown have been computed on a computer with a 64 cores Intel(R) Xeon(R) w7-3565X CPU. The resulting computational details including the total number of degrees of freedom (DoFs), average number of Newton iterations and the average time per step are summarized in Table~\ref{tab:computational_info}.

\begin{table}
	\centering
	\begin{tabular}{|c||c|c||c|c|}
	\hline
		\multirow{2}{*}{Total DoFs} & \multicolumn{2}{|c||}{Average Newton iterations} & \multicolumn{2}{|c|}{Average time per step}\\
	\cline{2-5}
		& Range across tests & Mean across tests & Range across tests & Mean across tests \\
	\hline
		641602 & (1.88, 3.18) & 2.34 & (8.37 s, 19.63 s) & 13.33 s\\
	\hline
	\end{tabular}
	\caption{Computational details for the tests shown in Figures~\ref{fig:test-2_reference_symmetric}--\ref{fig:test-2_chi-1_u}.}
	\label{tab:computational_info}
\end{table}

These results have been computed using the Python interface of the library \texttt{FEniCSx}, \cite{AlnaesEtal2014, ScroggsEtal2022,BasixJoss}, and the figures have been plotted using \texttt{PyVista}, \cite{sullivan2019pyvista}. The code to reproduce these results is available at \url{https://github.com/danielacos/Papers-src}.

Notice that, $\up^{m}$ and $\Pi_1^hn^{m}$ are the approximations of the phase field variable $u$ and the nutrients variable $n$, respectively, which have been plotted in the following figures. In addition, the vector $\vv^{m}$ has been plotted as a streamline plot in order to show the behavior of the Darcy flow, where the magnitude of the vector is represented by the color, with darker arrows indicating higher magnitude.

In particular, we represent the behavior of the solution of the model under different set of parameters and we study the effects of the Darcy flow in the evolution of the tumor in scenarios in which this effect is increasingly important as we take $K=0.1, 1, 10$, see Figures~\ref{fig:test-2_reference_symmetric}--\ref{fig:test-2_chi-1_u}. We summarize the choice of the parameters for the reference tests \ref{fig:test-2_reference_symmetric}--\ref{fig:test-2_reference} in Table~\ref{tab:test-2_reference_parameters}. In the rest of the text we will only indicate the value of the parameters that differ from the reference test. For the sake of brevity, we only show the nutrients variable for the reference test.

\begin{table}
	\centering
	\begin{tabular}{|c|c|c|c|c|c|c|c|c|}
		\hline
		$K$ & $C_u$ & $C_n$ & $P_0$ & $\chi_0$ & $\delta$ & $\varepsilon$ & $\Delta t$ & $h$ \\
		\hline
		0.1, 1, 10 & 2.8 & $2.8\times 10^{-4}$ & 0.5 & 0.1 & 0.01 & 0.1 & 0.05 & 0.14 \\
		\hline
	\end{tabular}
	\caption{Parameters for the reference tests \ref{fig:test-2_reference_symmetric}--\ref{fig:test-2_reference}.}
	\label{tab:test-2_reference_parameters}
\end{table}

In fact, as in \cite{acosta2023structure},we have considered two different types of mobility and proliferation functions, based on different choices of the function $h_{p,q}$ as illustrated in Figure~\ref{fig:mobility}. On the one hand, the typical symmetric functions
\begin{equation}
	\label{symmetric_functions}
	M_u(v)=M_n(v)=h_{1,1}(v),\quad P(u,n)=h_{1,1}(u) n_\oplus,
\end{equation}
that give the same importance to the proliferation and dissemination of the tumor at an early and advanced stages, have been used (see Figure~\ref{fig:test-2_reference_symmetric}). However, on the other hand, we have considered the following non-symmetric choice of the mobility and proliferation functions
\begin{equation}
	\label{nonsymmetric_functions}
	M_u(v)=M_n(v)=h_{5,1}(v),\quad P(u,n)=h_{1,3}(u) n_\oplus,
\end{equation}
which emphasize the growth of the tumor in a non-saturated state and its dissemination in a saturated state. The associated results are plotted in Figure~\ref{fig:test-2_reference}.

In fact, aside from the more local interaction between the tumor and nutrients for the nonsymmetric choice of the mobility and proliferation function with respect to the symmetric one, we can observe how the Darcy flow has a stronger effect in the nonsymmetric case. In this sense, we can observe that for $K=0.1$ (first and fourth rows of Figures~\ref{fig:test-2_reference_symmetric}, \ref{fig:test-2_reference} and \ref{fig:test-2_P0-0.001_u} --\ref{fig:test-2_chi-1_u}), the tumor has the same behavior as in \cite{acosta2023structure}. However, as long as we increase the value of $K$, we observe a faster phase separation due to the Darcy flow (as reported in previous works such as \cite{brunk2026review}), particularly in the case of the nonsymmetric choice. Regarding the symmetric mobility and proliferation function, the differences between $K=0.1$ and $K=1$ are barely visible while they are more significant when $K=10$.

In all of these cases the pointwise bounds of $u$ and $n$ are preserved and the energy is nonincreasing as shown in Figure~\ref{fig:test-2_min_max_energy} for the reference test. In particular, we can observe that the pointwise bounds for $n$ and the upper bound for $u$ are preserved up to machine precision, while the lower bound for $u$ is preserved up to a small negative value of order $10^{-5}$, in the case of the symmetric mobility and proliferation function, while in the case of the nonsymmetric choice, the lower bound for $u$ is preserved up to a small negative value of order $10^{-10}$. In addition, we can observe that the energy decreases faster as we increase the value of $K$ for the nonsymmetric choice, which is consistent with the fact that the Darcy flow accelerates significantly the phase separation process in this case.

Moreover, we have carried out a numerical study of the order of convergence of the scheme \eqref{esquema_DG_mod_model} in space and time for the reference test (with $K=1$ nonsymmetric mobility and proliferation functions), which is reported in Tables~\ref{tab:space_order} and~\ref{tab:time_order}. As a reference solution to estimate the order of convergence in space we have used the approximation obtained with $h\approx5.656\times 10^{-2}$ and $\Delta t=10^{-5}$. On the other hand, for the time convergence, we have used the approximation obtained with $\Delta t=1.25\times 10^{-4}$ and $h\approx1.414\times 10^{-1}$.  In particular, regarding Table~\ref{tab:space_order}, we have observed that the order of convergence in space with $L^2$-norm is approximately $O(h)$ for all the variables except $\up$ for which we have obtained order $O(h^2)$. Regarding the order of convergence in time, we have observed that it is approximately $O(\Delta t)$ for all the variables.

\begin{table}[h]
	\centering
	\begin{tabular}{|c|cc|cc|cc|}
        \hline
         $20\sqrt{2}/h$ & $\|u_\text{ref}-u^m\|_{L^2(\Omega)}$ & Rate & $\|\up_\text{ref}-\up^m\|_{L^2(\Omega)}$ & Rate & $\|n_\text{ref}-n^m\|_{L^2(\Omega)}$ & Rate \\
        \hline
        100 & $2.33 \times 10^{-1}$ & -- & $1.70 \times 10^{-1}$ & -- & $2.40 \times 10^{-1}$ & -- \\
        200 & $1.24 \times 10^{-1}$ & 0.91 & $4.79 \times 10^{-2}$ & 1.83 & $1.19 \times 10^{-1}$ & 1.01 \\ 
        300 & $8.99 \times 10^{-2}$ & 0.80 & $1.99 \times 10^{-2}$ & 2.16 & $8.26 \times 10^{-2}$ & 0.90 \\ 
        400 & $6.71 \times 10^{-2}$ & 1.02 & $9.41 \times 10^{-3}$ & 2.61 & $5.94 \times 10^{-2}$ & 1.15 \\
        \hline
	\end{tabular}

	\vspace*{0.3cm}

	\begin{tabular}{|c|cc|cc|}
		\hline
		 $20\sqrt{2}/h$ & $\|\vv_\text{ref}-\vv^m\|_{L^2(\Omega)^2}$ & Rate & $\|p_\text{ref}-p^m\|_{L^2(\Omega)}$ & Rate \\
		\hline
		100 & $1.44 \times 10^{-2}$ & -- & $1.59 \times 10^{1}$ & -- \\
        200 & $1.27 \times 10^{-2}$ & 0.18 & $7.87 \times 10^{0}$ & 1.01 \\ 
        300 & $8.39 \times 10^{-3}$ & 1.03 & $5.47 \times 10^{0}$ & 0.90 \\ 
        400 & $5.63 \times 10^{-3}$ & 1.38 & $3.93 \times 10^{0}$ & 1.15 \\
        \hline
    \end{tabular}
	\caption{Estimated order of convergence in space with the $L^2(\Omega)$-norm for the variables in the reference test ($K=1$ and symmetric mobility and proliferation functions) at $T=10^{-3}$ with $\Delta t=10^{-5}$.}
	\label{tab:space_order}
\end{table}

\begin{table}[H]
	\centering
	\begin{tabular}{|c|cc|cc|cc|}
        \hline
        $\Delta t$ & $\|u_\text{ref}-u^m\|_{L^2(Q_T)}$ & Rate & $\|\up_\text{ref}-\up^m\|_{L^2(Q_T)}$ & Rate & $\|n_\text{ref}-n^m\|_{L^2(Q_T)}$ & Rate \\
        \hline
        $2.00 \times 10^{-3}$ & $1.06 \times 10^{-3}$ & -- & $4.73 \times 10^{-4}$ & -- & $6.76 \times 10^{-6}$ & -- \\
        $1.00 \times 10^{-3}$ & $5.66 \times 10^{-4}$ & 0.91 & $2.38 \times 10^{-4}$ & 0.99 & $3.18 \times 10^{-6}$ & 1.09 \\ 
        $5.00 \times 10^{-4}$ & $2.65 \times 10^{-4}$ & 1.10 & $1.07 \times 10^{-4}$ & 1.15 & $1.38 \times 10^{-6}$ & 1.20 \\ 
        $2.50 \times 10^{-4}$ & $9.30 \times 10^{-5}$ & 1.51 & $3.68 \times 10^{-5}$ & 1.54 & $4.68 \times 10^{-7}$ & 1.56 \\
        \hline 
    \end{tabular}

	\vspace*{0.3cm}

	\begin{tabular}{|c|cc|cc|}
        \hline
        $\Delta t$ & $\|\vv_\text{ref}-\vv^m\|_{L^2(Q_T)}$ & Rate & $\|p_\text{ref}-p^m\|_{L^2(Q_T)}$ & Rate \\
        \hline
        $2.00 \times 10^{-3}$ & $5.35 \times 10^{-4}$ & -- & $4.98 \times 10^{-4}$ & -- \\
        $1.00 \times 10^{-3}$ & $3.20 \times 10^{-4}$ & 0.74 & $2.29 \times 10^{-4}$ & 1.12 \\ 
        $5.00 \times 10^{-4}$ & $1.63 \times 10^{-4}$ & 0.98 & $9.79 \times 10^{-5}$ & 1.23 \\ 
        $2.50 \times 10^{-4}$ & $6.03 \times 10^{-5}$ & 1.43 & $3.27 \times 10^{-5}$ & 1.58 \\
        \hline 
    \end{tabular}
	\caption{Estimated order of convergence in time with the $L^2(Q_T)$-norm ($Q_T = \Omega\times (0,T)$) for the variables in the reference test ($K=1$ and symmetric mobility and proliferation functions) with $T=10^{-2}$ and $20\sqrt{2}/h=200$.}
	\label{tab:time_order}
\end{table}

Also, for this test, we have study how these residuals decrease when refining the mesh in Table~\ref{tab:residuals_order}, obtaining estimated orders $O(h^2)$ and $O(h^{5/2})$ of convergence to $0$, for $\tau_u$ and $\tau_n$ respectively, as expected in Remark~\ref{rmk:residuals_energy}.
Moreover, we have included the range of the residuals across the different tests in Table~\ref{tab:residuals_energy}, which show that, $\tau_u+\tau_n$ is always nonnegative in every time step, which means that the residual acts as a dissipative term in \eqref{ley_energia_discreta}.

\begin{table}[h]
	\centering
	\begin{tabular}{|c|cc|cc|}
		\hline
		$20\sqrt{2}/h$ & $\|\tau_u^m\|_{L^2(0,T)}$ & Rate & $\|\tau_n^m\|_{L^2(0,T)}$ & Rate \\
		\hline
		300 & $3.4\times 10^{-5}$ & -- & $2.75\times 10^{-3}$ & -- \\
		350 & $2.61\times 10^{-5}$ & 1.71 & $1.86\times 10^{-3}$ & 2.55 \\
		400 & $2.11\times 10^{-5}$ & 1.59 & 
		$1.35\times 10^{-3}$ & 2.38 \\
		450 & $1.59\times 10^{-5}$ & 2.41 & $9.69\times 10^{-4}$ & 2.82 \\
		500 & $1.25\times 10^{-5}$ & 2.26 & $7.33\times 10^{-4}$ & 2.65 \\
		\hline
	\end{tabular}
	\caption{Estimated order of convergence in space with the $L^2(0,T)$-norm for the residuals $\tau_u^m$ and $\tau_n^m$ in \eqref{ley_energia_discreta} in the reference test ($K=1$ and symmetric mobility and proliferation functions) with $T=10^{-3}$ and $\Delta t=10^{-5}$.}
	\label{tab:residuals_order}
\end{table}

\begin{table}
	\centering
    \begin{tabular}{cc|cc}
        & & \multicolumn{2}{c}{Range of $\tau_u+\tau_n$} \\
		\cline{3-4}
		& & Symmetric & Nonsymmetric \\
		\hline
        \multirow{3}{*}{Test~\ref{fig:test-2_reference_symmetric}--\ref{fig:test-2_reference}} & $K=0.1$ & $(0, 3.50 \times 10^{-3})$ & $(0, 2.67 \times 10^{-2})$ \\
        & $K=1$ & $(0, 3.49 \times 10^{-2})$ & $(0, 2.98 \times 10^{-1})$ \\
        & $K=10$ & $(0, 3.45 \times 10^{-1})$ & $(0, 1.07 \times 10^{0})$ \\
        \hline
        \multirow{3}{*}{Test~\ref{fig:test-2_P0-0.001_u}} & $K=0.1$ & $(0, 3.50 \times 10^{-3})$ & $(0, 2.69 \times 10^{-2})$ \\
        & $K=1$ & $(0, 3.49 \times 10^{-2})$ & $(0, 1.81 \times 10^{-1})$ \\
        & $K=10$ & $(0, 3.47 \times 10^{-1})$ & $(0, 2.45 \times 10^{0})$ \\
        \hline
        \multirow{3}{*}{Test~\ref{fig:test-2_P0-0.05_u}} & $K=0.1$ & $(0, 3.50 \times 10^{-3})$ & $(0, 1.60 \times 10^{-2})$ \\
        & $K=1$ & $(0, 3.49 \times 10^{-2})$ & $(0, 1.28 \times 10^{-1})$ \\
        & $K=10$ & $(0, 3.47 \times 10^{-1})$ & $(0, 2.27 \times 10^{0})$ \\
        \hline
        \multirow{3}{*}{Test~\ref{fig:test-2_P0-2_u}} & $K=0.1$ & $(0, 3.56 \times 10^{-3})$ & $(0, 8.51 \times 10^{-2})$ \\
        & $K=1$ & $(0, 3.56 \times 10^{-2})$ & $(0, 8.51 \times 10^{-1})$ \\
        & $K=10$ & $(0, 3.58 \times 10^{-1})$ & $(0, 4.89 \times 10^{0})$ \\
        \hline
        \multirow{3}{*}{Test~\ref{fig:test-2_chi-0.01_u}} & $K=0.1$ & $(0, 9.01 \times 10^{-4})$ & $(0, 2.37 \times 10^{-2})$ \\
        & $K=1$ & $(0, 8.85 \times 10^{-3})$ & $(0, 2.34 \times 10^{-1})$ \\
        & $K=10$ & $(0, 7.78 \times 10^{-2})$ & $(0, 1.05 \times 10^{0})$ \\
        \hline
        \multirow{3}{*}{Test~\ref{fig:test-2_chi-0.5_u}} & $K=0.1$ & $(0, 1.55 \times 10^{-2})$ & $(0, 4.09 \times 10^{-2})$ \\
        & $K=1$ & $(0, 1.57 \times 10^{-1})$ & $(0, 3.70 \times 10^{-1})$ \\
        & $K=10$ & $(0, 1.88 \times 10^{0})$ & $(0, 6.47 \times 10^{0})$ \\
        \hline
        \multirow{3}{*}{Test~\ref{fig:test-2_chi-1_u}} & $K=0.1$ & $(0, 3.88 \times 10^{-2})$ & $(0, 4.30 \times 10^{-2})$ \\
        & $K=1$ & $(0, 3.92 \times 10^{-1})$ & $(0, 4.37 \times 10^{-1})$ \\
        & $K=10$ & $(0, 5.07 \times 10^{0})$ & $(0, 1.62 \times 10^{1})$ \\
    \end{tabular}
	\caption{Range of the residual $\tau_u+\tau_n$ in \eqref{ley_energia_discreta} for the tests shown in Figures~\ref{fig:test-2_reference_symmetric}--\ref{fig:test-2_chi-1_u}.}
	\label{tab:residuals_energy}
\end{table}

In particular, in the tests represented in Figures~\ref{fig:test-2_P0-0.001_u}--\ref{fig:test-2_P0-2_u} we have vary the proliferation rate $P_0$ in \{0.001, 0.05, 2\} with respect to the reference test, for which $P_0=0.5$. In these tests, we can observe that the proliferation rate has a strong effect in the evolution of the tumor, as expected. In particular, for $P_0=0.001$ the tumor does not grow significantly, for values $P_0=0.05, 0.5$ the tumor grows until it disseminates and form clusters in different parts of the domain, while for $P_0=2$ it grows very fast and, in the case of the nonsymmetric choice of the mobility and proliferation functions, it remains more concentrated in the middle of the domain. On the other hand, in the tests represented in Figures~\ref{fig:test-2_chi-0.01_u}--\ref{fig:test-2_chi-1_u} we have vary the chemotaxis parameter $\chi_0$ in \{0.01, 0.5, 1\} with respect to the reference test, for which $\chi_0=0.1$. In these tests, we can observe how for small values of $\chi$ such as $\chi_0=0.01, 0.1$ the proliferation effect dominates the tumor growth leading to a more uniform distribution around the middle of the domain at the beginning of the simulation, while for values $\chi_0=0.5, 1$ the tumor moves faster towards the nutrients forming different clusters around these regions before it starts growing.

\begin{figure}
	\vspace*{-0.5cm}
	\centering
	\begin{tabular}{ccccc}
		& & \hspace*{-1cm}$t=10.0$ & \hspace*{-1cm}$t=20.0$ & \hspace*{-1cm}$t=50.0$ \\
		\multirow{3}{*}{\vspace*{-6.2cm}\rotatebox[origin=c]{90}{$\boldsymbol{u}$}}&\rotatebox[origin=c]{90}{$K=0.1$} &
		\raisebox{-0.47\height}{\includegraphics[scale=0.204]{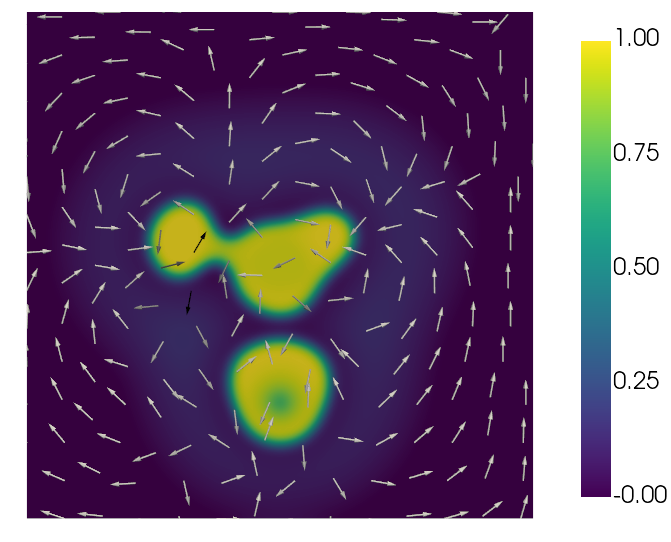}} &
		\raisebox{-0.47\height}{\includegraphics[scale=0.204]{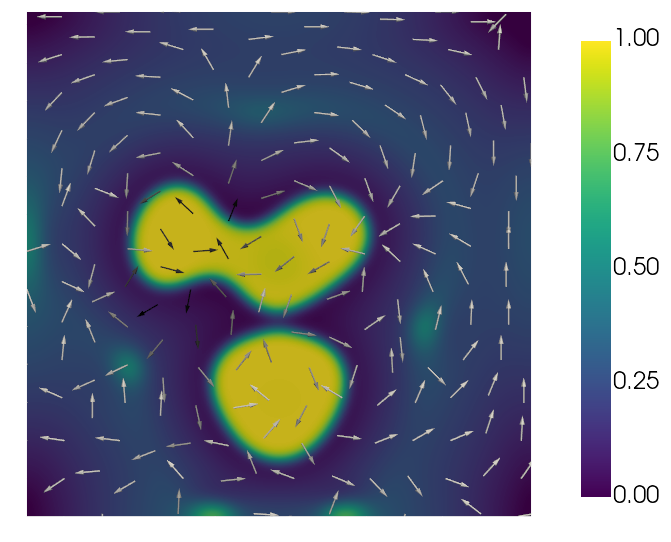}} &
		\raisebox{-0.47\height}{\includegraphics[scale=0.204]{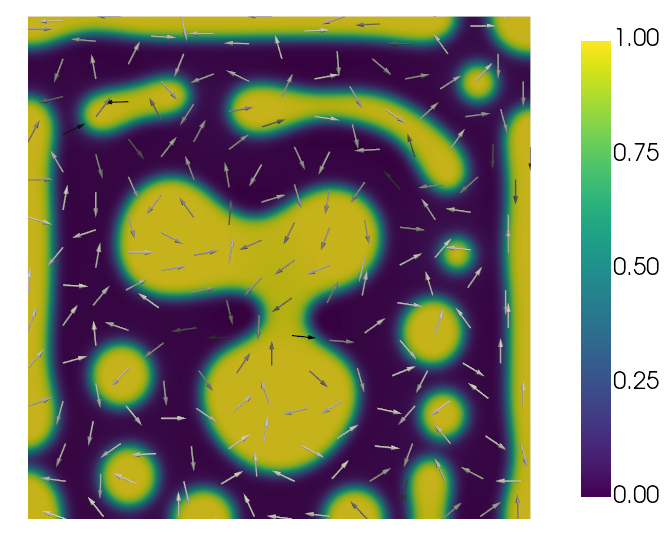}} \\
		&\rotatebox[origin=c]{90}{$K=1$} &
		\raisebox{-0.47\height}{\includegraphics[scale=0.204]{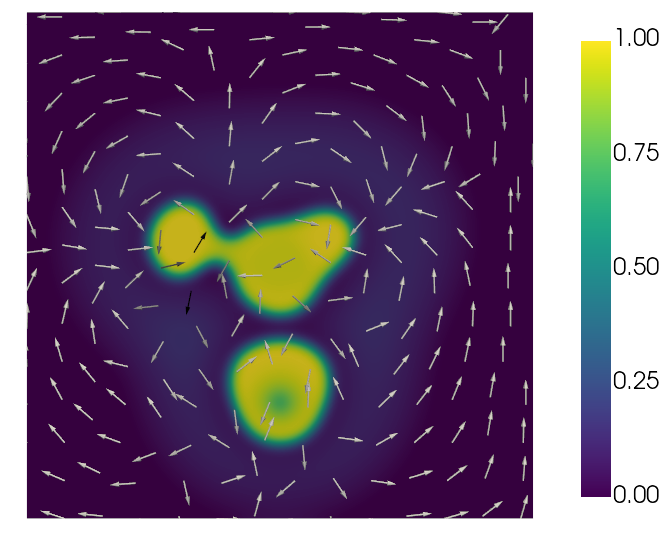}} &
		\raisebox{-0.47\height}{\includegraphics[scale=0.204]{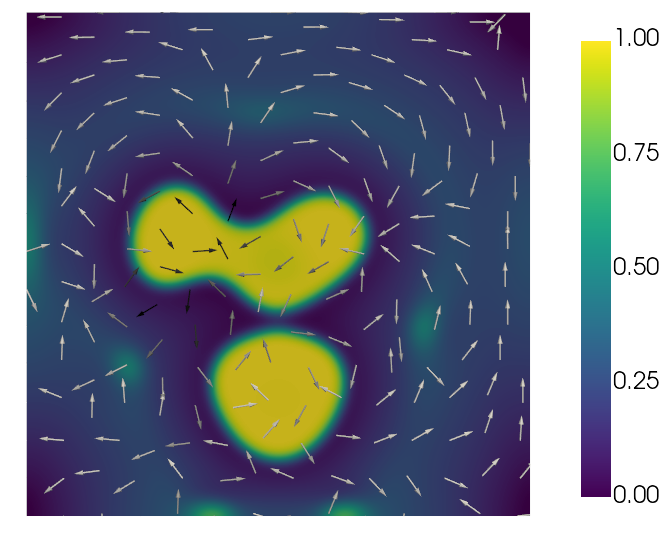}} &
		\raisebox{-0.47\height}{\includegraphics[scale=0.204]{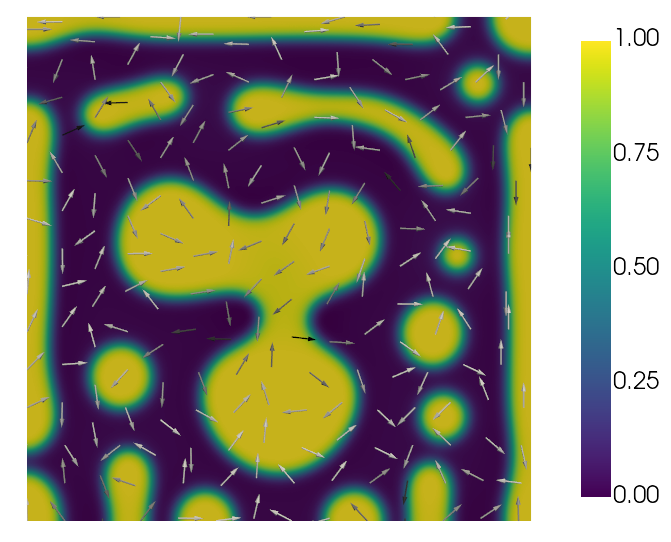}}\\
		&\rotatebox[origin=c]{90}{$K=10$} &
		\raisebox{-0.47\height}{\includegraphics[scale=0.204]{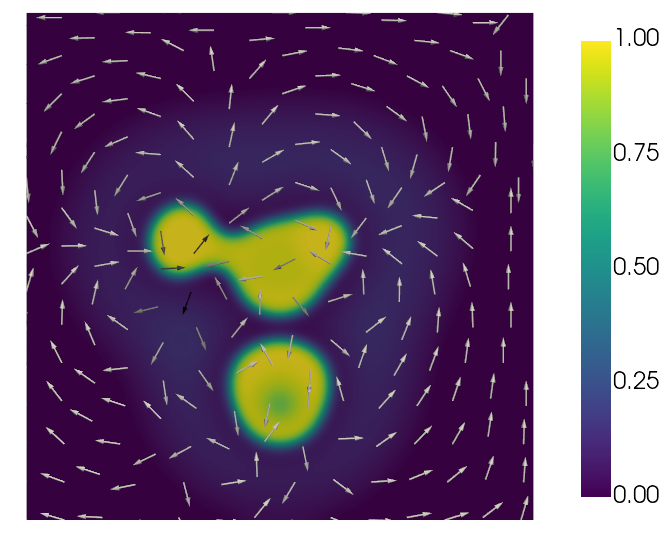}} &
		\raisebox{-0.47\height}{\includegraphics[scale=0.204]{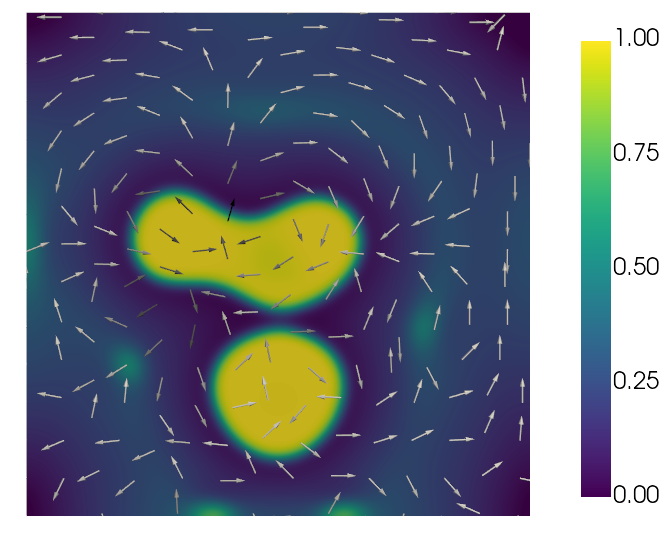}} &
		\raisebox{-0.47\height}{\includegraphics[scale=0.204]{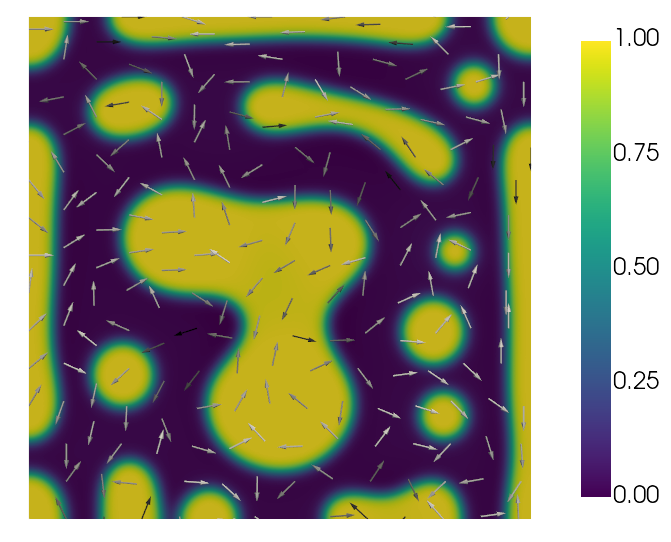}}\\
		\hdashline\vspace*{-0.5cm}\\
		\multirow{3}{*}{\vspace*{-6.2cm}\rotatebox[origin=c]{90}{$\boldsymbol{n}$}}&\rotatebox[origin=c]{90}{$K=0.1$} &
		\raisebox{-0.47\height}{\includegraphics[scale=0.204]{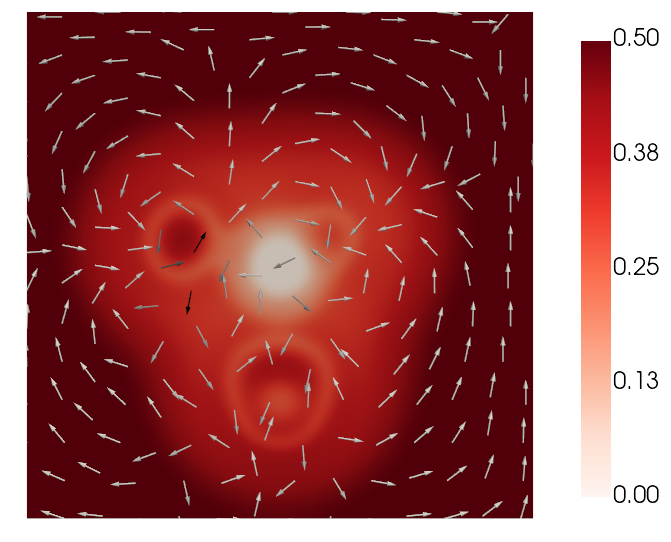}} &
		\raisebox{-0.47\height}{\includegraphics[scale=0.204]{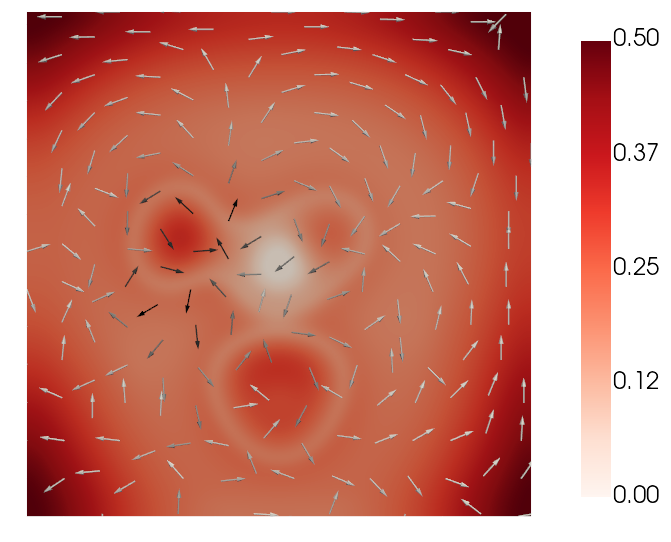}} &
		\raisebox{-0.47\height}{\includegraphics[scale=0.204]{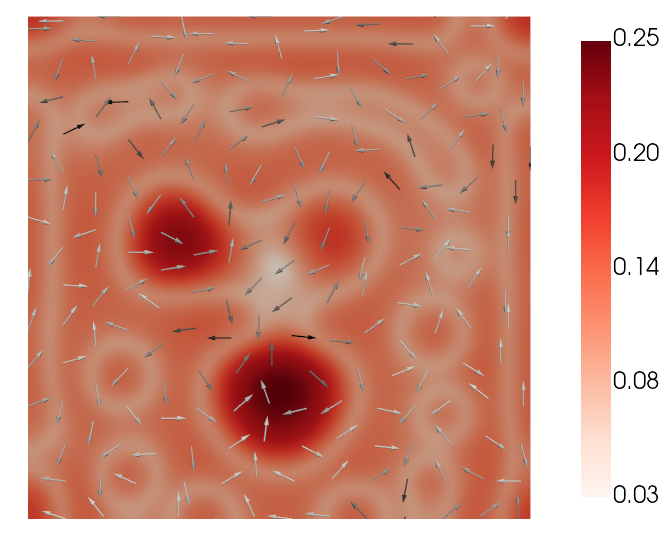}} \\
		&\rotatebox[origin=c]{90}{$K=1$} &
		\raisebox{-0.47\height}{\includegraphics[scale=0.204]{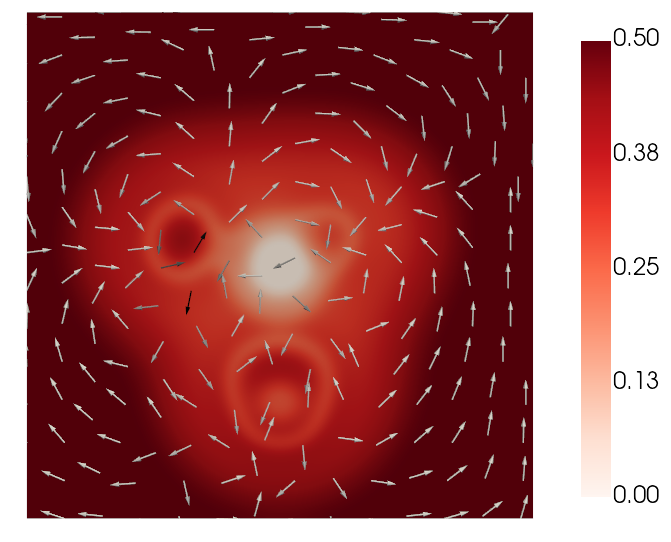}} &
		\raisebox{-0.47\height}{\includegraphics[scale=0.204]{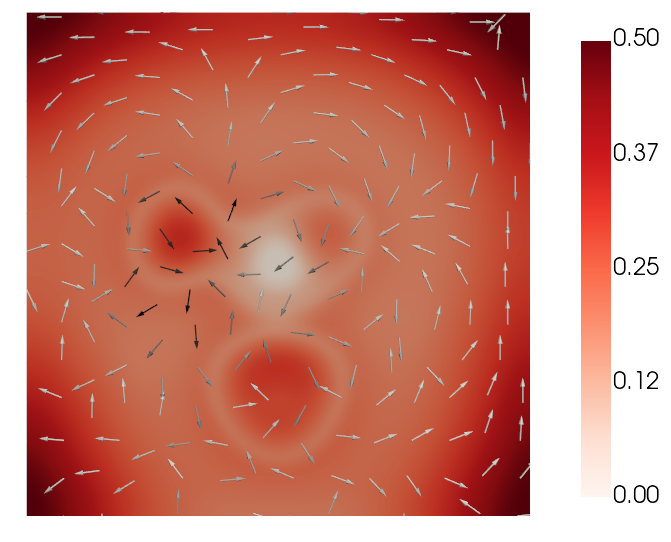}} &
		\raisebox{-0.47\height}{\includegraphics[scale=0.204]{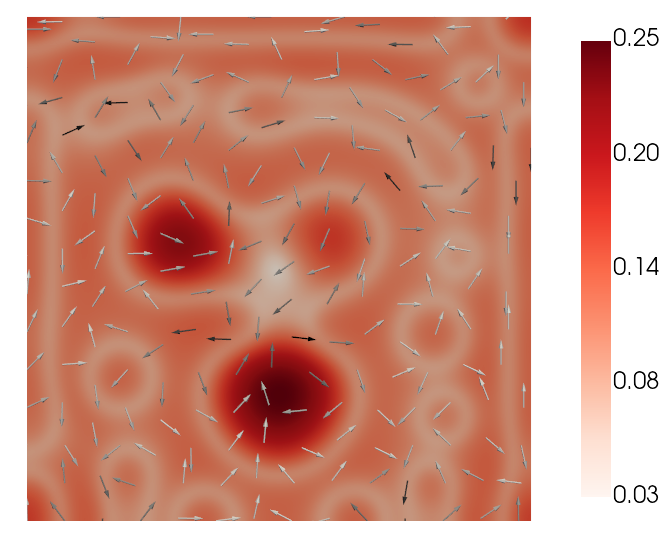}} \\
		&\rotatebox[origin=c]{90}{$K=10$} &
		\raisebox{-0.47\height}{\includegraphics[scale=0.204]{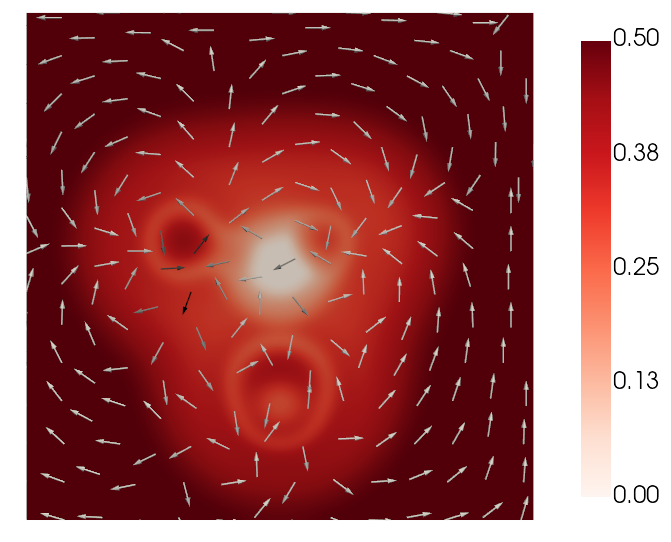}} &
		\raisebox{-0.47\height}{\includegraphics[scale=0.204]{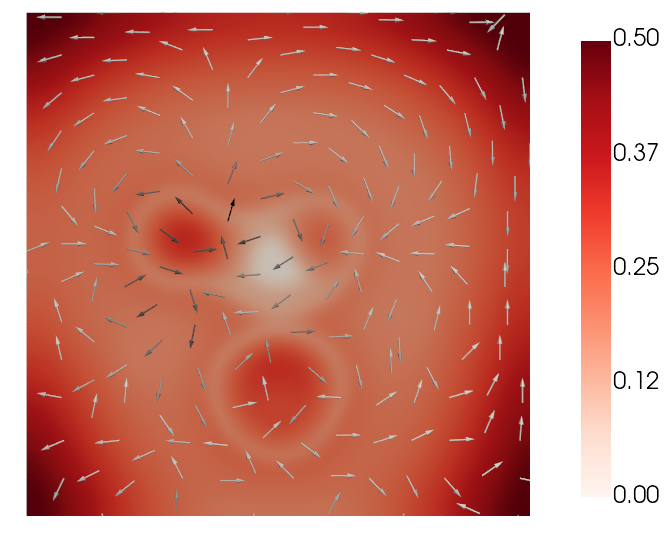}} &
		\raisebox{-0.47\height}{\includegraphics[scale=0.204]{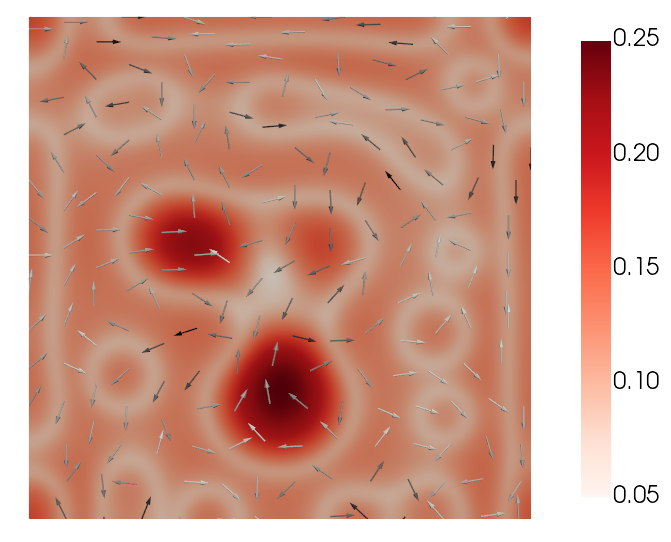}}
	\end{tabular}
	\caption{Tumor and nutrients with symmetric functions ($P_0=0.5$, $\chi_0=0.1$, $\Delta t=0.05$) at different time steps.}
	\label{fig:test-2_reference_symmetric}
\end{figure}

\begin{figure}
	\vspace*{-0.5cm}
	\centering
	\begin{tabular}{ccccc}
		& & \hspace*{-1cm}$t=10.0$ & \hspace*{-1cm}$t=20.0$ & \hspace*{-1cm}$t=50.0$ \\
		\multirow{3}{*}{\vspace*{-6.2cm}\rotatebox[origin=c]{90}{$\boldsymbol{u}$}}&\rotatebox[origin=c]{90}{$K=0.1$} &
		\raisebox{-0.47\height}{\includegraphics[scale=0.204]{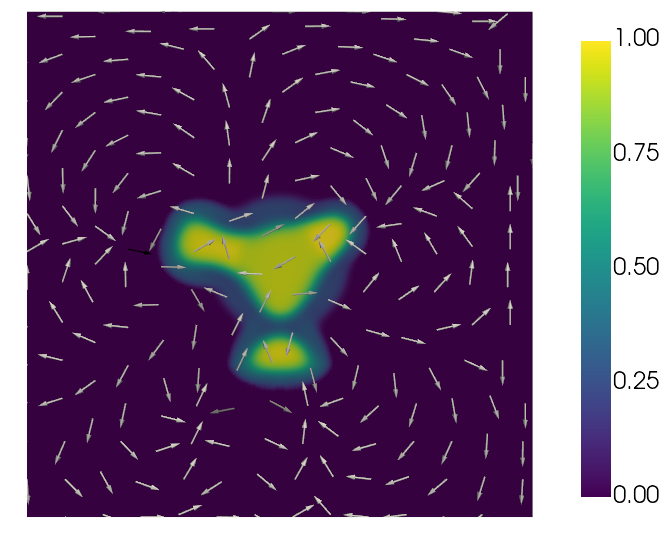}} &
		\raisebox{-0.47\height}{\includegraphics[scale=0.204]{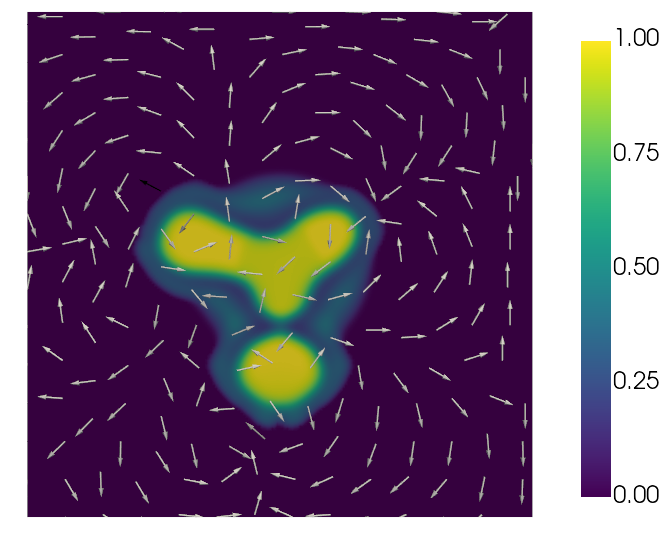}} &
		\raisebox{-0.47\height}{\includegraphics[scale=0.204]{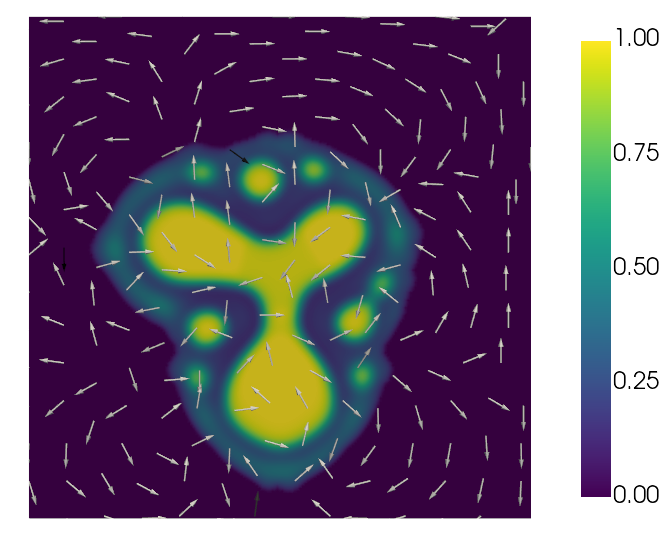}} \\
		&\rotatebox[origin=c]{90}{$K=1$} &
		\raisebox{-0.47\height}{\includegraphics[scale=0.204]{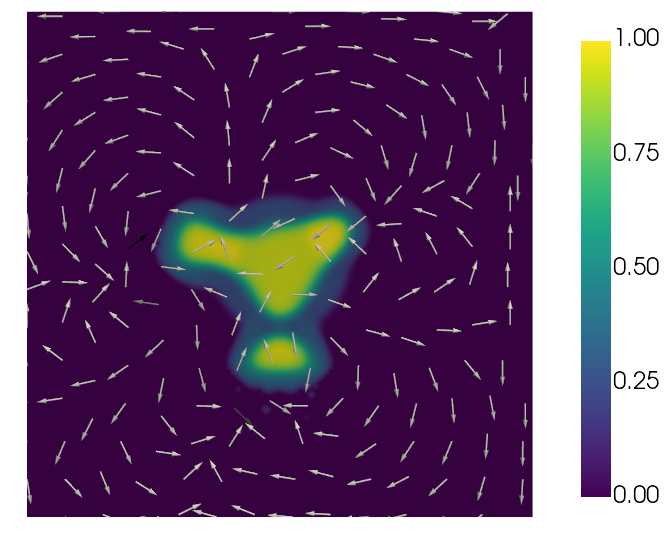}} &
		\raisebox{-0.47\height}{\includegraphics[scale=0.204]{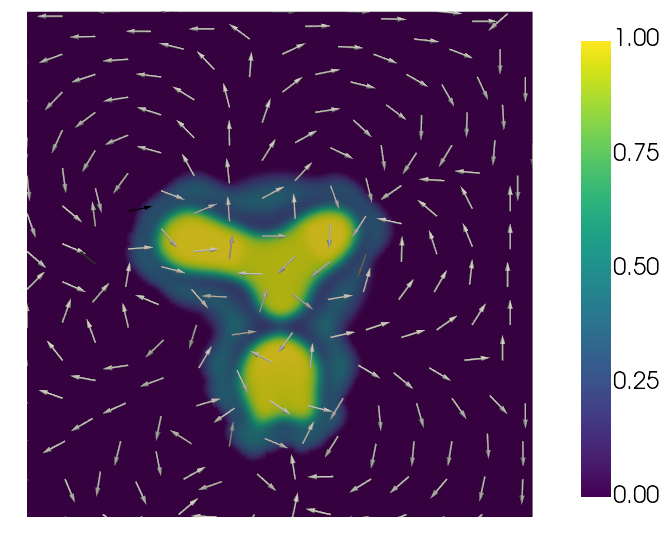}} &
		\raisebox{-0.47\height}{\includegraphics[scale=0.204]{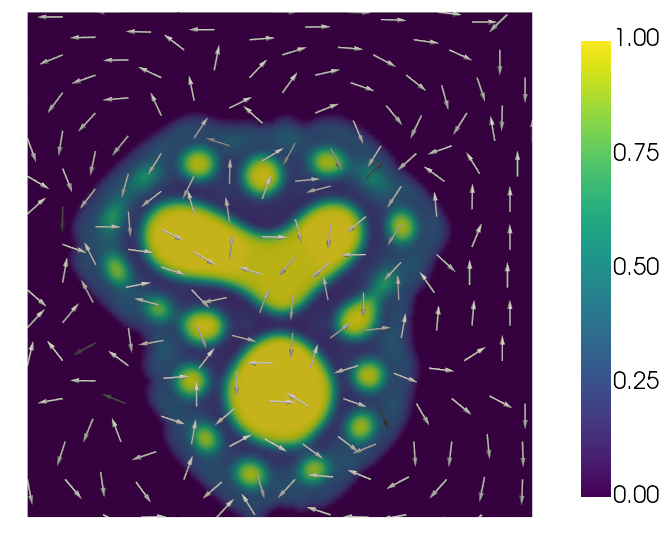}}\\
		&\rotatebox[origin=c]{90}{$K=10$} &
		\raisebox{-0.47\height}{\includegraphics[scale=0.204]{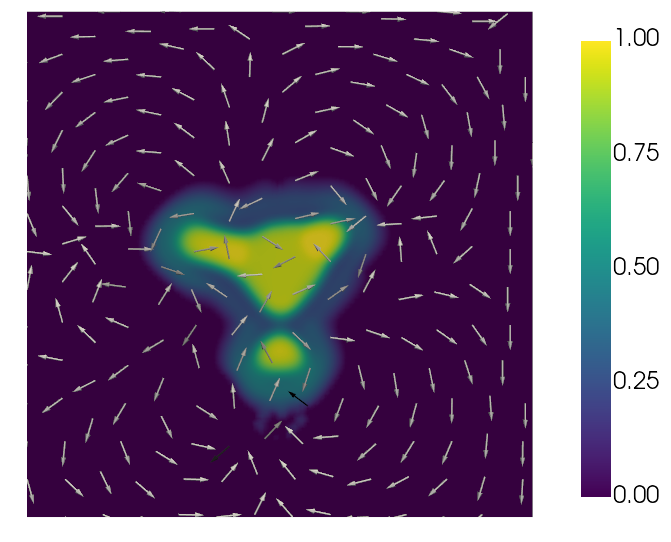}} &
		\raisebox{-0.47\height}{\includegraphics[scale=0.204]{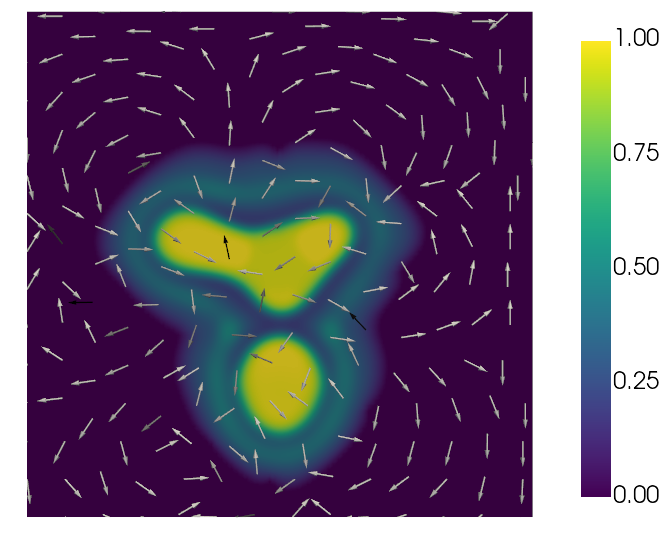}} &
		\raisebox{-0.47\height}{\includegraphics[scale=0.204]{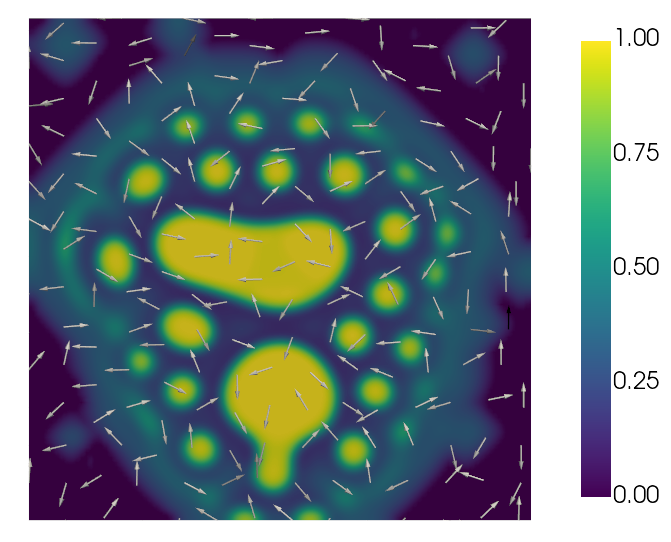}}\\
		\hdashline\vspace*{-0.5cm}\\
		\multirow{3}{*}{\vspace*{-6.2cm}\rotatebox[origin=c]{90}{$\boldsymbol{n}$}}&\rotatebox[origin=c]{90}{$K=0.1$} &
		\raisebox{-0.47\height}{\includegraphics[scale=0.204]{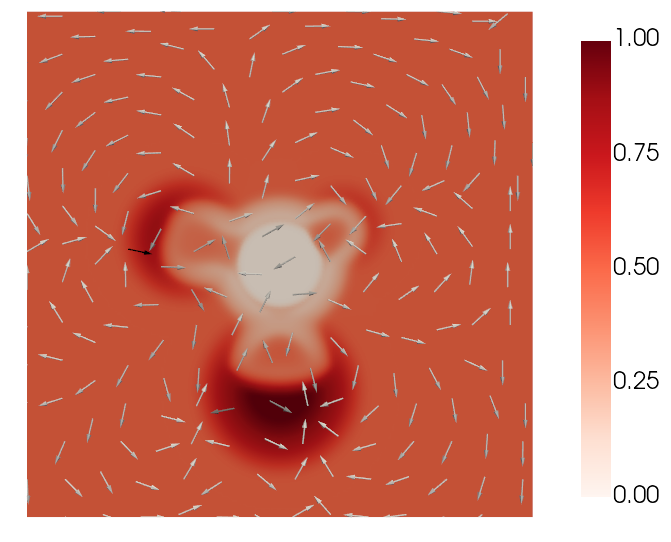}} &
		\raisebox{-0.47\height}{\includegraphics[scale=0.204]{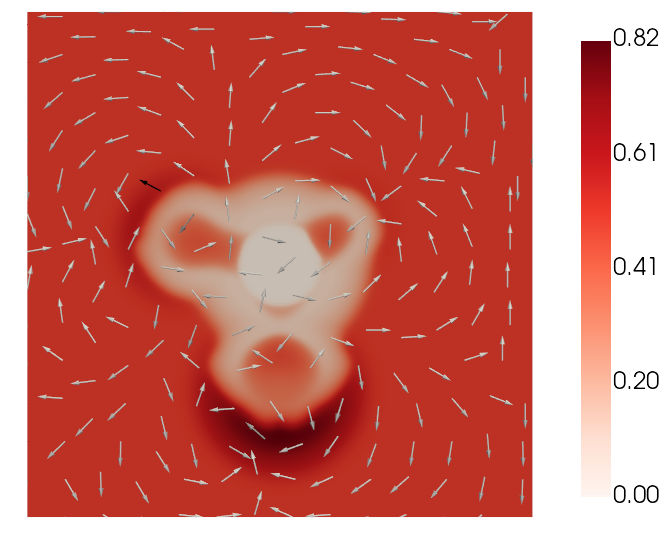}} &
		\raisebox{-0.47\height}{\includegraphics[scale=0.204]{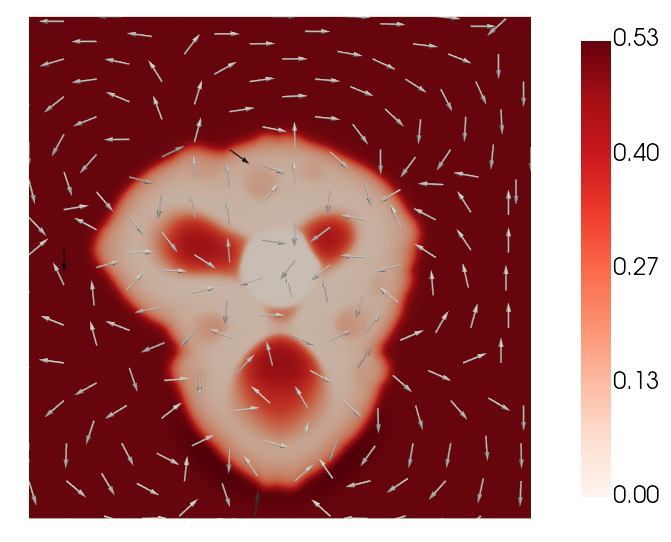}} \\
		&\rotatebox[origin=c]{90}{$K=1$} &
		\raisebox{-0.47\height}{\includegraphics[scale=0.204]{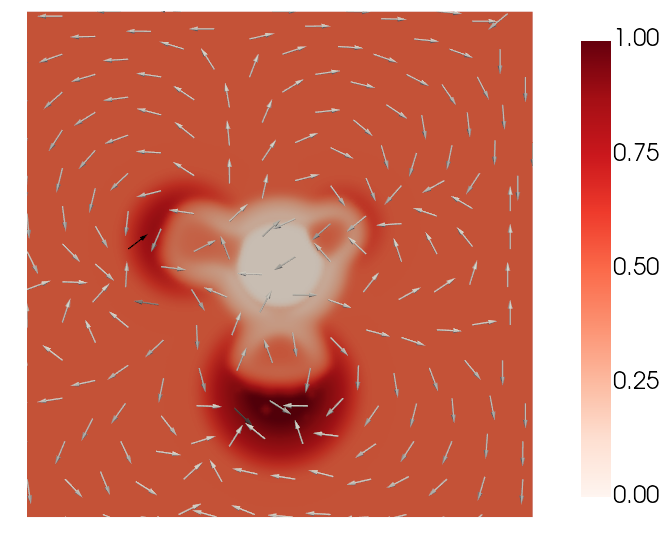}} &
		\raisebox{-0.47\height}{\includegraphics[scale=0.204]{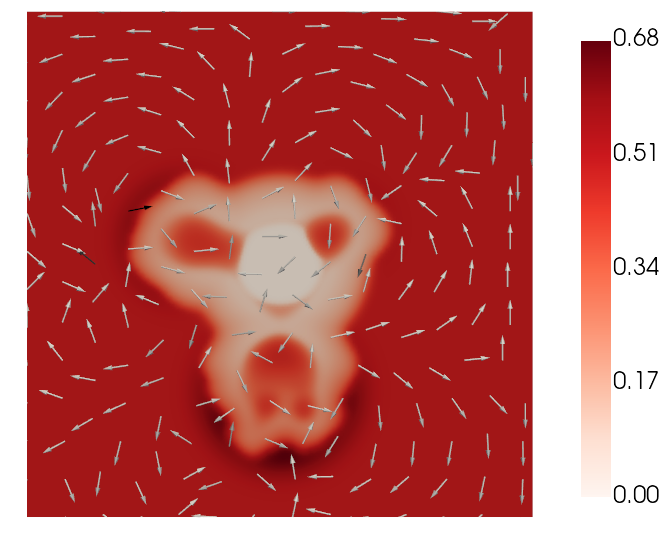}} &
		\raisebox{-0.47\height}{\includegraphics[scale=0.204]{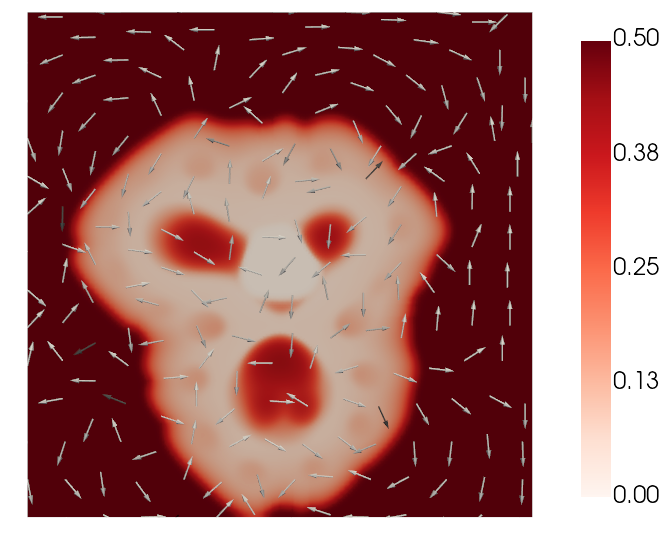}} \\
		&\rotatebox[origin=c]{90}{$K=10$} &
		\raisebox{-0.47\height}{\includegraphics[scale=0.204]{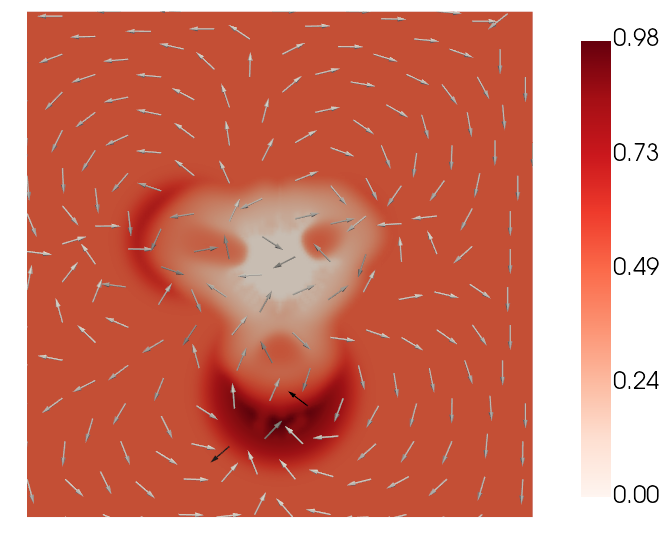}} &
		\raisebox{-0.47\height}{\includegraphics[scale=0.204]{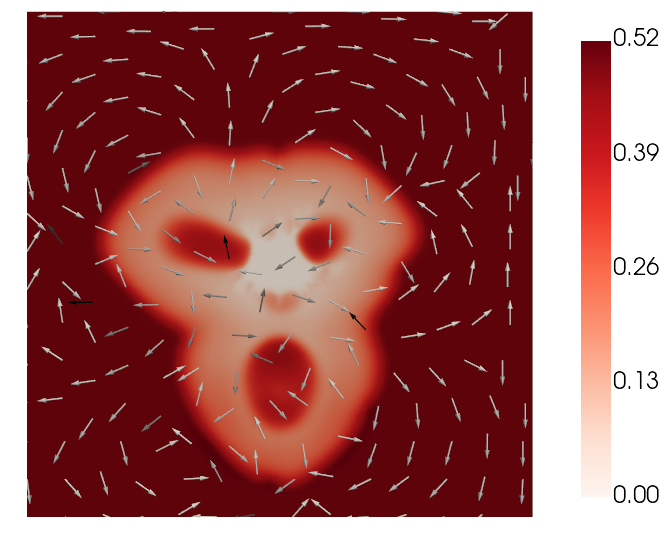}} &
		\raisebox{-0.47\height}{\includegraphics[scale=0.204]{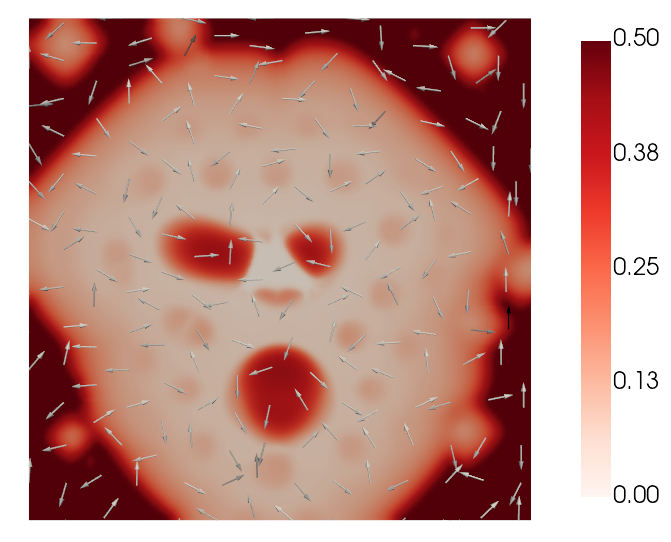}}
	\end{tabular}
	\caption{Tumor and nutrients with non-symmetric functions ($P_0=0.5$, $\chi_0=0.1$, $\Delta t=0.05$) at different time steps.}
	\label{fig:test-2_reference}
\end{figure}

\begin{figure}
	\vspace*{-0.5cm}
	\centering
	\begin{tabular}{ccc}
		& Symmetric \eqref{symmetric_functions} & Non-symmetric \eqref{nonsymmetric_functions} \\
		\rotatebox[origin=c]{90}{$\boldsymbol{u}$} &
		\hspace*{-0.4cm}\raisebox{-0.5\height}{\includegraphics[scale=0.52]{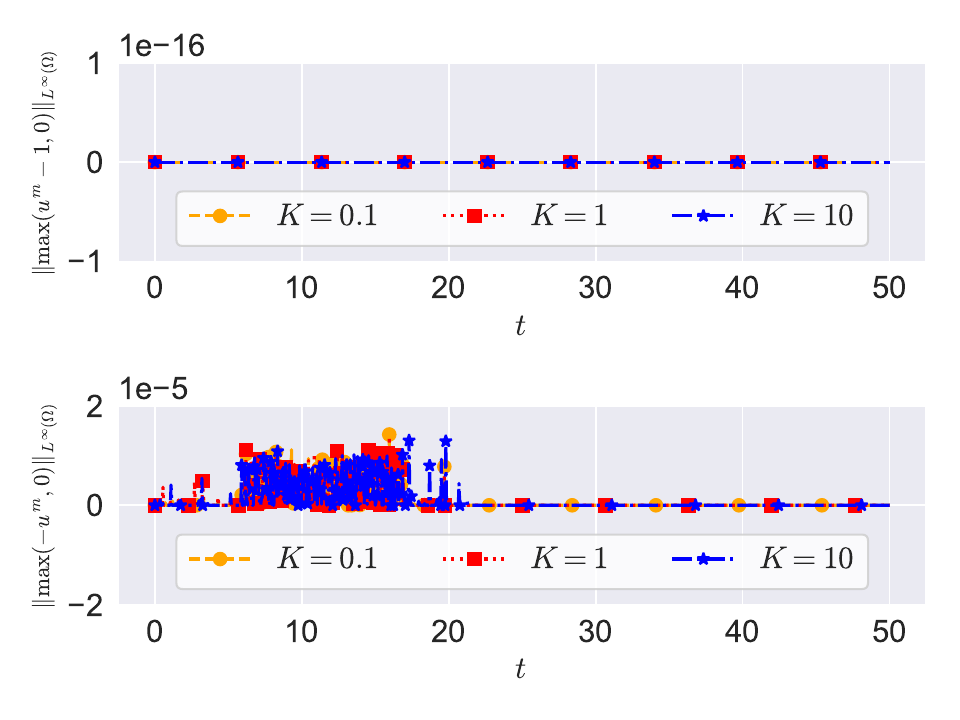}} &
		\hspace*{-0.4cm}\raisebox{-0.5\height}{\includegraphics[scale=0.52]{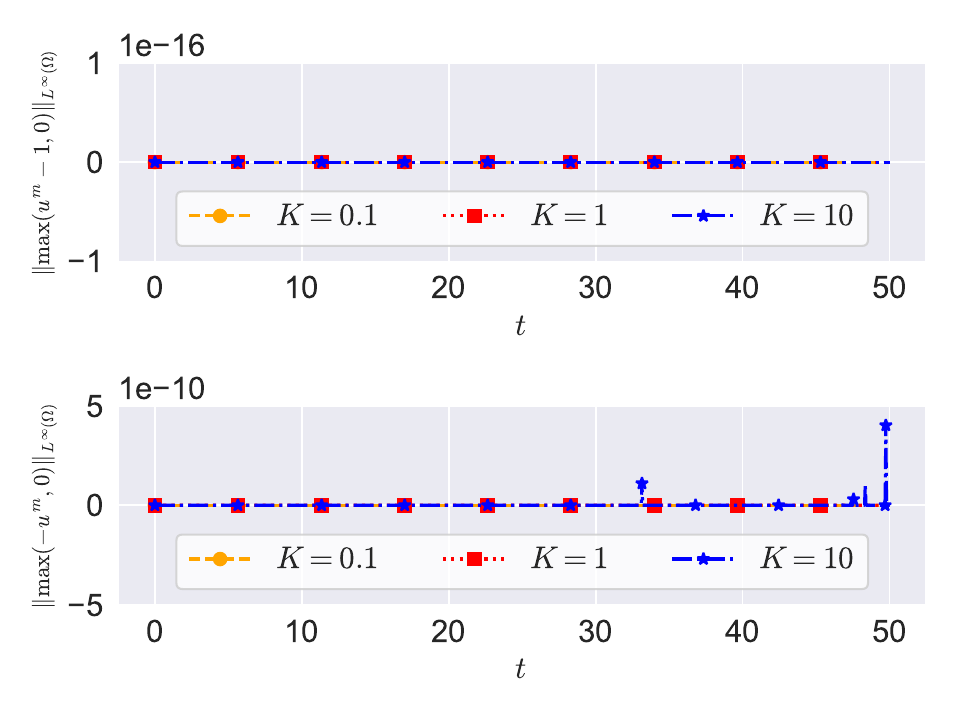}} \\
		\rotatebox[origin=c]{90}{$\boldsymbol{n}$} & \hspace*{-0.4cm}\raisebox{-0.5\height}{\includegraphics[scale=0.52]{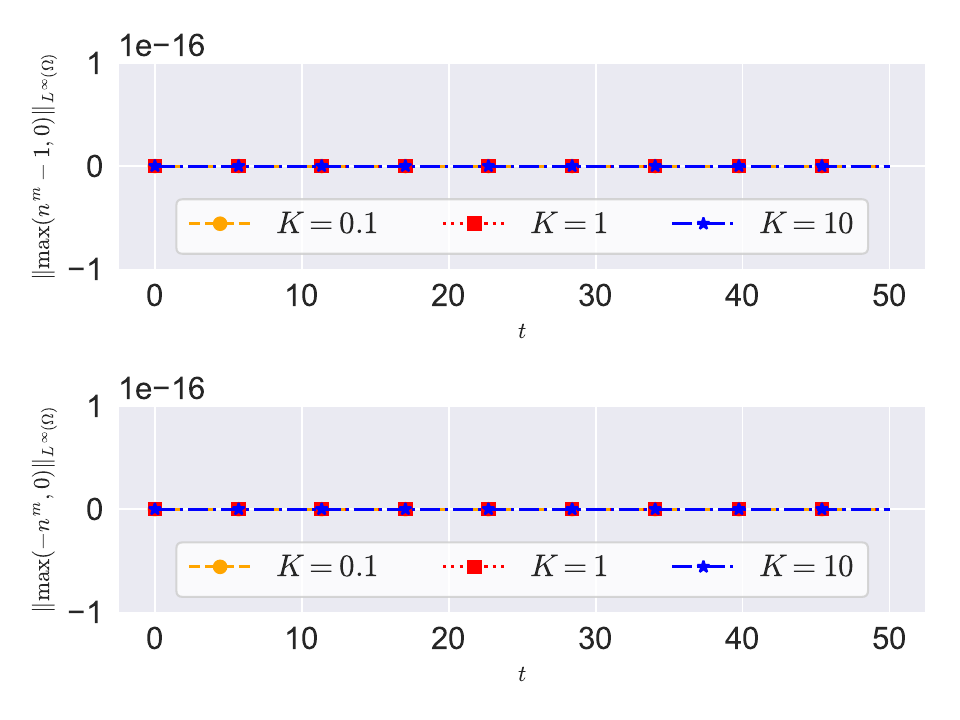}} &
		\hspace*{-0.4cm}\raisebox{-0.5\height}{\includegraphics[scale=0.52]{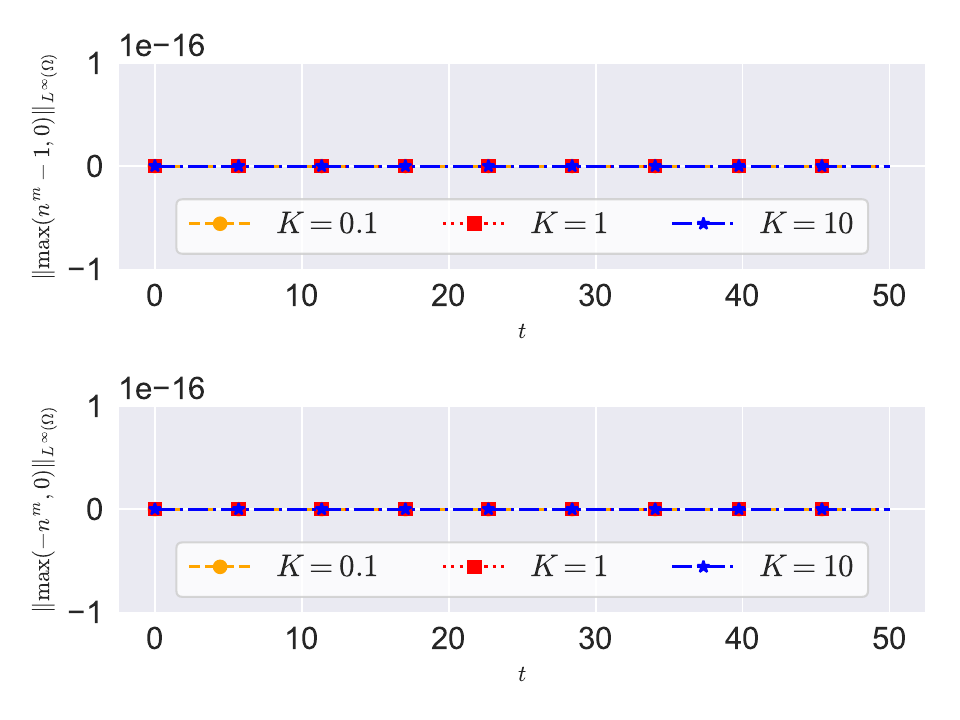}} \\
		\rotatebox[origin=c]{90}{\textbf{Energy}} & \raisebox{-0.5\height}{\includegraphics[scale=0.49]{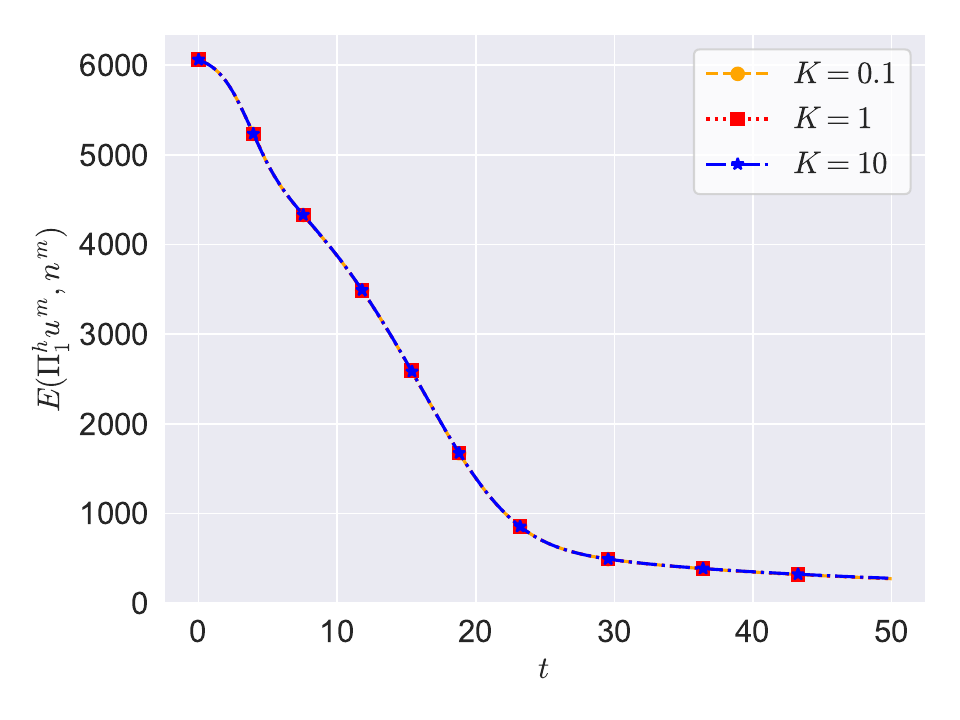}} &
		\raisebox{-0.47\height}{\includegraphics[scale=0.49]{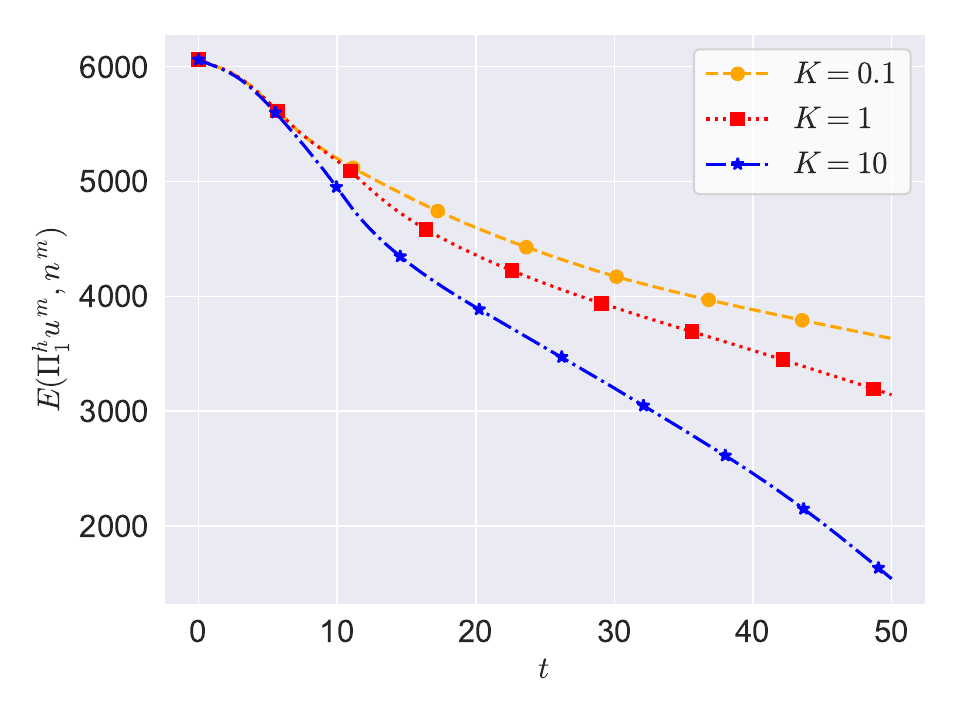}}
	\end{tabular}
	\caption{Evolution over time for different values of $K$ ($P_0=0.5$, $\chi_0=0.1$, $\Delta t=0.1$) of the quantities: $\|\max(u^m-1,0)\|_{L^\infty(\Omega)}$ and $\|\max(-u^m,0)\|_{L^\infty(\Omega)}$ (first row); $\|\max(n^m-1,0)\|_{L^\infty(\Omega)}$ and $\|\max(-n^m,0)\|_{L^\infty(\Omega)}$ (second row); energy $E(\up^m, n^m)$ (third row).}
	\label{fig:test-2_min_max_energy}
\end{figure}

\begin{figure}
	\centering
	\begin{tabular}{ccccc}
			&& \hspace*{-1cm}$t=30.0$ & \hspace*{-1cm}$t=50.0$ & \hspace*{-1cm}$t=100.0$ \\
			\multirow{3}{*}{\vspace*{-6.2cm}\rotatebox[origin=c]{90}{\textbf{Symmetric \eqref{symmetric_functions}}}} & \rotatebox[origin=c]{90}{$K=0.1$} &
			\raisebox{-0.47\height}{\includegraphics[scale=0.204]{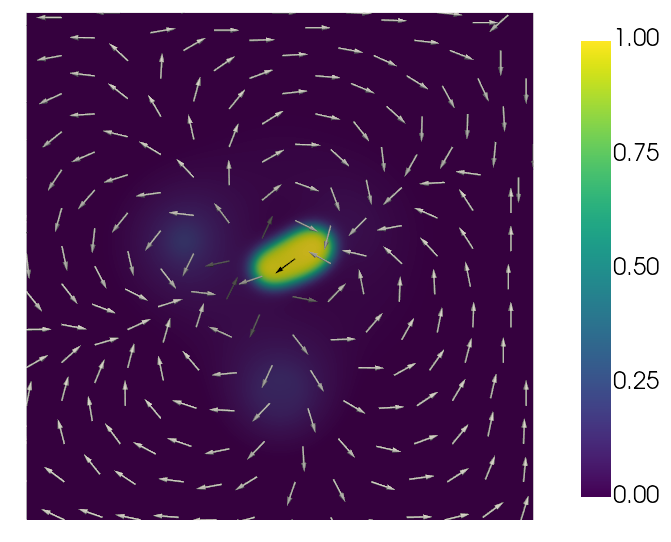}} &
			\raisebox{-0.47\height}{\includegraphics[scale=0.204]{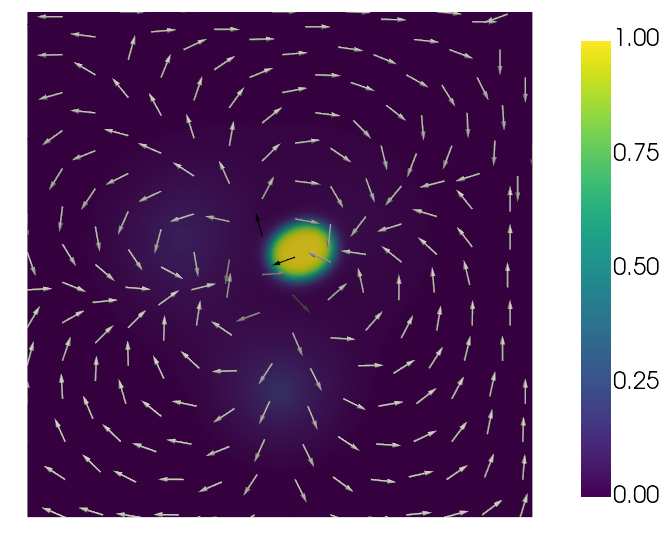}} &
			\raisebox{-0.47\height}{\includegraphics[scale=0.204]{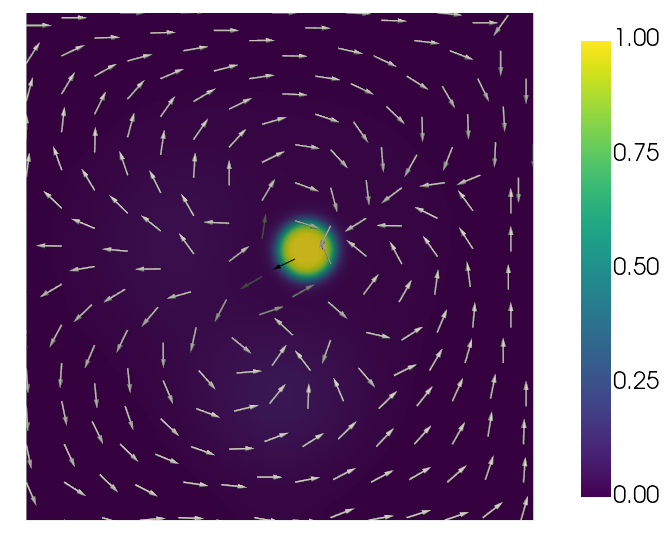}} \\
			& \rotatebox[origin=c]{90}{$K=1$} &
			\raisebox{-0.47\height}{\includegraphics[scale=0.204]{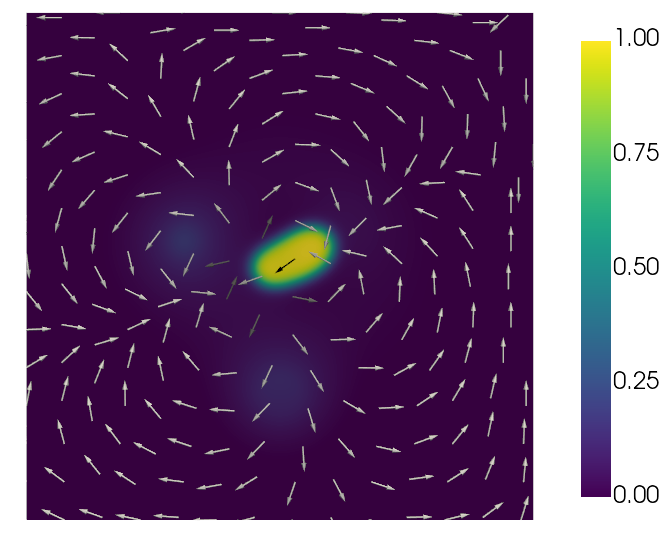}} &
			\raisebox{-0.47\height}{\includegraphics[scale=0.204]{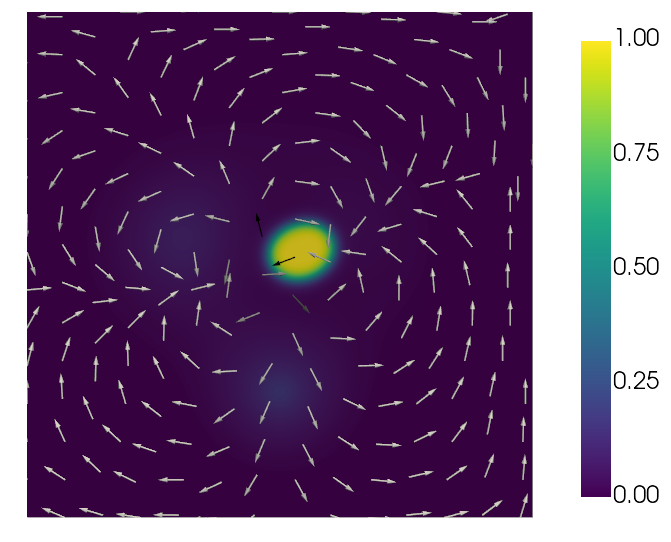}} &
			\raisebox{-0.47\height}{\includegraphics[scale=0.204]{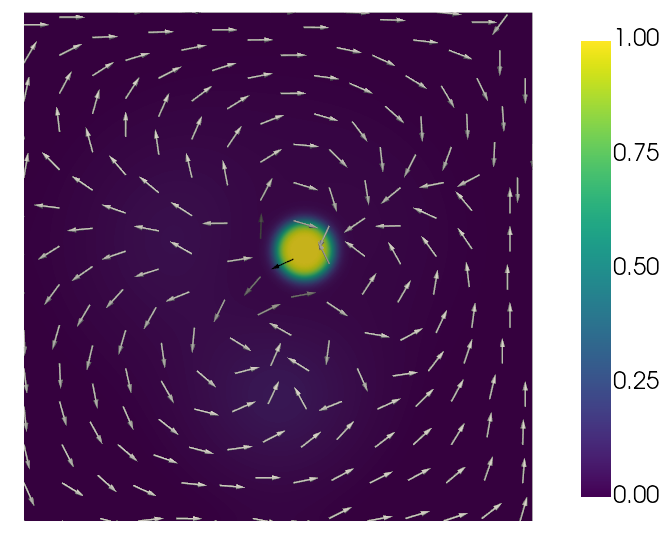}} \\
			& \rotatebox[origin=c]{90}{$K=10$} &
			\raisebox{-0.47\height}{\includegraphics[scale=0.204]{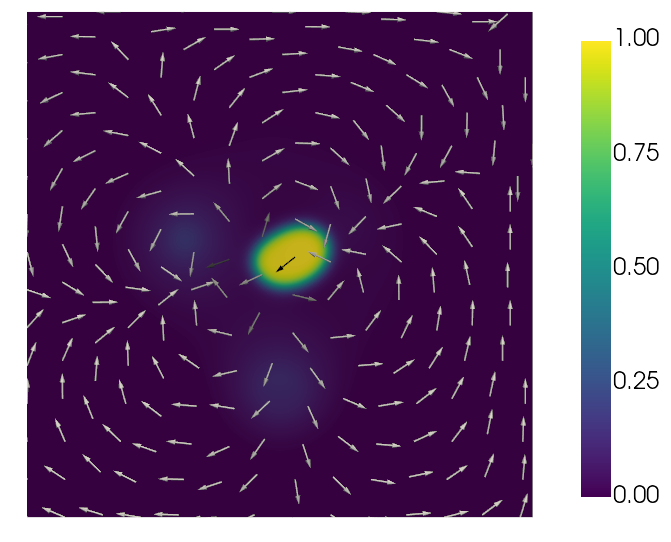}} &
			\raisebox{-0.47\height}{\includegraphics[scale=0.204]{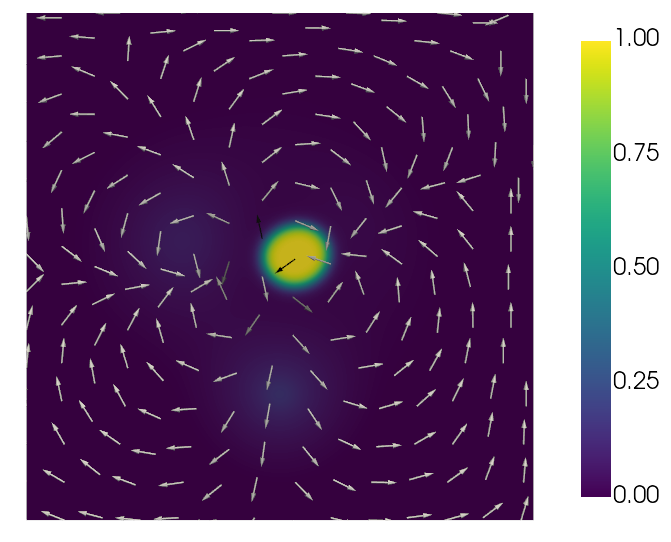}} &
			\raisebox{-0.47\height}{\includegraphics[scale=0.204]{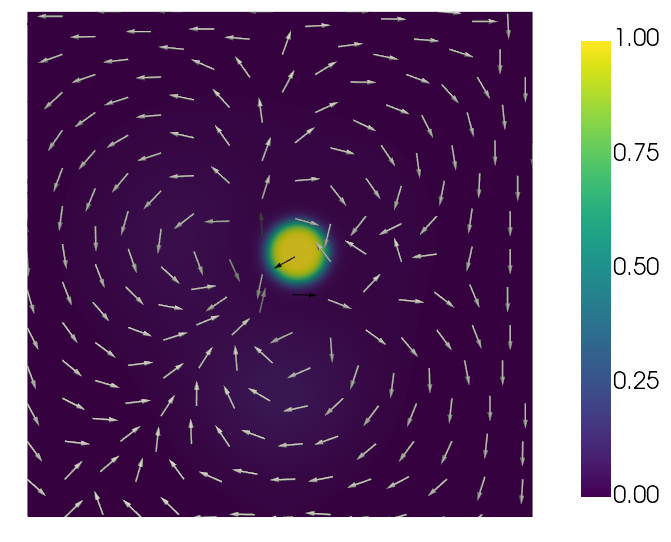}} \\
			\hdashline\vspace*{-0.5cm}\\
			\multirow{3}{*}{\vspace*{-6.2cm}\rotatebox[origin=c]{90}{\textbf{Non-symmetric \eqref{nonsymmetric_functions}}}} & \rotatebox[origin=c]{90}{$K=0.1$} &
			\raisebox{-0.47\height}{\includegraphics[scale=0.204]{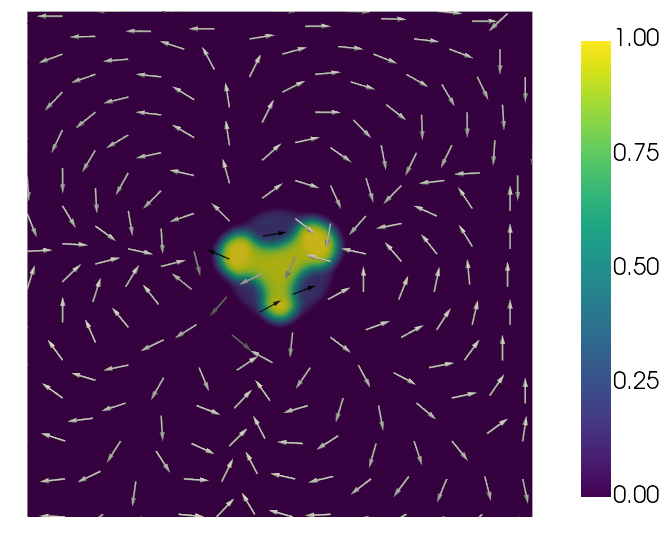}} &
			\raisebox{-0.47\height}{\includegraphics[scale=0.204]{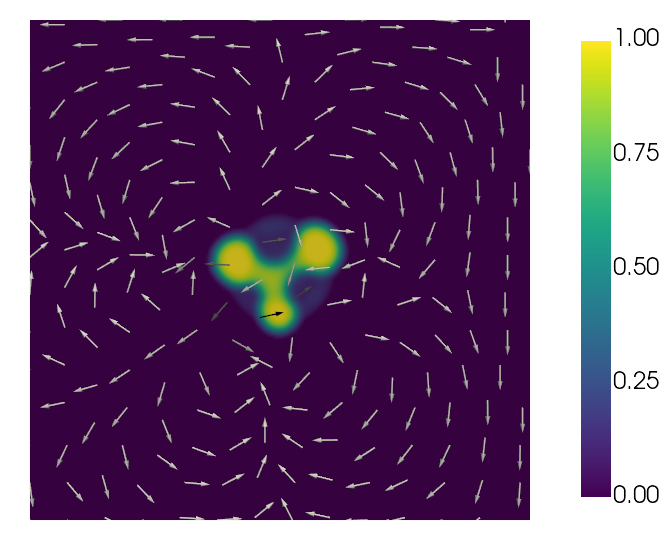}} &
			\raisebox{-0.47\height}{\includegraphics[scale=0.204]{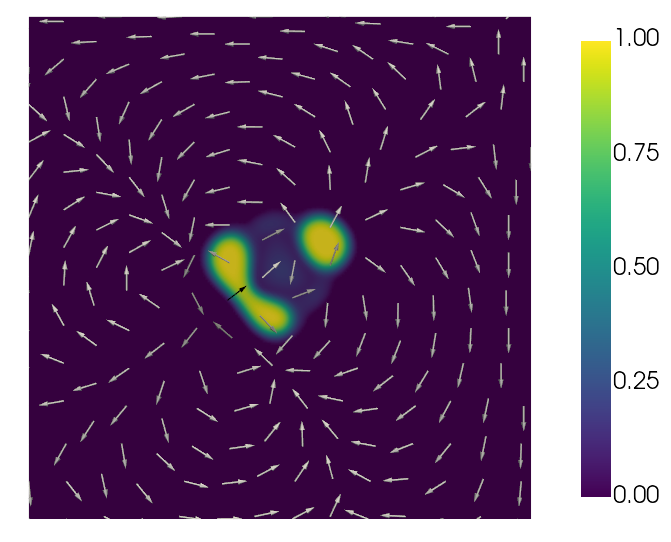}} \\
			& \rotatebox[origin=c]{90}{$K=1$} &
			\raisebox{-0.47\height}{\includegraphics[scale=0.204]{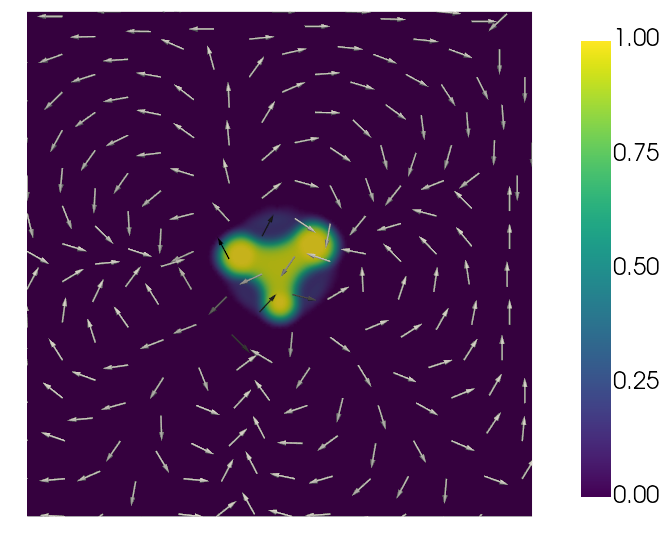}} &
			\raisebox{-0.47\height}{\includegraphics[scale=0.204]{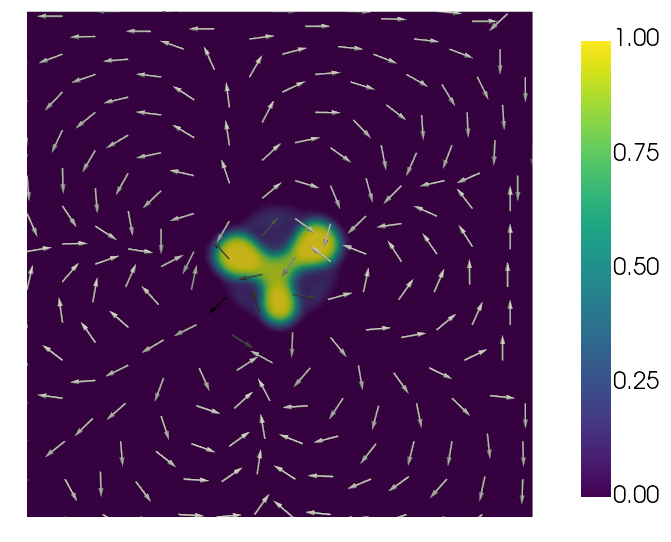}} &
			\raisebox{-0.47\height}{\includegraphics[scale=0.204]{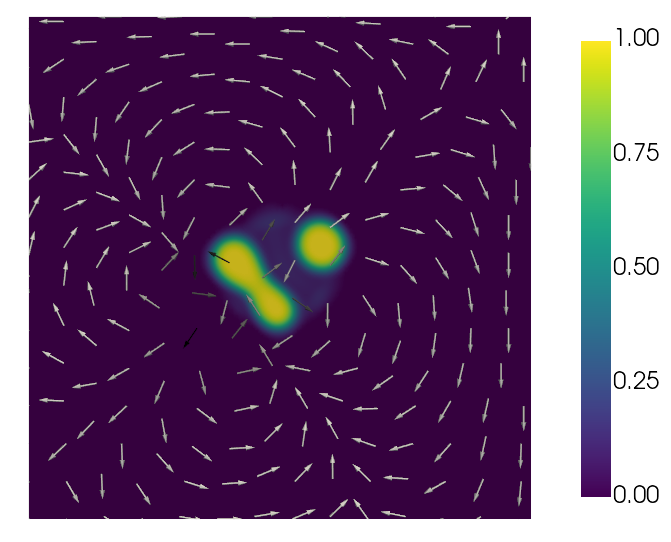}} \\
			& \rotatebox[origin=c]{90}{$K=10$} &
			\raisebox{-0.47\height}{\includegraphics[scale=0.204]{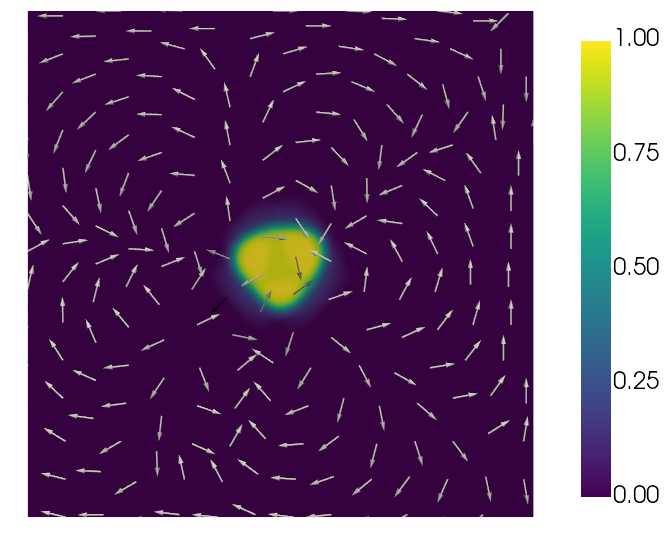}} &
			\raisebox{-0.47\height}{\includegraphics[scale=0.204]{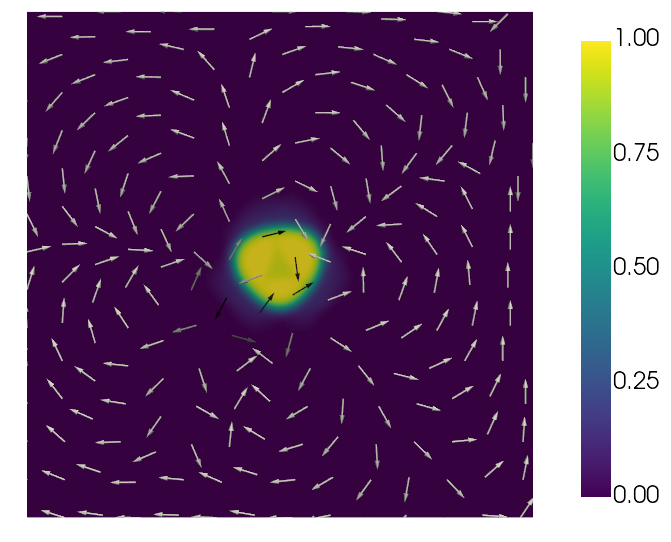}} &
			\raisebox{-0.47\height}{\includegraphics[scale=0.204]{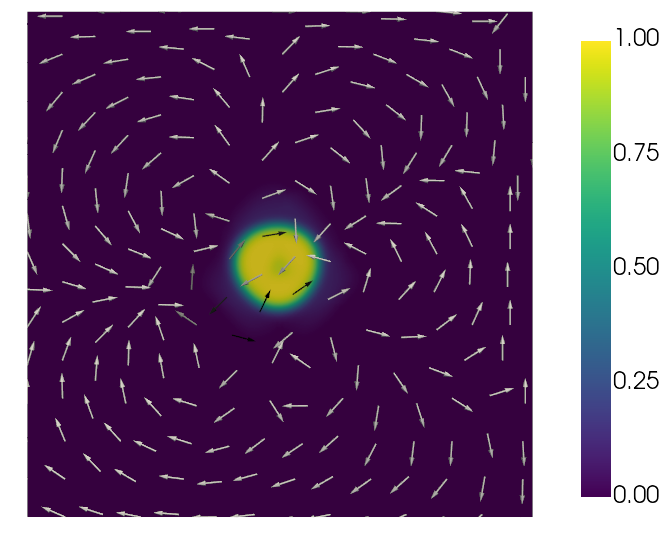}}
	\end{tabular}
	\caption{Tumor for test with $P_0=0.001$, $\chi_0=0.1$, $\Delta t=0.05$ at different time steps.}
	\label{fig:test-2_P0-0.001_u}
\end{figure}

\begin{figure}
	\centering
	\begin{tabular}{ccccc}
			&& \hspace*{-1cm}$t=50.0$ & \hspace*{-1cm}$t=80.0$ & \hspace*{-1cm}$t=200.0$ \\
			\multirow{3}{*}{\vspace*{-6.2cm}\rotatebox[origin=c]{90}{\textbf{Symmetric \eqref{symmetric_functions}}}} & \rotatebox[origin=c]{90}{$K=0.1$} &
			\raisebox{-0.47\height}{\includegraphics[scale=0.204]{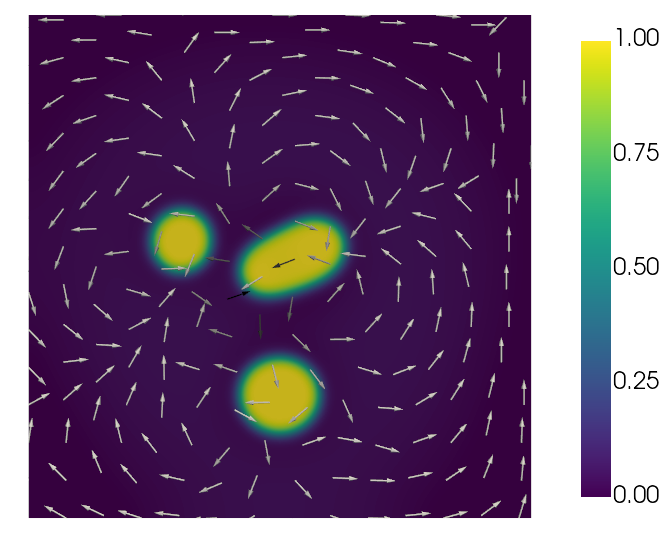}} &
			\raisebox{-0.47\height}{\includegraphics[scale=0.204]{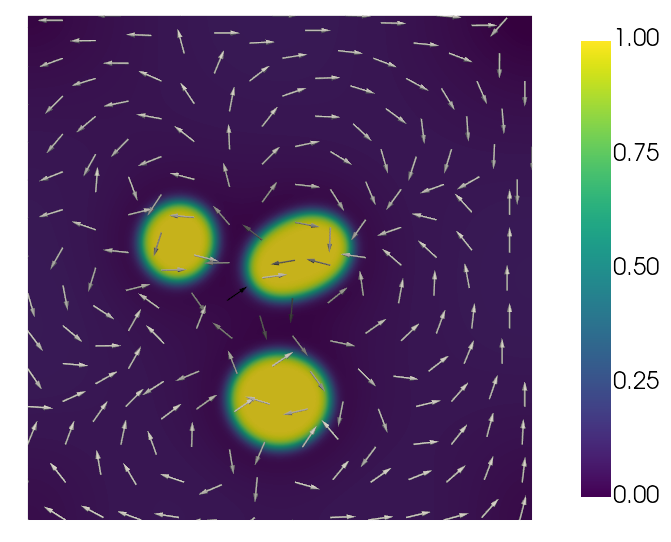}} &
			\raisebox{-0.47\height}{\includegraphics[scale=0.204]{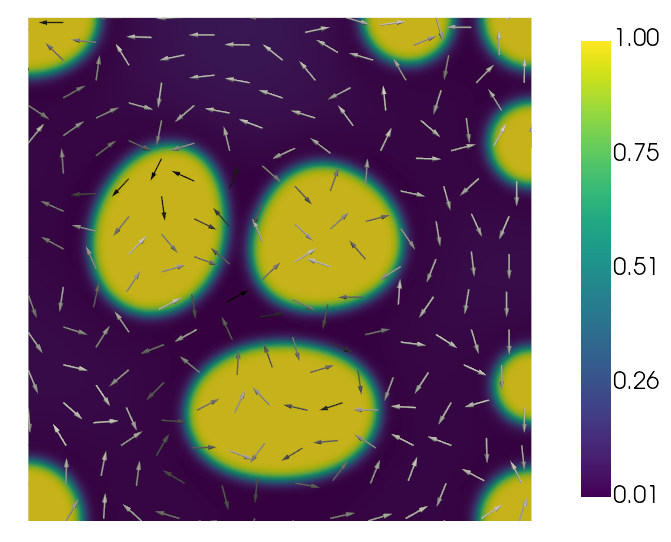}} \\
			& \rotatebox[origin=c]{90}{$K=1$} &
			\raisebox{-0.47\height}{\includegraphics[scale=0.204]{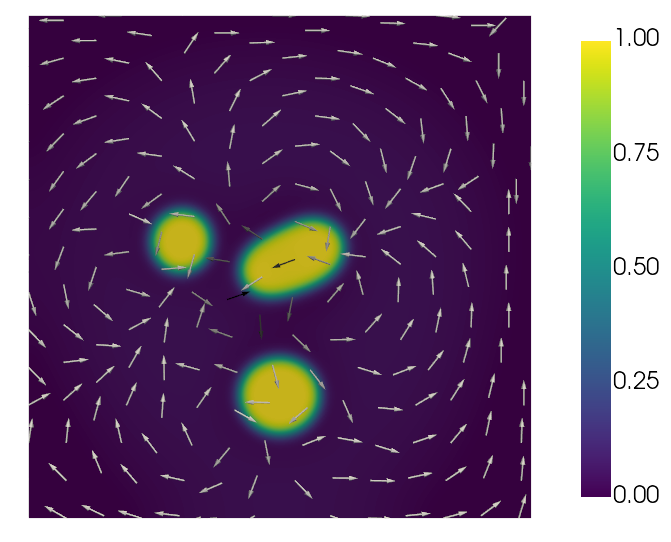}} &
			\raisebox{-0.47\height}{\includegraphics[scale=0.204]{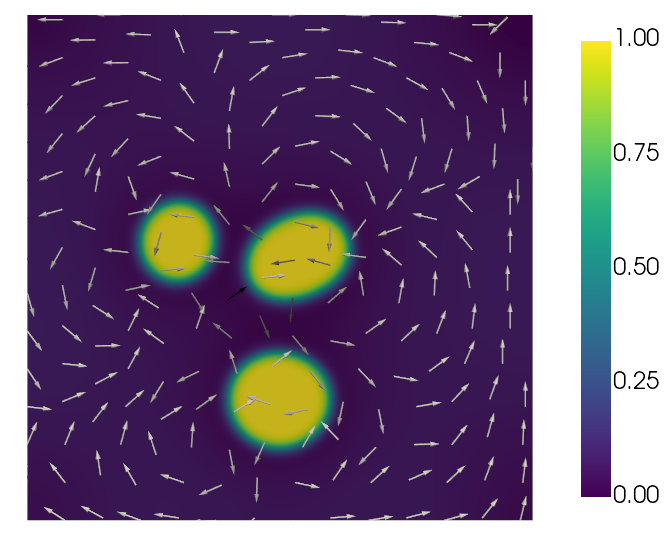}} &
			\raisebox{-0.47\height}{\includegraphics[scale=0.204]{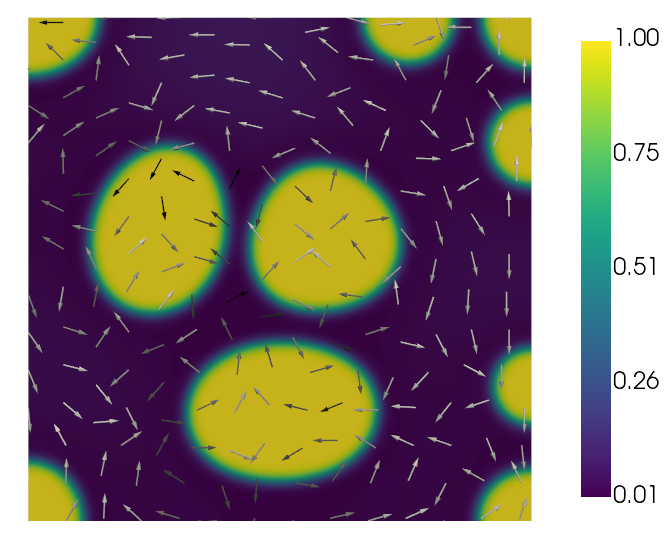}} \\
			& \rotatebox[origin=c]{90}{$K=10$} &
			\raisebox{-0.47\height}{\includegraphics[scale=0.204]{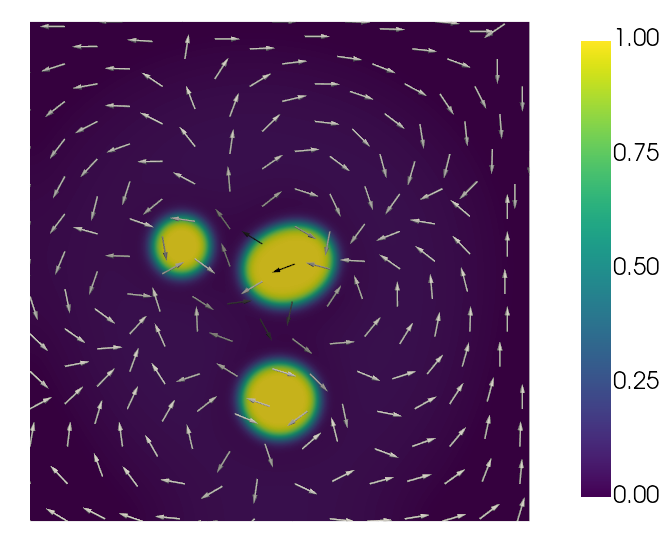}} &
			\raisebox{-0.47\height}{\includegraphics[scale=0.204]{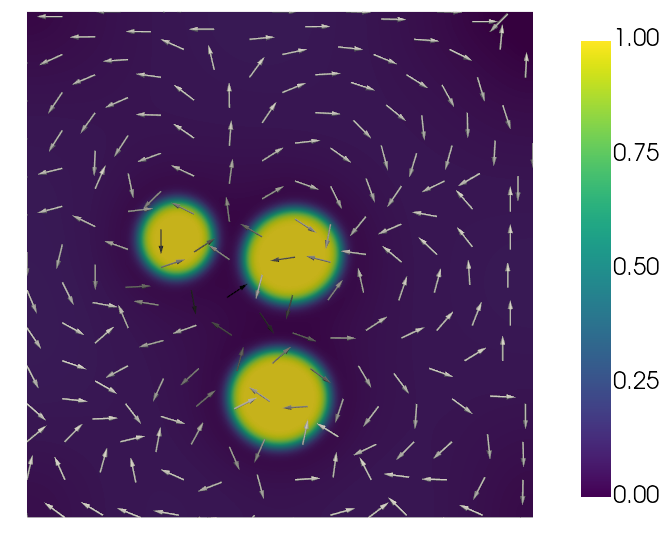}} &
			\raisebox{-0.47\height}{\includegraphics[scale=0.204]{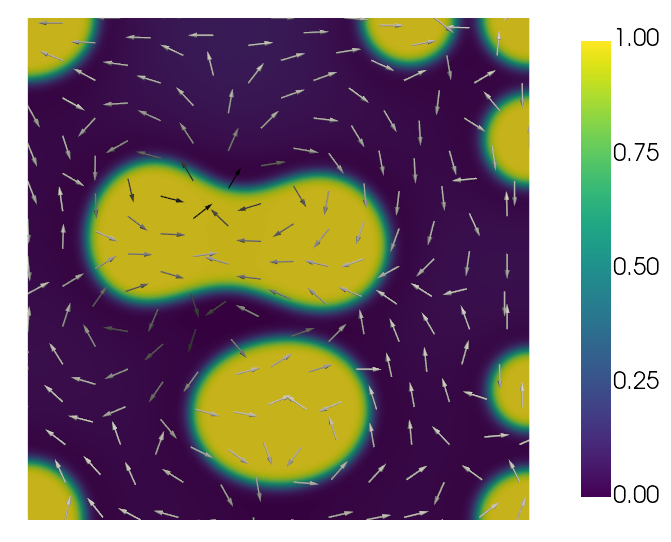}} \\
			\hdashline\vspace*{-0.5cm}\\
			\multirow{3}{*}{\vspace*{-6.2cm}\rotatebox[origin=c]{90}{\textbf{Non-symmetric \eqref{nonsymmetric_functions}}}} & \rotatebox[origin=c]{90}{$K=0.1$} &
			\raisebox{-0.47\height}{\includegraphics[scale=0.204]{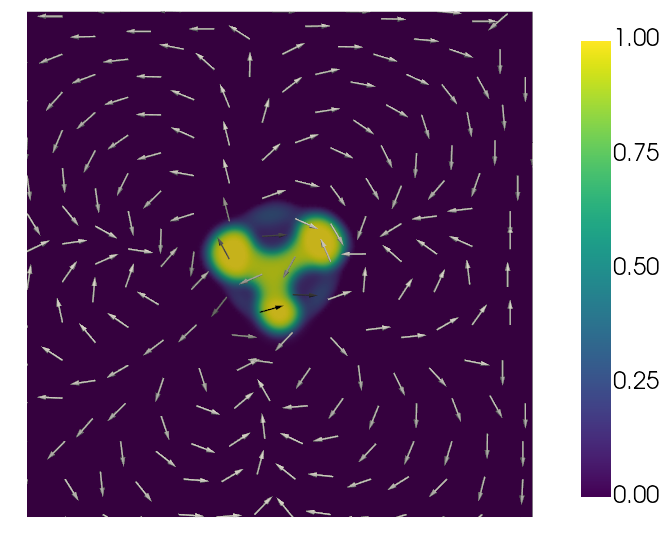}} &
			\raisebox{-0.47\height}{\includegraphics[scale=0.204]{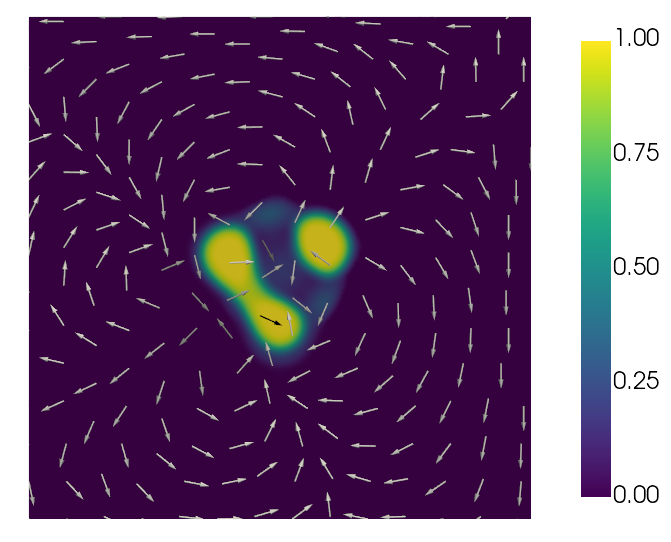}} &
			\raisebox{-0.47\height}{\includegraphics[scale=0.204]{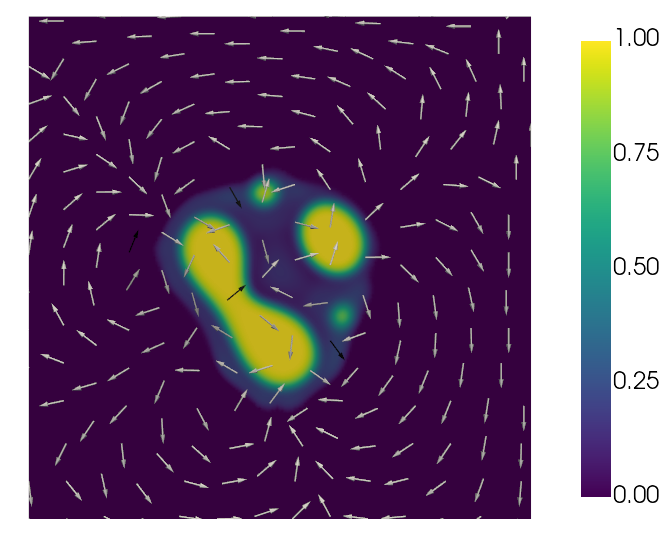}} \\
			& \rotatebox[origin=c]{90}{$K=1$} &
			\raisebox{-0.47\height}{\includegraphics[scale=0.204]{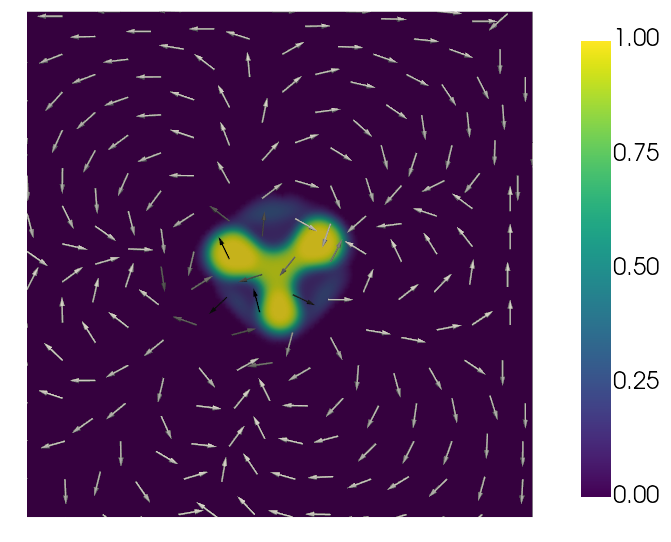}} &
			\raisebox{-0.47\height}{\includegraphics[scale=0.204]{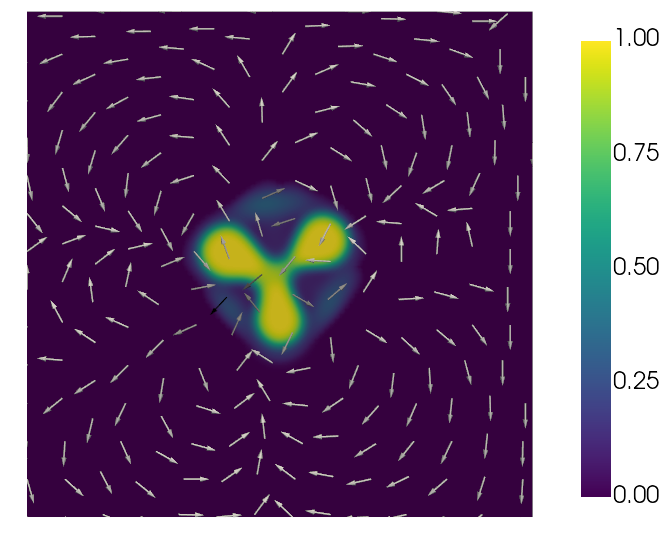}} &
			\raisebox{-0.47\height}{\includegraphics[scale=0.204]{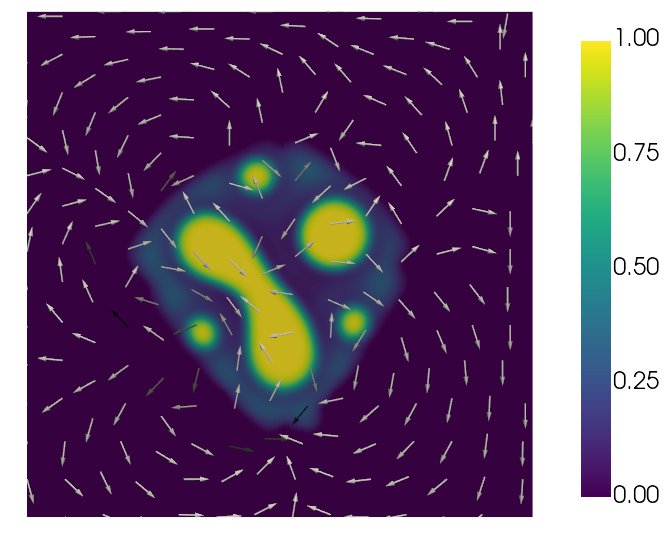}} \\
			& \rotatebox[origin=c]{90}{$K=10$} &
			\raisebox{-0.47\height}{\includegraphics[scale=0.204]{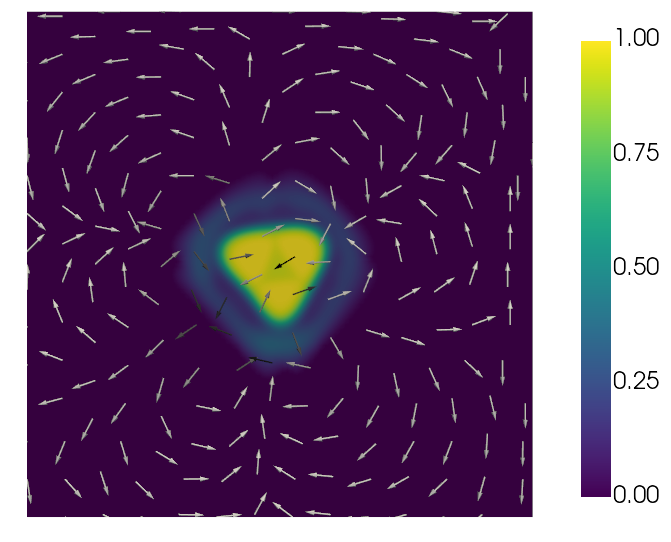}} &
			\raisebox{-0.47\height}{\includegraphics[scale=0.204]{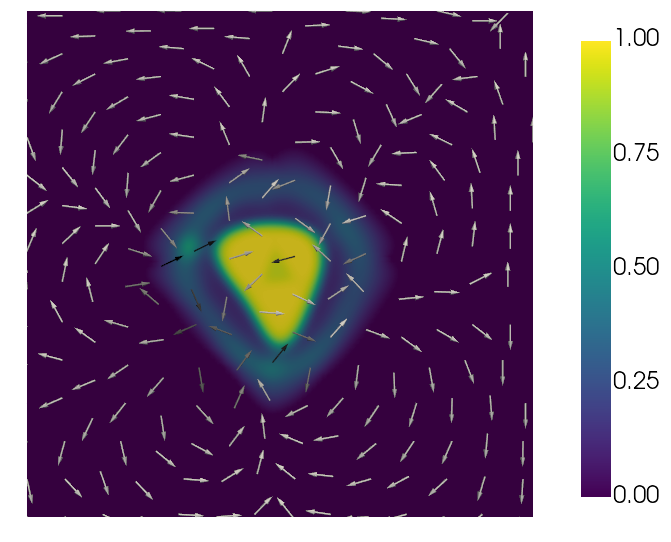}} &
			\raisebox{-0.47\height}{\includegraphics[scale=0.204]{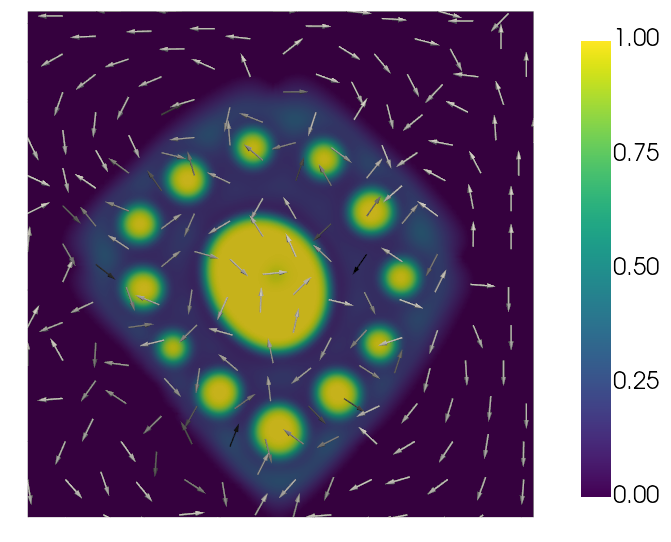}}
	\end{tabular}
	\caption{Tumor for test with $P_0=0.05$, $\chi_0=0.1$, $\Delta t=0.05$ at different time steps.}
	\label{fig:test-2_P0-0.05_u}
\end{figure}

\begin{figure}
	\centering
	\begin{tabular}{ccccc}
			&& \hspace*{-1cm}$t=1.25$ & \hspace*{-1cm}$t=6.25$ & \hspace*{-1cm}$t=12.5$ \\
			\multirow{3}{*}{\vspace*{-6.2cm}\rotatebox[origin=c]{90}{\textbf{Symmetric \eqref{symmetric_functions}}}} & \rotatebox[origin=c]{90}{$K=0.1$} &
			\raisebox{-0.47\height}{\includegraphics[scale=0.204]{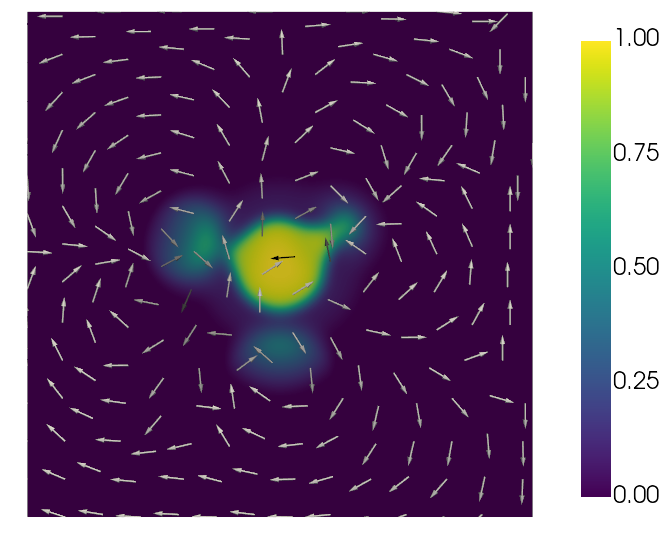}} &
			\raisebox{-0.47\height}{\includegraphics[scale=0.204]{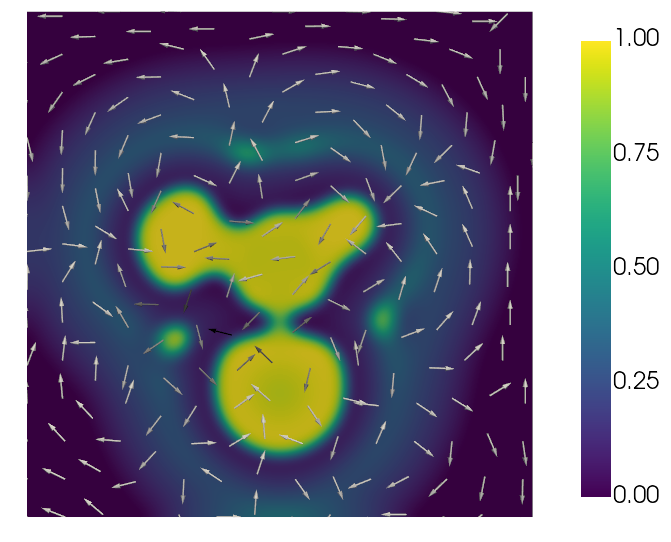}} &
			\raisebox{-0.47\height}{\includegraphics[scale=0.204]{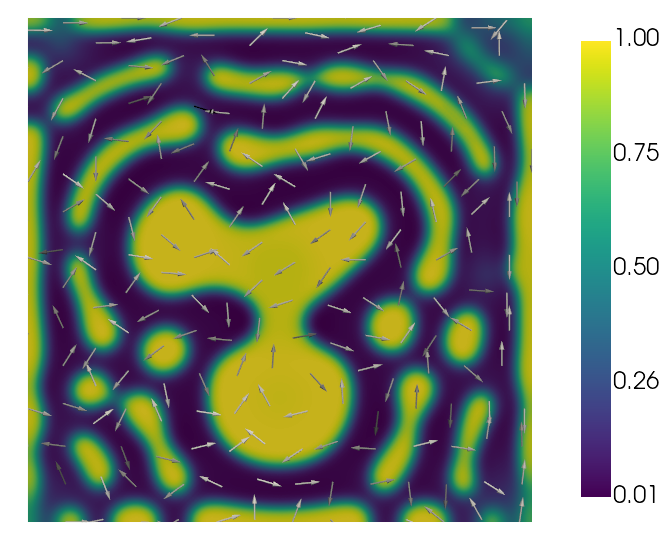}} \\
			& \rotatebox[origin=c]{90}{$K=1$} &
			\raisebox{-0.47\height}{\includegraphics[scale=0.204]{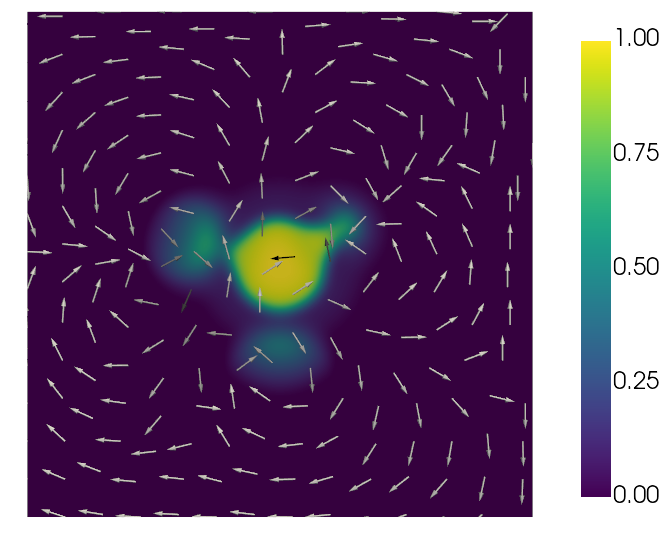}} &
			\raisebox{-0.47\height}{\includegraphics[scale=0.204]{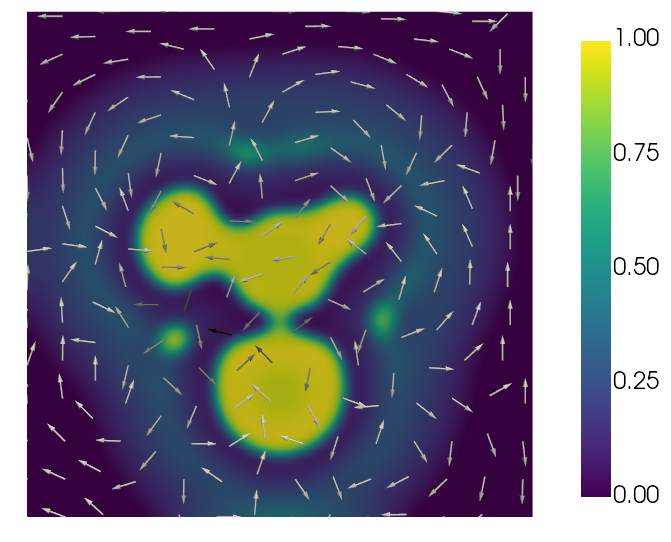}} &
			\raisebox{-0.47\height}{\includegraphics[scale=0.204]{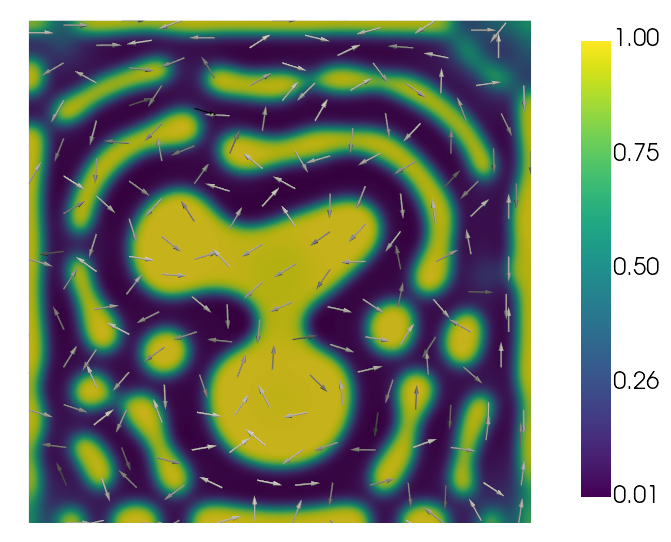}} \\
			& \rotatebox[origin=c]{90}{$K=10$} &
			\raisebox{-0.47\height}{\includegraphics[scale=0.204]{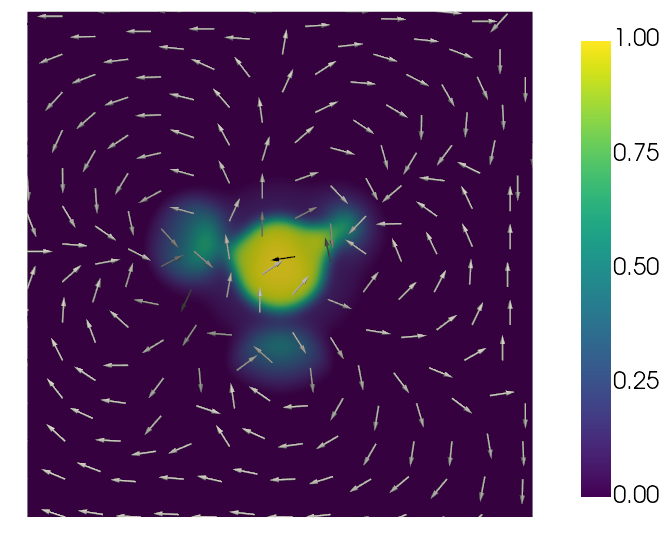}} &
			\raisebox{-0.47\height}{\includegraphics[scale=0.204]{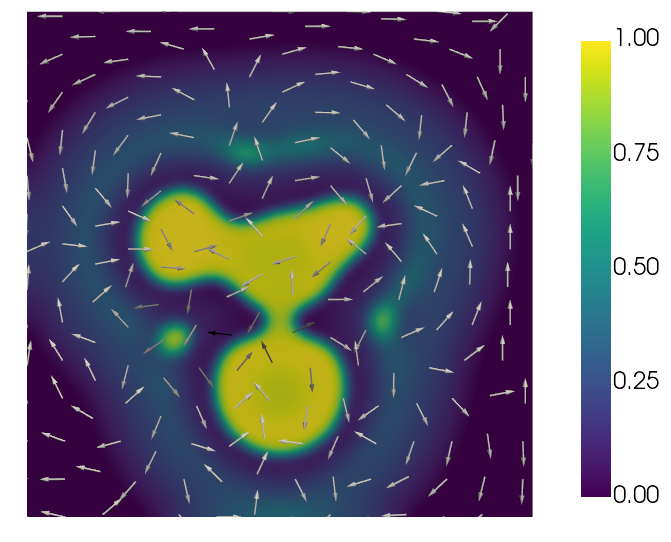}} &
			\raisebox{-0.47\height}{\includegraphics[scale=0.204]{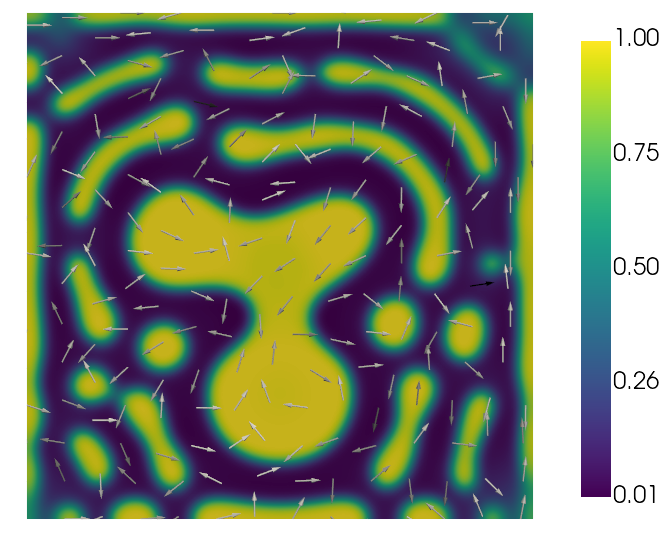}} \\
			\hdashline\vspace*{-0.5cm}\\
			\multirow{3}{*}{\vspace*{-6.2cm}\rotatebox[origin=c]{90}{\textbf{Non-symmetric \eqref{nonsymmetric_functions}}}} & \rotatebox[origin=c]{90}{$K=0.1$} &
			\raisebox{-0.47\height}{\includegraphics[scale=0.204]{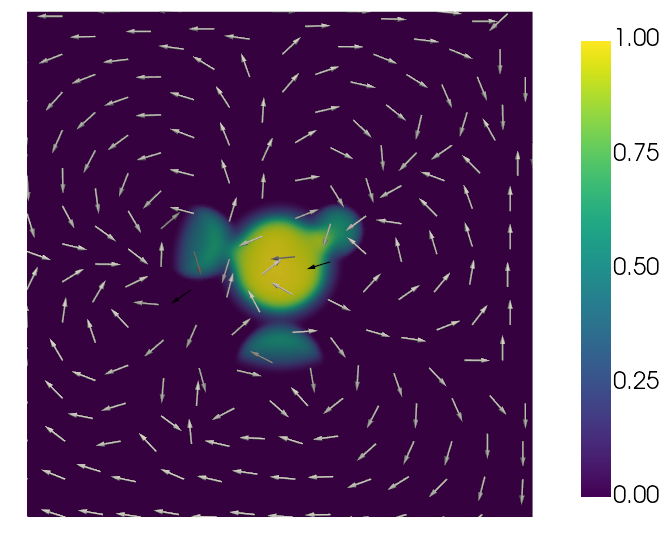}} &
			\raisebox{-0.47\height}{\includegraphics[scale=0.204]{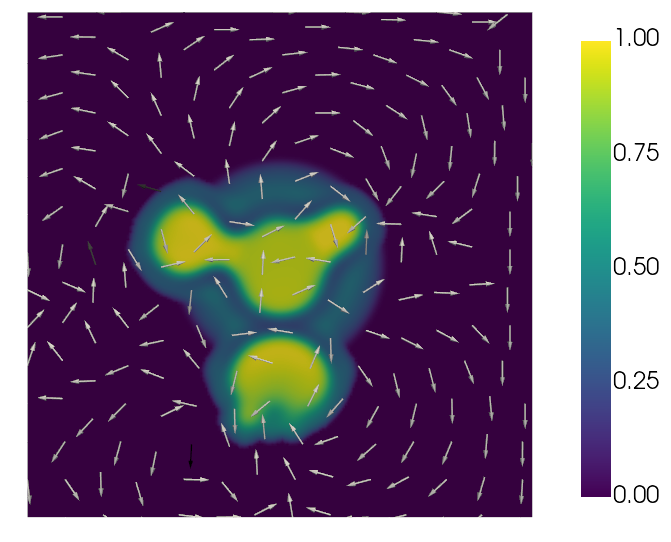}} &
			\raisebox{-0.47\height}{\includegraphics[scale=0.204]{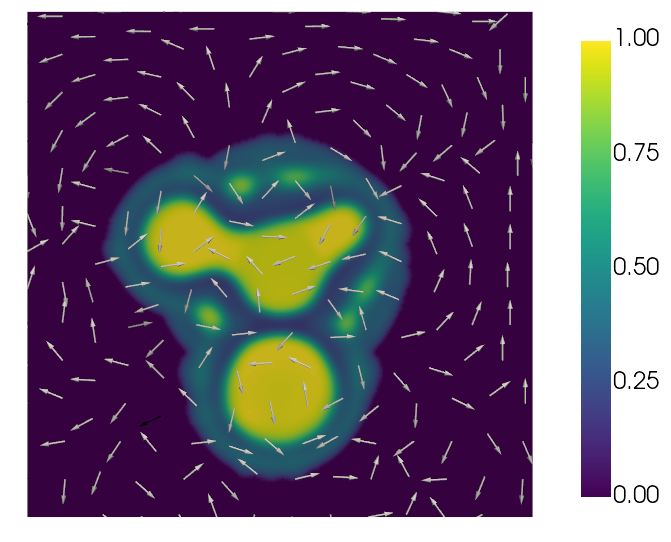}} \\
			& \rotatebox[origin=c]{90}{$K=1$} &
			\raisebox{-0.47\height}{\includegraphics[scale=0.204]{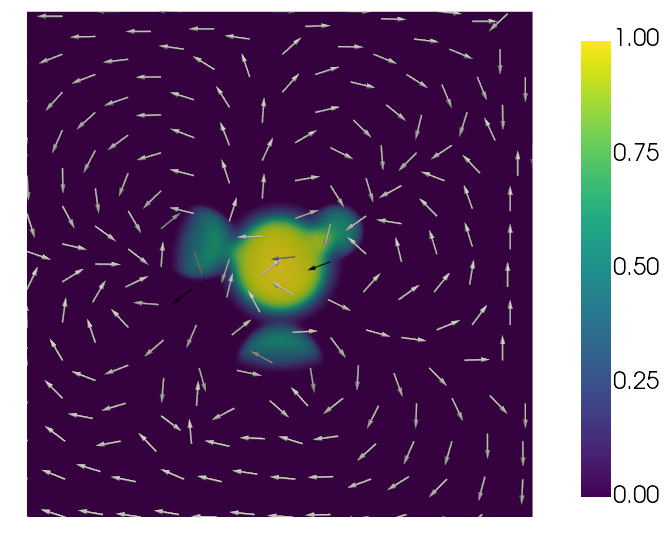}} &
			\raisebox{-0.47\height}{\includegraphics[scale=0.204]{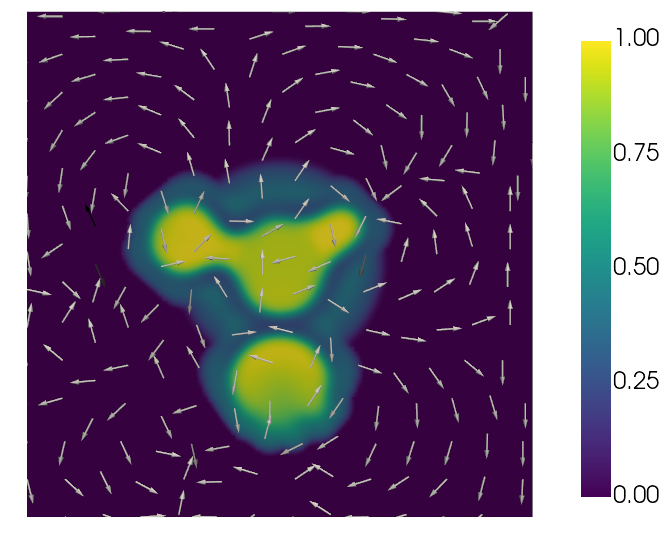}} &
			\raisebox{-0.47\height}{\includegraphics[scale=0.204]{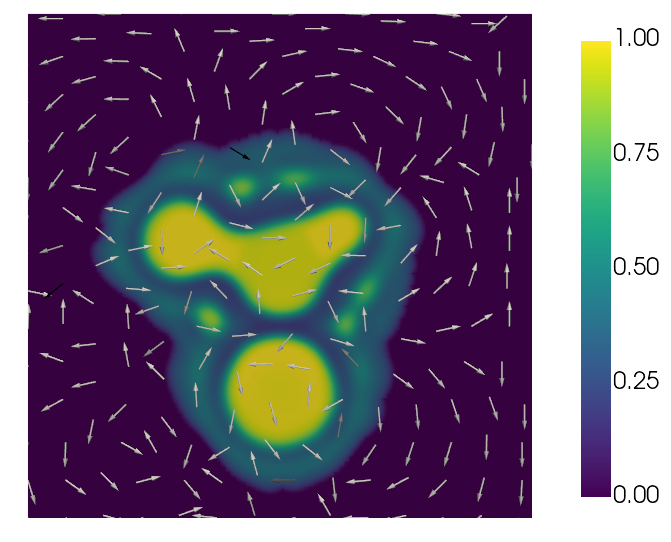}} \\
			& \rotatebox[origin=c]{90}{$K=10$} &
			\raisebox{-0.47\height}{\includegraphics[scale=0.204]{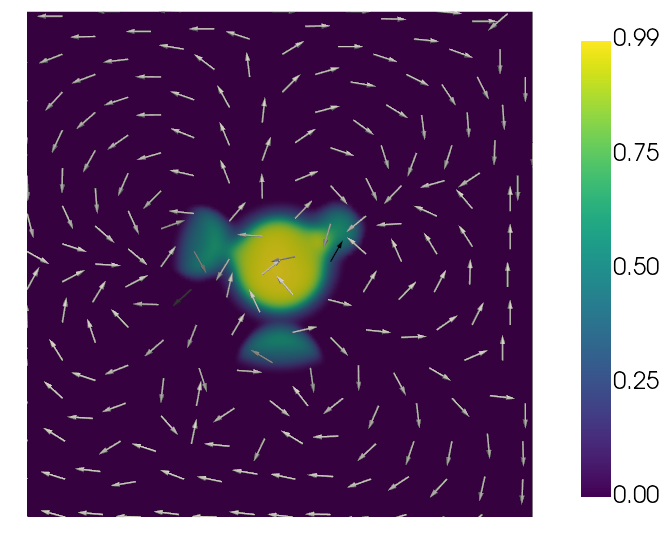}} &
			\raisebox{-0.47\height}{\includegraphics[scale=0.204]{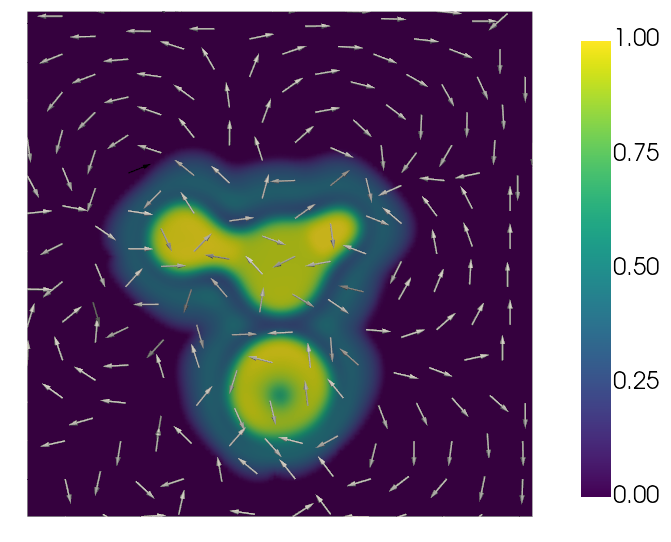}} &
			\raisebox{-0.47\height}{\includegraphics[scale=0.204]{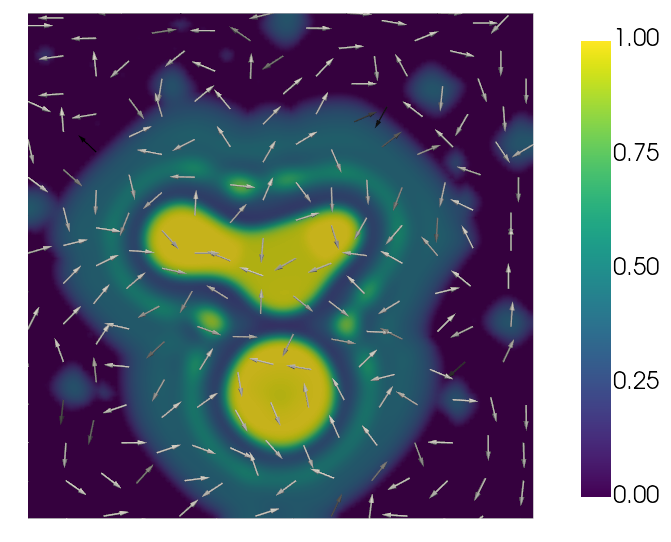}}
	\end{tabular}
	\caption{Tumor for test with $P_0=2$, $\chi_0=0.1$, $\Delta t=0.025$ at different time steps.}
	\label{fig:test-2_P0-2_u}
\end{figure}

\begin{figure}
	\centering
	\begin{tabular}{ccccc}
			&& \hspace*{-1cm}$t=10$ & \hspace*{-1cm}$t=20$ & \hspace*{-1cm}$t=50$ \\
			\multirow{3}{*}{\vspace*{-6.2cm}\rotatebox[origin=c]{90}{\textbf{Symmetric \eqref{symmetric_functions}}}} & \rotatebox[origin=c]{90}{$K=0.1$} &
			\raisebox{-0.47\height}{\includegraphics[scale=0.204]{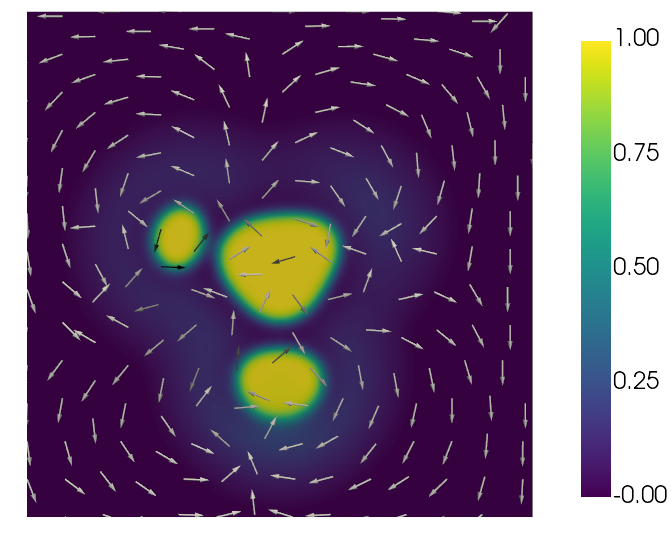}} &
			\raisebox{-0.47\height}{\includegraphics[scale=0.204]{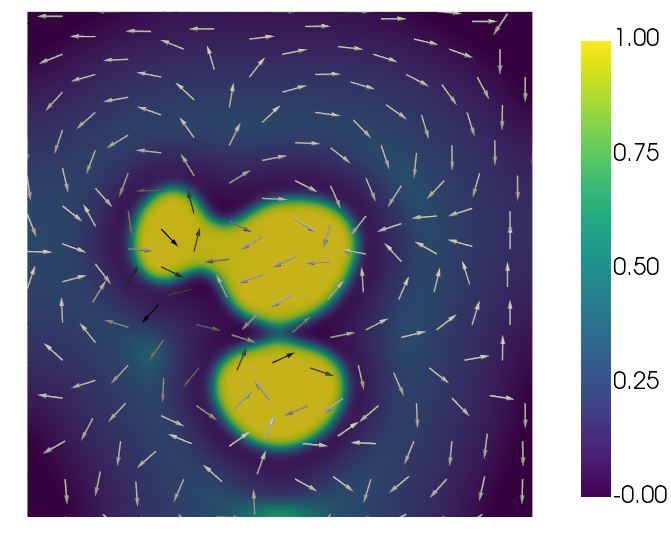}} &
			\raisebox{-0.47\height}{\includegraphics[scale=0.204]{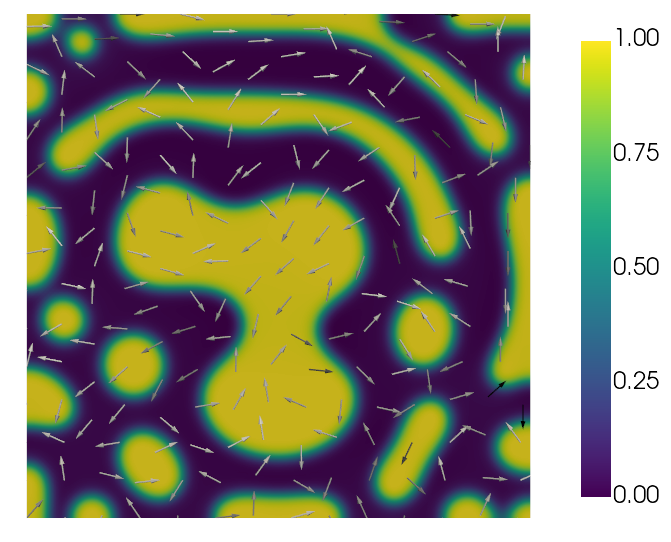}} \\
			& \rotatebox[origin=c]{90}{$K=1$} &
			\raisebox{-0.47\height}{\includegraphics[scale=0.204]{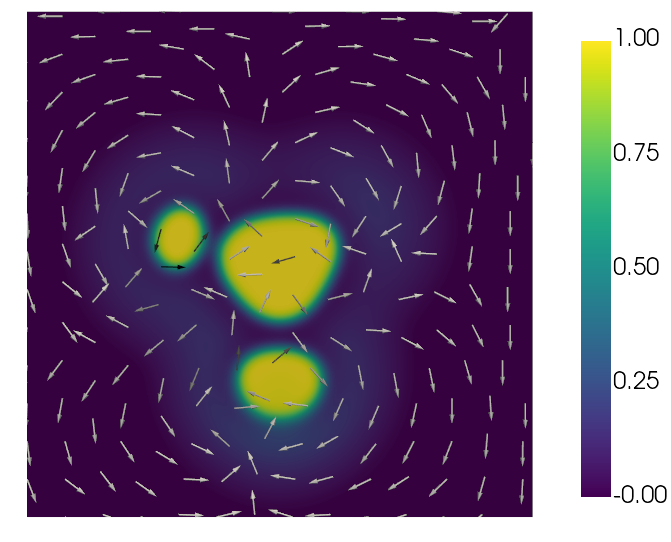}} &
			\raisebox{-0.47\height}{\includegraphics[scale=0.204]{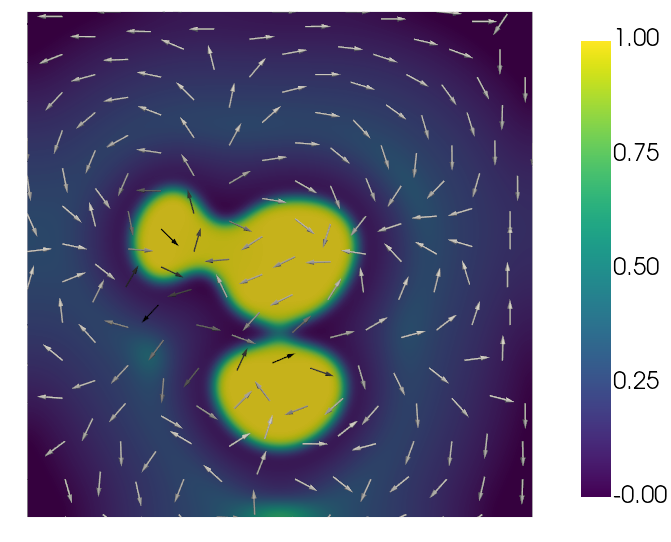}} &
			\raisebox{-0.47\height}{\includegraphics[scale=0.204]{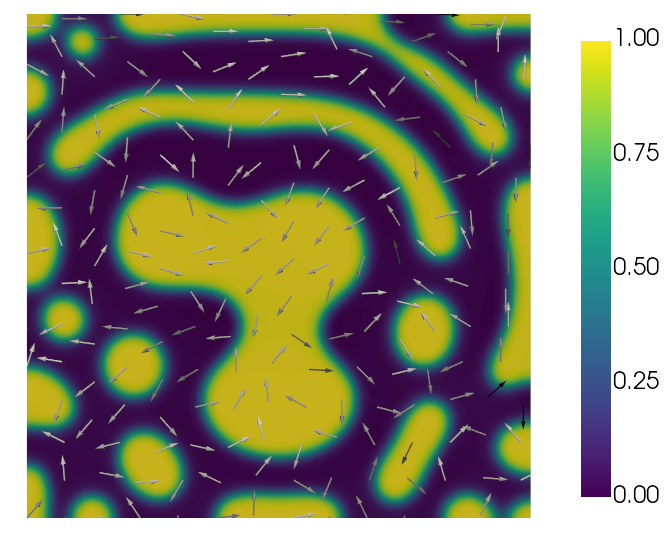}} \\
			& \rotatebox[origin=c]{90}{$K=10$} &
			\raisebox{-0.47\height}{\includegraphics[scale=0.204]{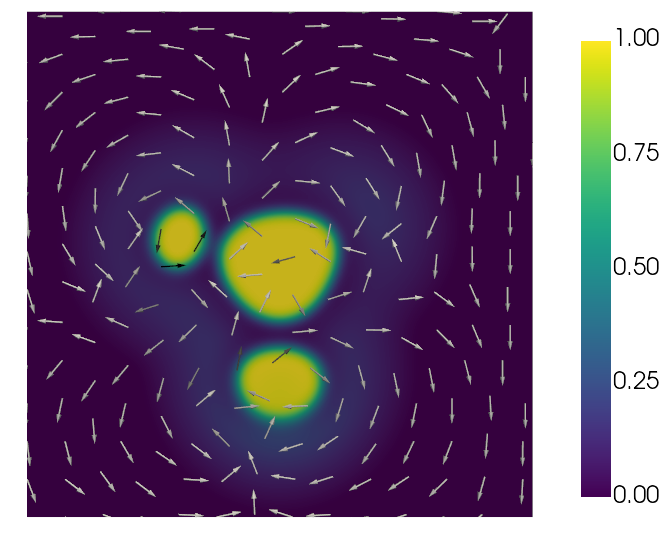}} &
			\raisebox{-0.47\height}{\includegraphics[scale=0.204]{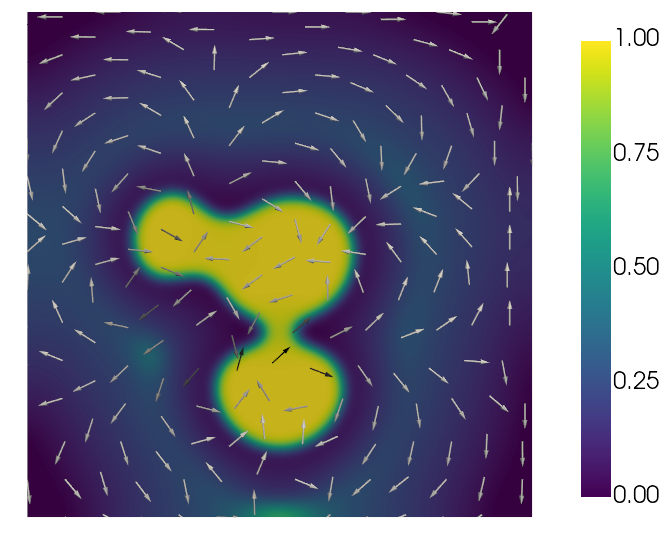}} &
			\raisebox{-0.47\height}{\includegraphics[scale=0.204]{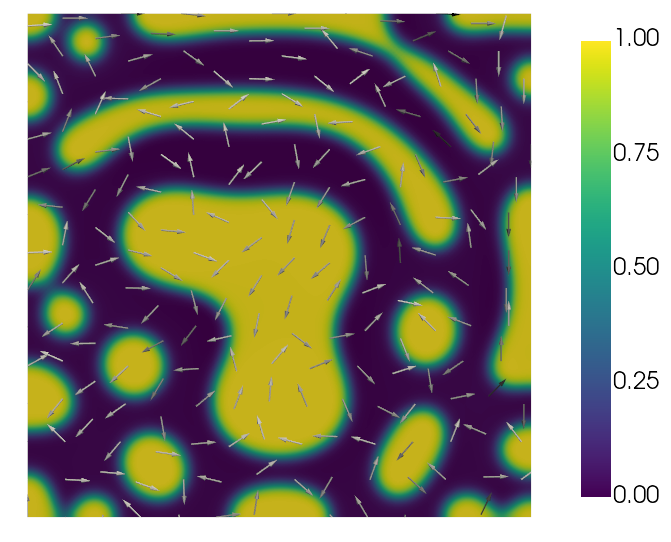}} \\
			\hdashline\vspace*{-0.5cm}\\
			\multirow{3}{*}{\vspace*{-6.2cm}\rotatebox[origin=c]{90}{\textbf{Non-symmetric \eqref{nonsymmetric_functions}}}} & \rotatebox[origin=c]{90}{$K=0.1$} &
			\raisebox{-0.47\height}{\includegraphics[scale=0.204]{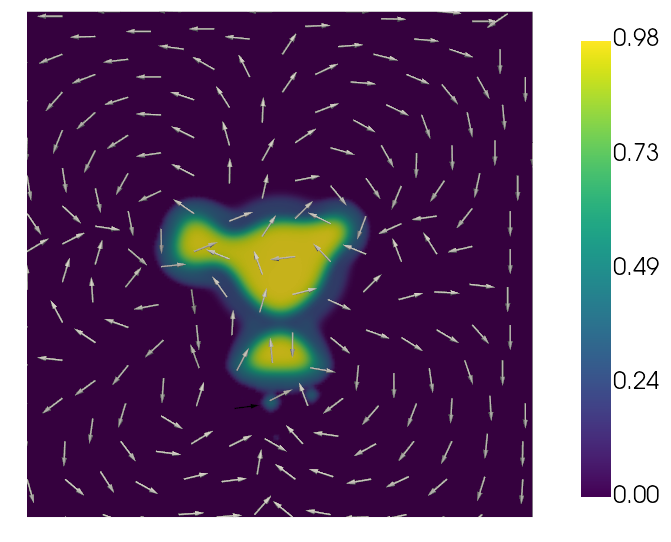}} &
			\raisebox{-0.47\height}{\includegraphics[scale=0.204]{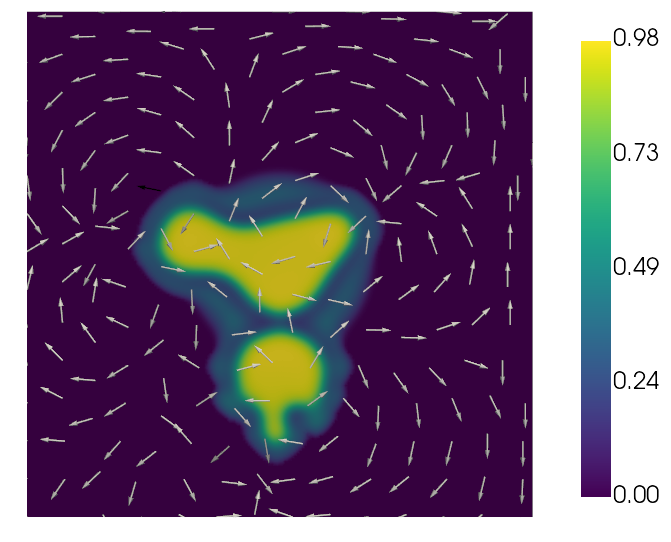}} &
			\raisebox{-0.47\height}{\includegraphics[scale=0.204]{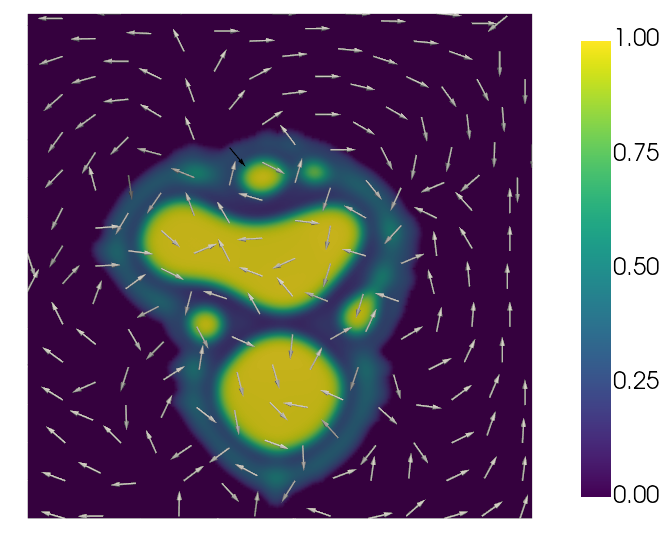}} \\
			& \rotatebox[origin=c]{90}{$K=1$} &
			\raisebox{-0.47\height}{\includegraphics[scale=0.204]{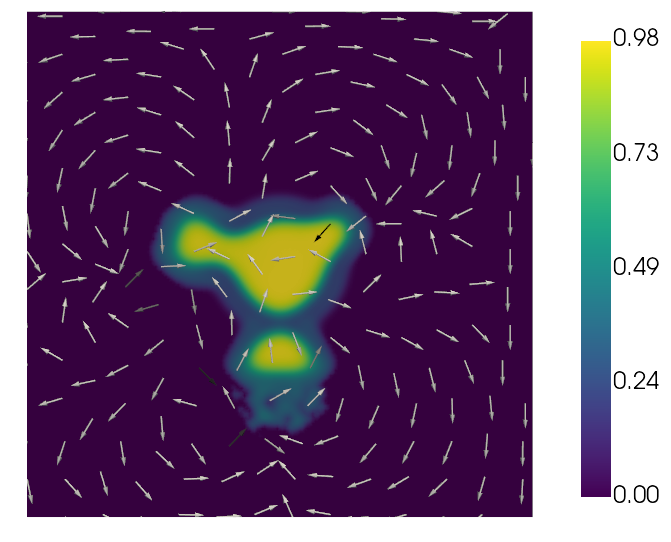}} &
			\raisebox{-0.47\height}{\includegraphics[scale=0.204]{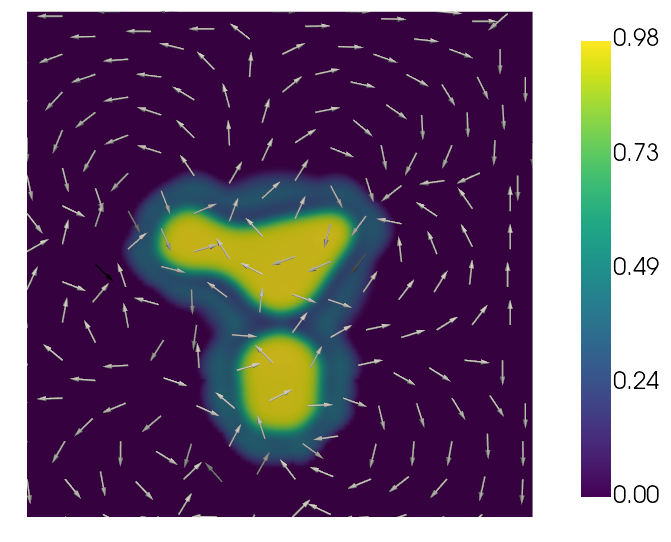}} &
			\raisebox{-0.47\height}{\includegraphics[scale=0.204]{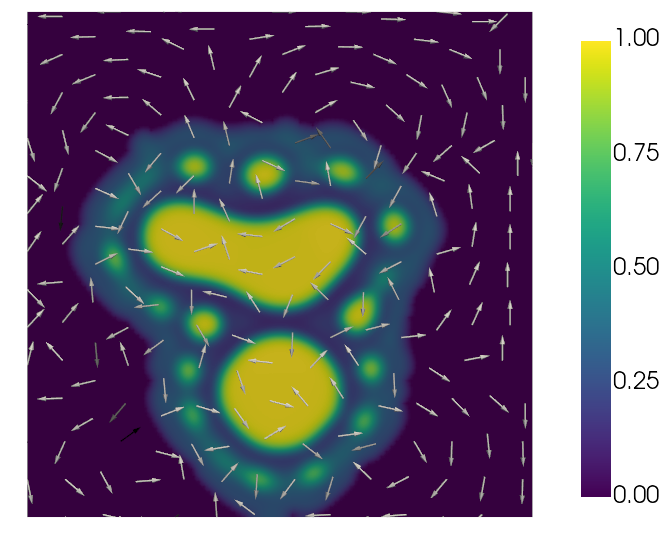}} \\
			& \rotatebox[origin=c]{90}{$K=10$} &
			\raisebox{-0.47\height}{\includegraphics[scale=0.204]{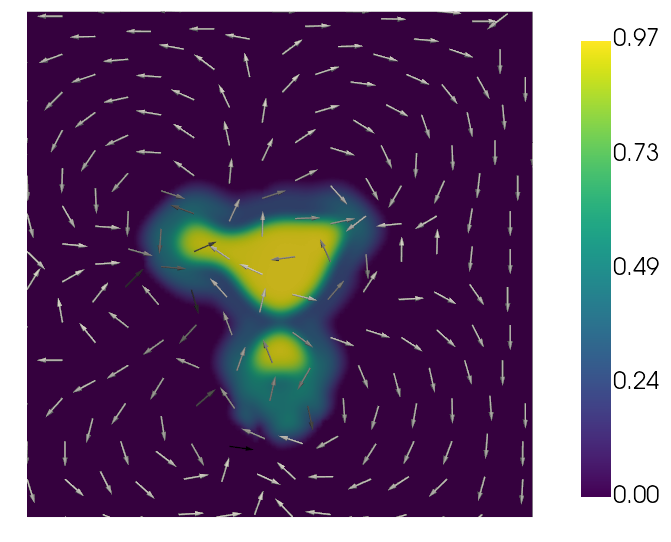}} &
			\raisebox{-0.47\height}{\includegraphics[scale=0.204]{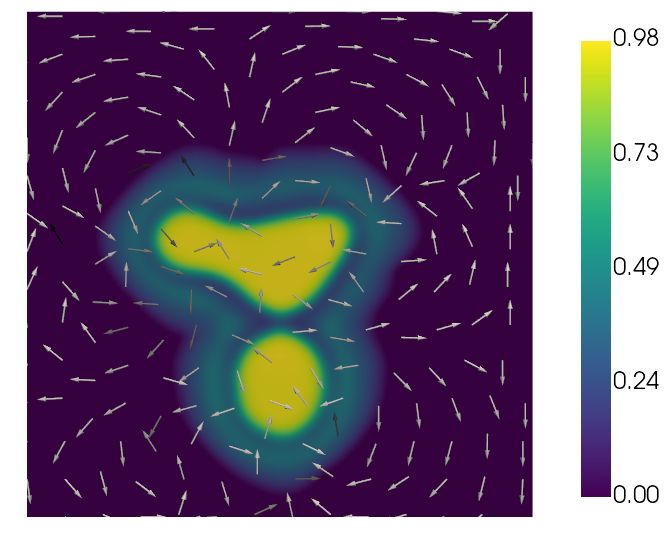}} &
			\raisebox{-0.47\height}{\includegraphics[scale=0.204]{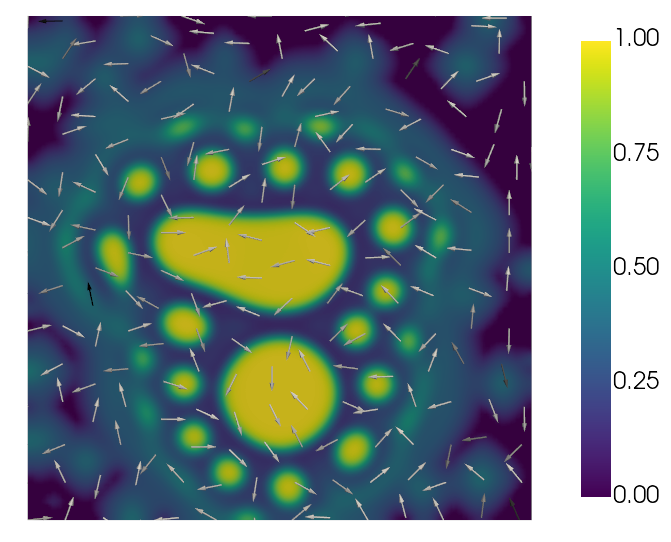}}
	\end{tabular}
	\caption{Tumor for test with $P_0=0.5$, $\chi_0=0.01$, $\Delta t=0.1$ at different time steps.}
	\label{fig:test-2_chi-0.01_u}
\end{figure}

\begin{figure}
	\centering
	\begin{tabular}{ccccc}
			&& \hspace*{-1cm}$t=3$ & \hspace*{-1cm}$t=10$ & \hspace*{-1cm}$t=17$ \\
			\multirow{3}{*}{\vspace*{-6.2cm}\rotatebox[origin=c]{90}{\textbf{Symmetric \eqref{symmetric_functions}}}} & \rotatebox[origin=c]{90}{$K=0.1$} &
			\raisebox{-0.47\height}{\includegraphics[scale=0.204]{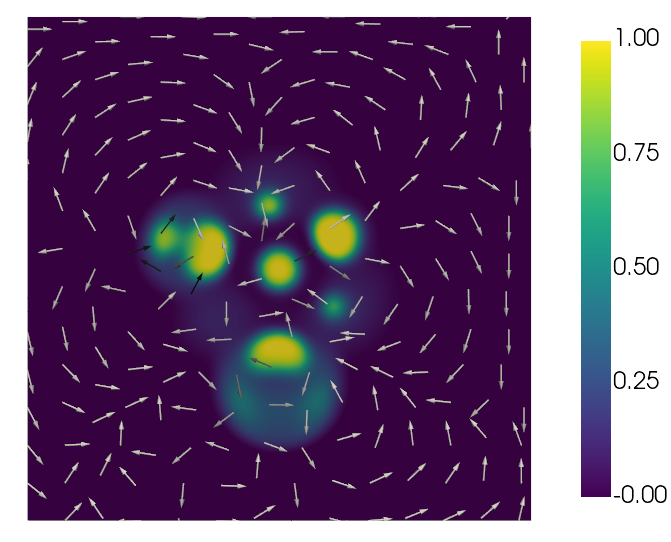}} &
			\raisebox{-0.47\height}{\includegraphics[scale=0.204]{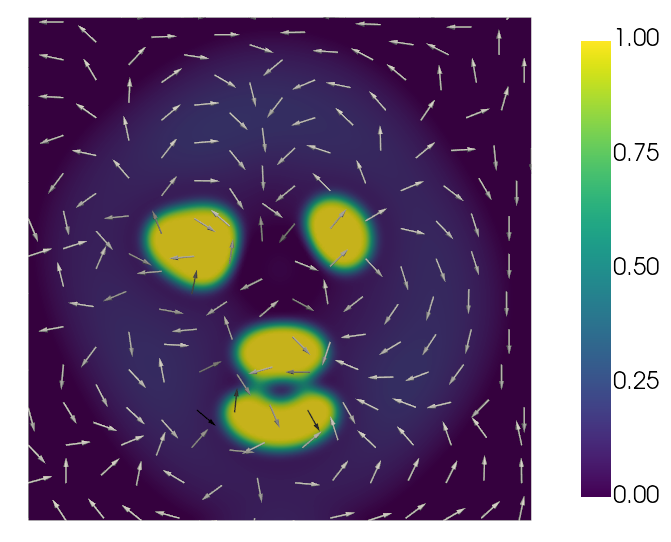}} &
			\raisebox{-0.47\height}{\includegraphics[scale=0.204]{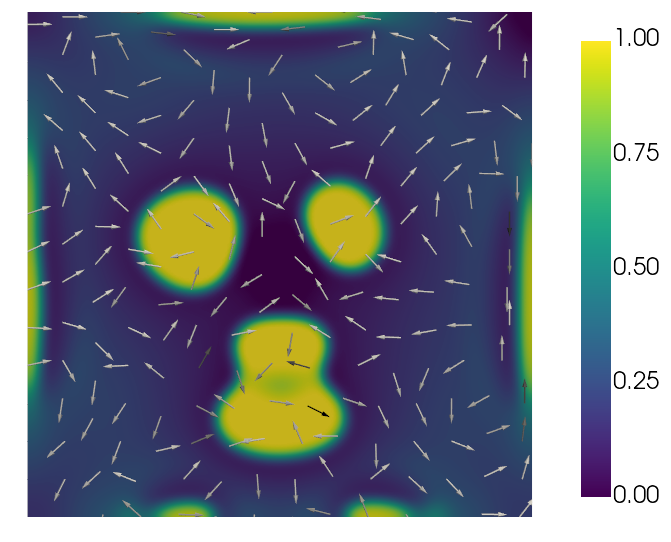}} \\
			& \rotatebox[origin=c]{90}{$K=1$} &
			\raisebox{-0.47\height}{\includegraphics[scale=0.204]{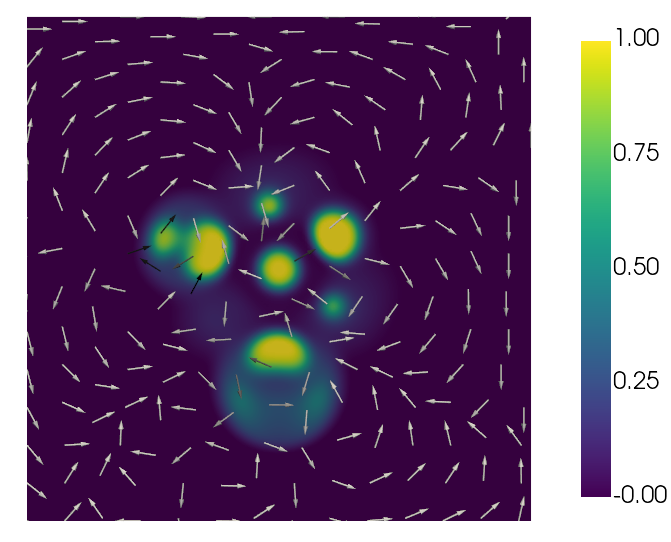}} &
			\raisebox{-0.47\height}{\includegraphics[scale=0.204]{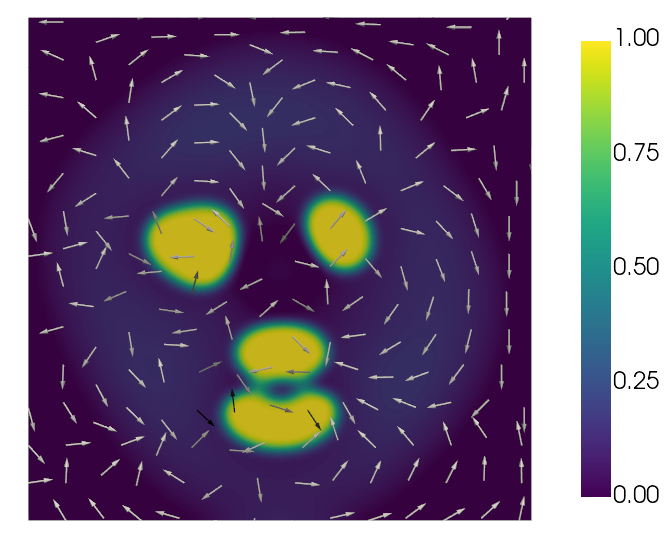}} &
			\raisebox{-0.47\height}{\includegraphics[scale=0.204]{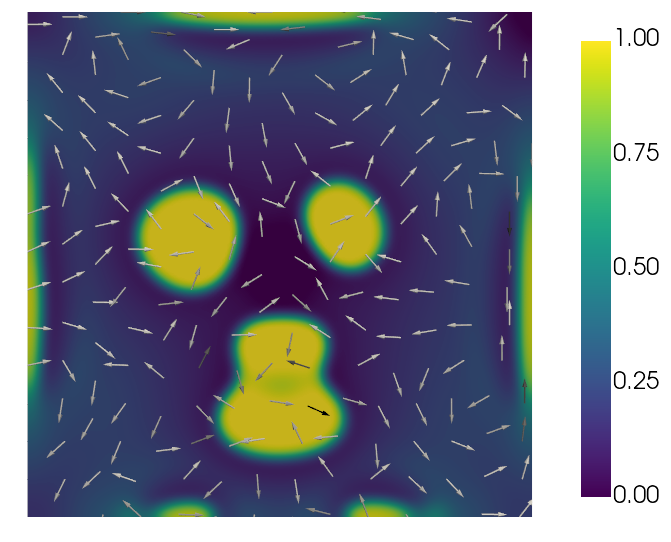}} \\
			& \rotatebox[origin=c]{90}{$K=10$} &
			\raisebox{-0.47\height}{\includegraphics[scale=0.204]{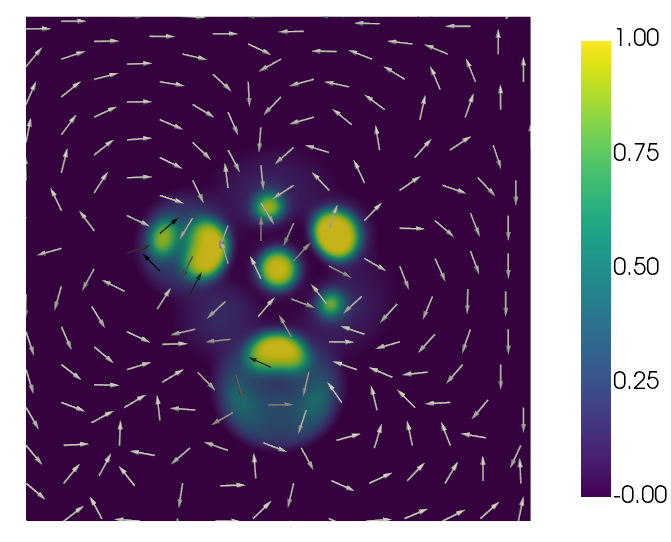}} &
			\raisebox{-0.47\height}{\includegraphics[scale=0.204]{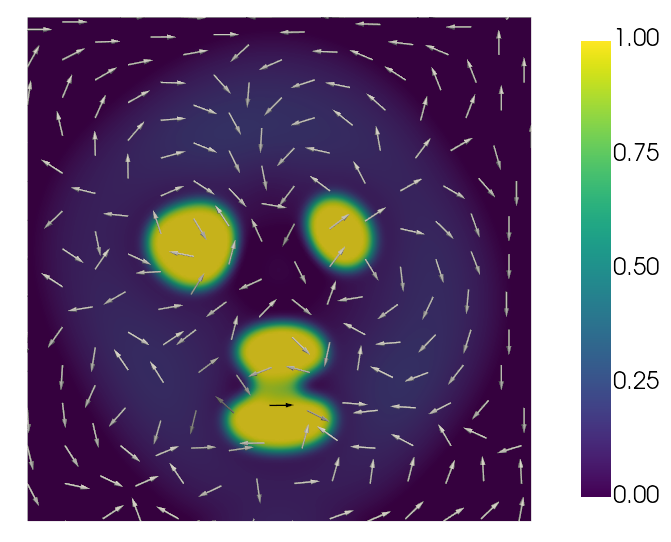}} &
			\raisebox{-0.47\height}{\includegraphics[scale=0.204]{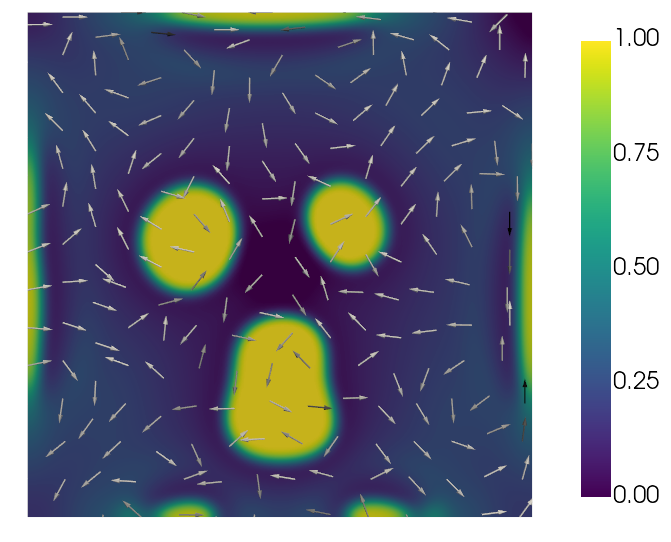}} \\
			\hdashline\vspace*{-0.5cm}\\
			\multirow{3}{*}{\vspace*{-6.2cm}\rotatebox[origin=c]{90}{\textbf{Non-symmetric \eqref{nonsymmetric_functions}}}} & \rotatebox[origin=c]{90}{$K=0.1$} &
			\raisebox{-0.47\height}{\includegraphics[scale=0.204]{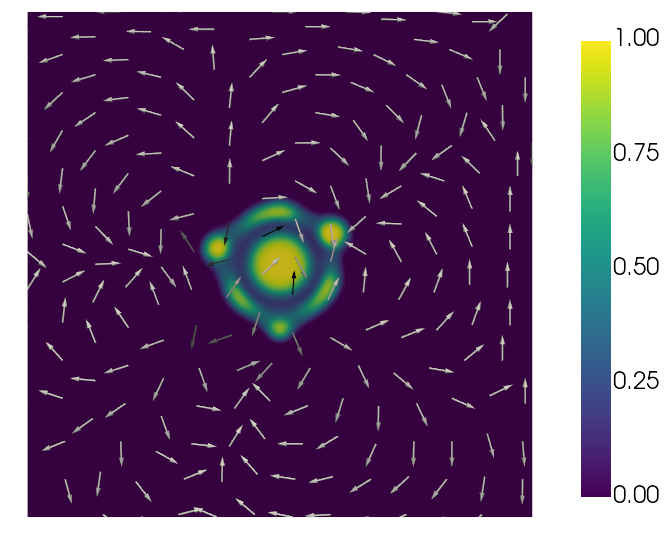}} &
			\raisebox{-0.47\height}{\includegraphics[scale=0.204]{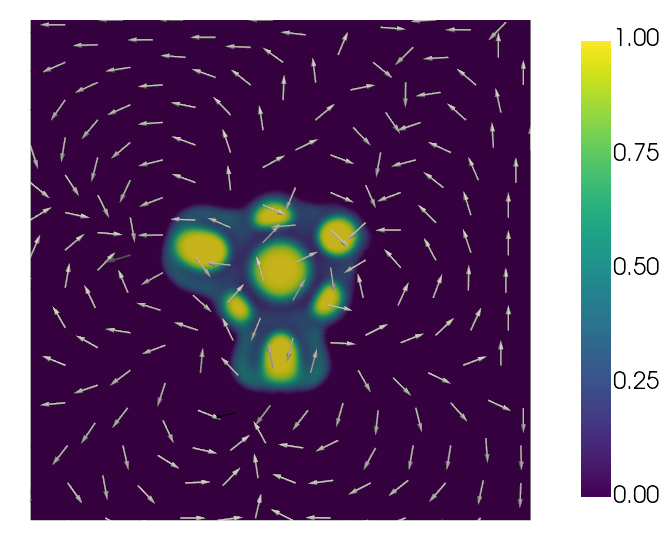}} &
			\raisebox{-0.47\height}{\includegraphics[scale=0.204]{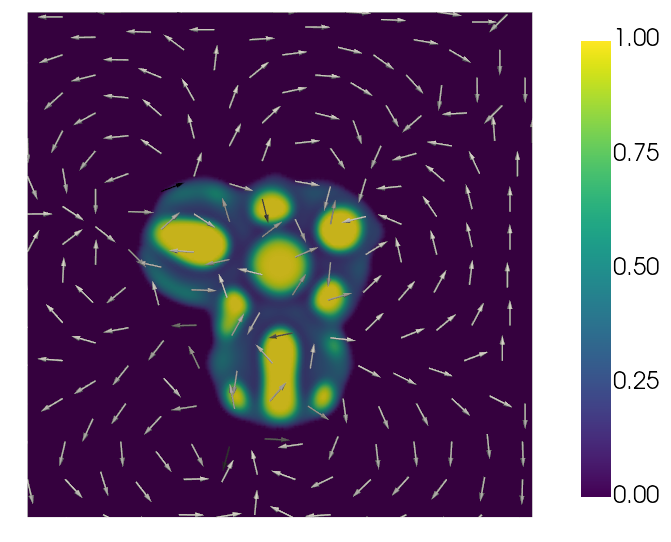}} \\
			& \rotatebox[origin=c]{90}{$K=1$} &
			\raisebox{-0.47\height}{\includegraphics[scale=0.204]{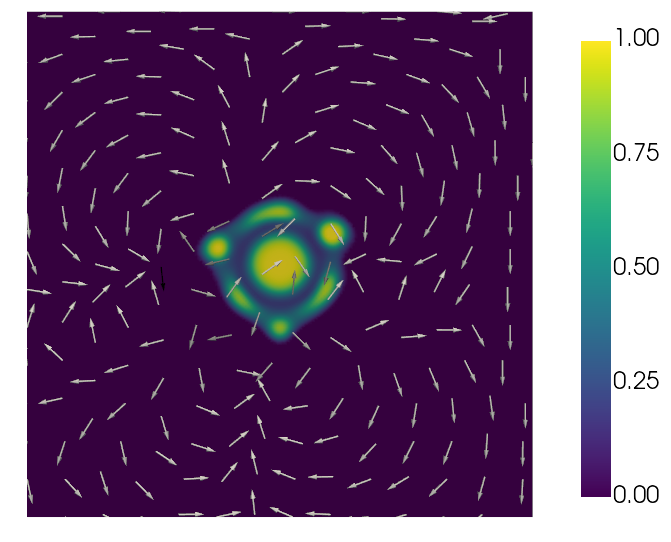}} &
			\raisebox{-0.47\height}{\includegraphics[scale=0.204]{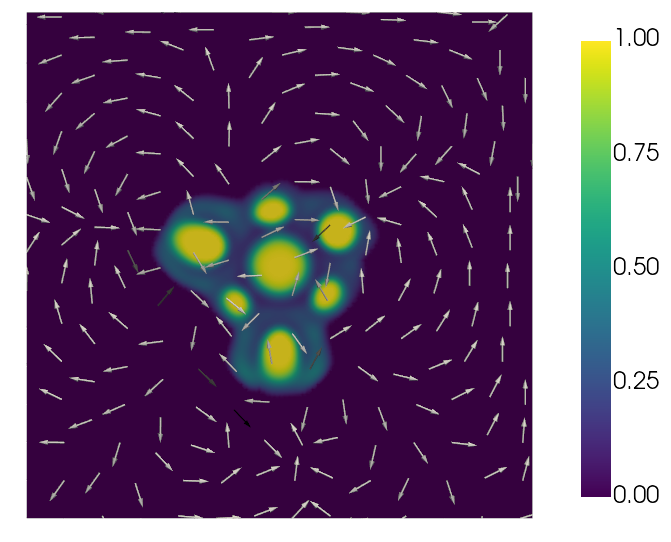}} &
			\raisebox{-0.47\height}{\includegraphics[scale=0.204]{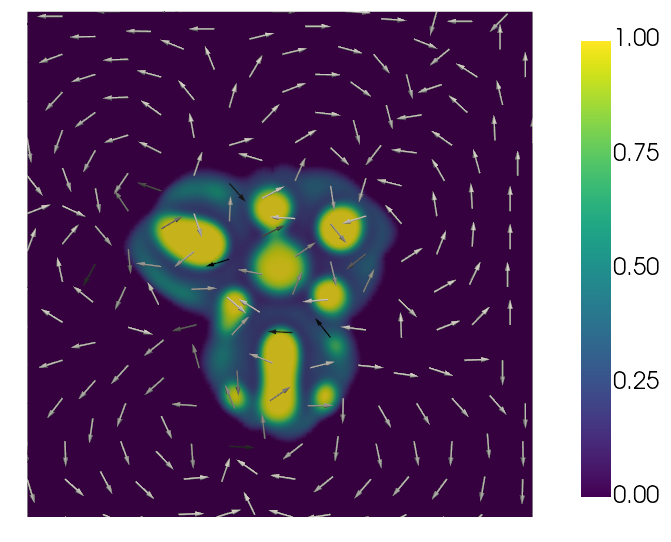}} \\
			& \rotatebox[origin=c]{90}{$K=10$} &
			\raisebox{-0.47\height}{\includegraphics[scale=0.204]{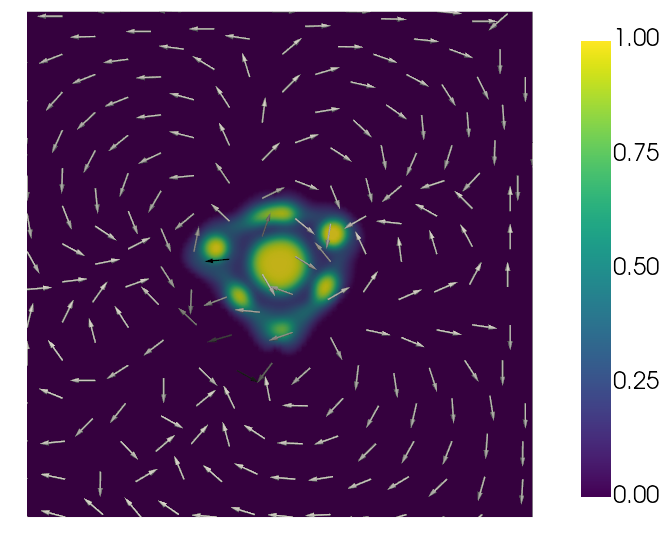}} &
			\raisebox{-0.47\height}{\includegraphics[scale=0.204]{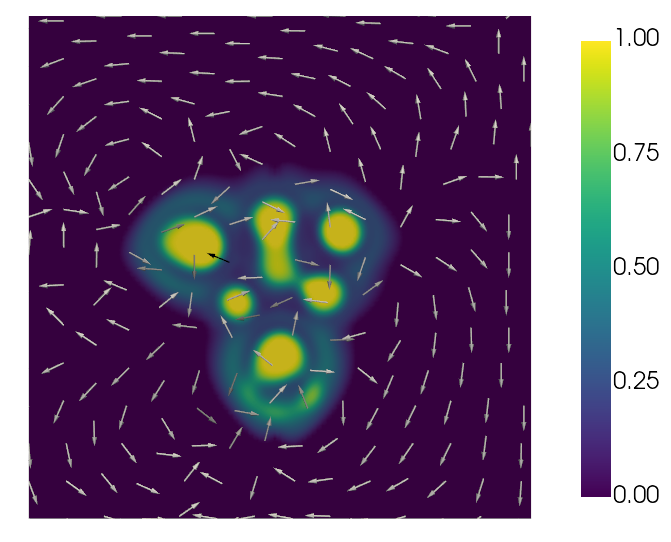}} &
			\raisebox{-0.47\height}{\includegraphics[scale=0.204]{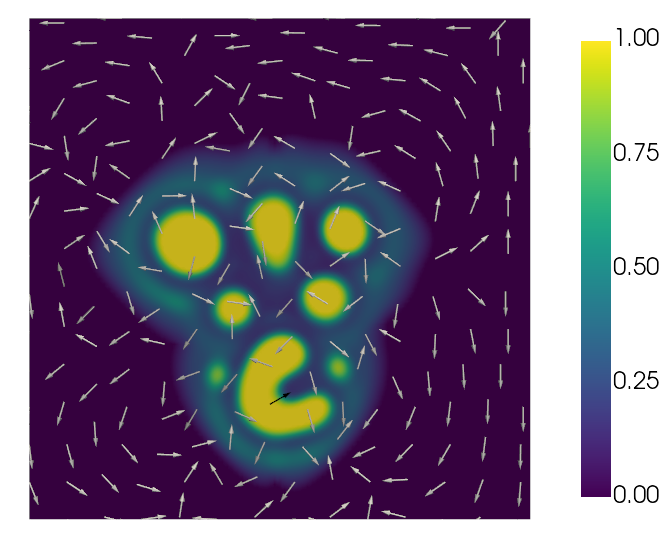}}
	\end{tabular}
	\caption{Tumor for test with $P_0=0.5$, $\chi_0=0.5$, $\Delta t=0.01$ at different time steps.}
	\label{fig:test-2_chi-0.5_u}
\end{figure}

\begin{figure}
	\centering
	\begin{tabular}{ccccc}
			&& \hspace*{-1cm}$t=2.5$ & \hspace*{-1cm}$t=5$ & \hspace*{-1cm}$t=10$ \\
			\multirow{3}{*}{\vspace*{-6.2cm}\rotatebox[origin=c]{90}{\textbf{Symmetric \eqref{symmetric_functions}}}} & \rotatebox[origin=c]{90}{$K=0.1$} &
			\raisebox{-0.47\height}{\includegraphics[scale=0.204]{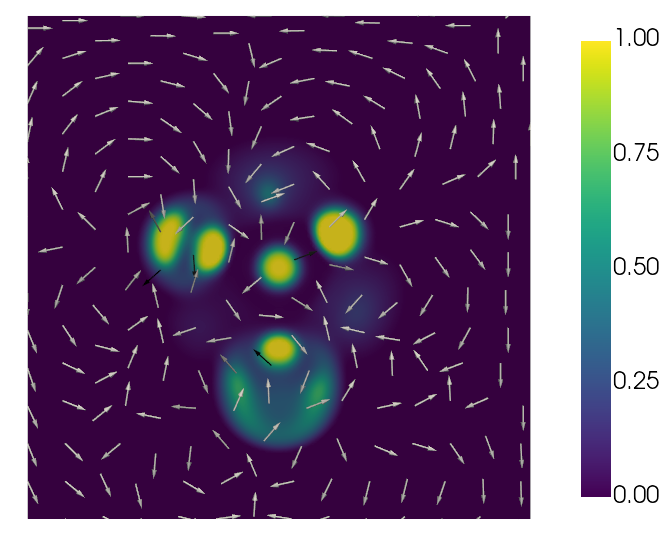}} &
			\raisebox{-0.47\height}{\includegraphics[scale=0.204]{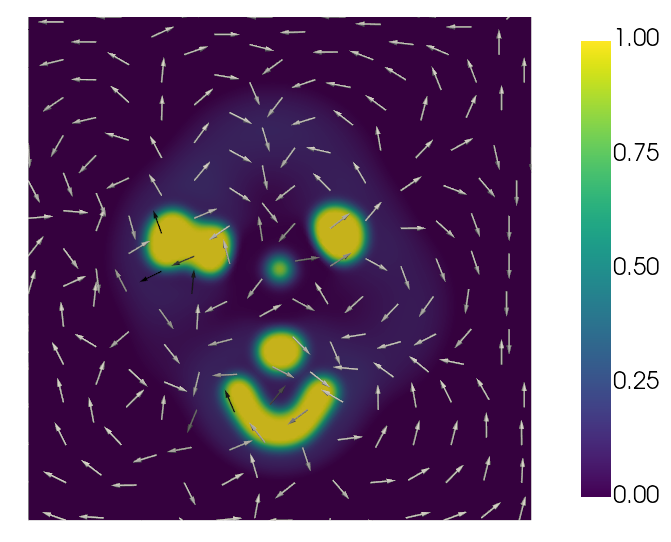}} &
			\raisebox{-0.47\height}{\includegraphics[scale=0.204]{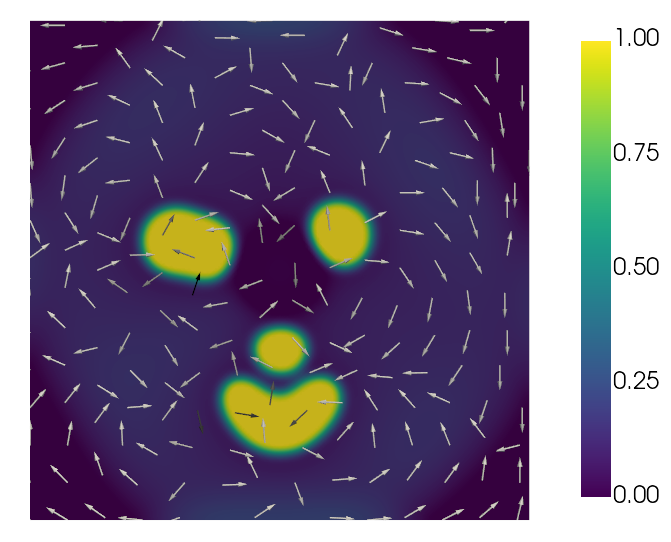}} \\
			& \rotatebox[origin=c]{90}{$K=1$} &
			\raisebox{-0.47\height}{\includegraphics[scale=0.204]{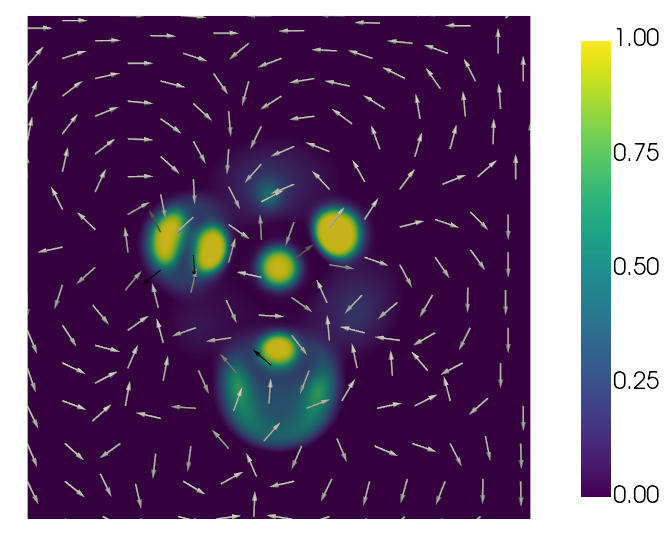}} &
			\raisebox{-0.47\height}{\includegraphics[scale=0.204]{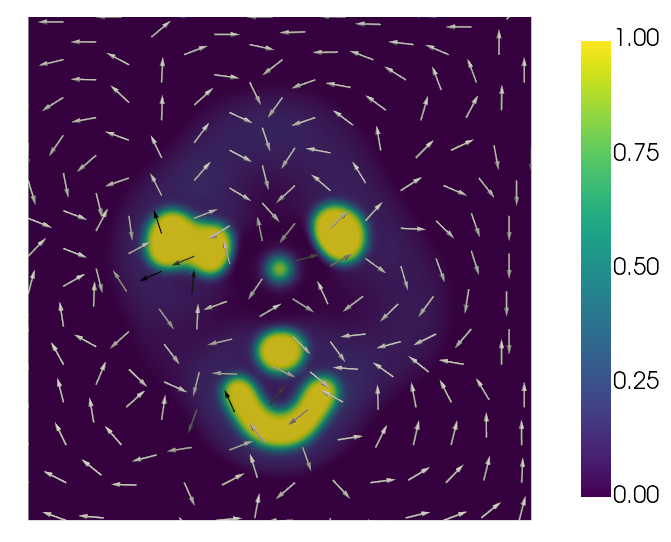}} &
			\raisebox{-0.47\height}{\includegraphics[scale=0.204]{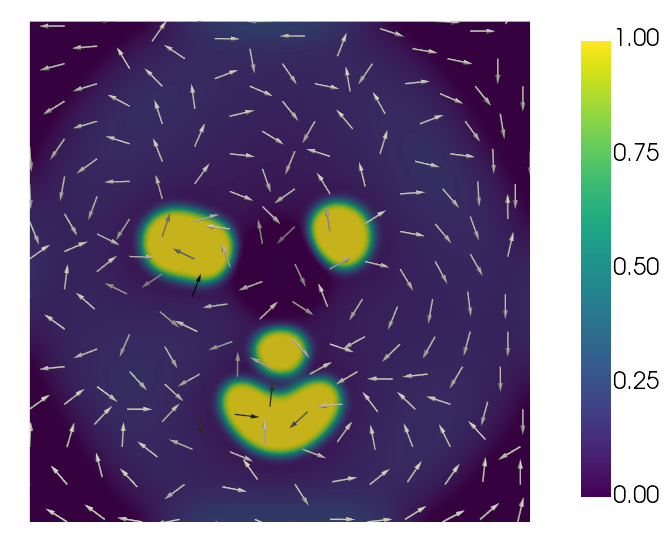}} \\
			& \rotatebox[origin=c]{90}{$K=10$} &
			\raisebox{-0.47\height}{\includegraphics[scale=0.204]{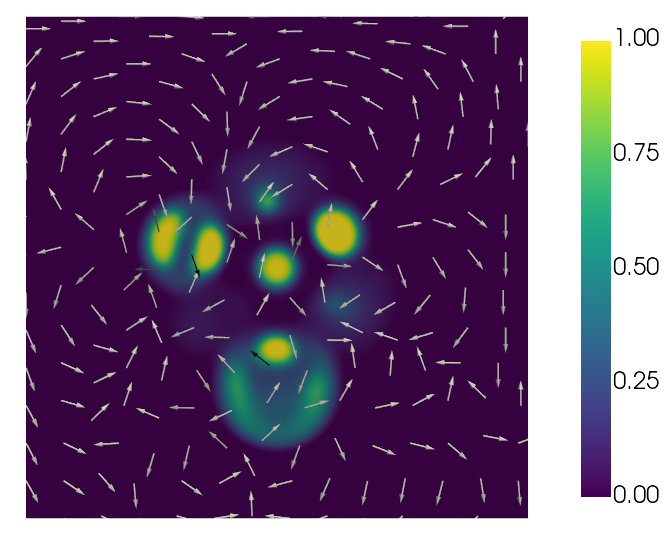}} &
			\raisebox{-0.47\height}{\includegraphics[scale=0.204]{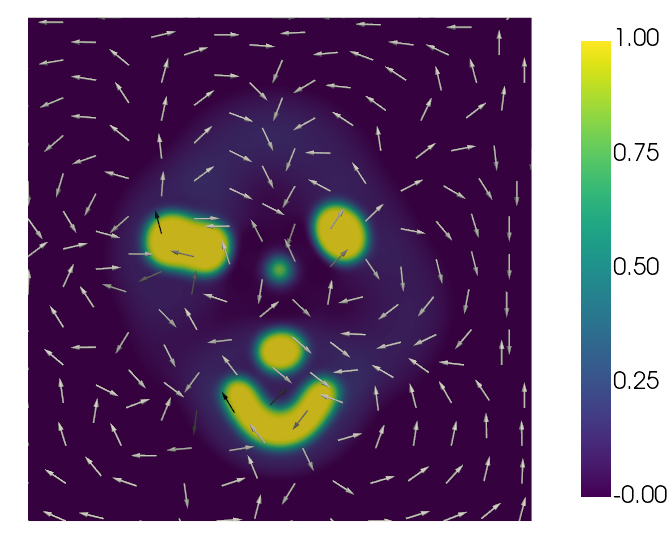}} &
			\raisebox{-0.47\height}{\includegraphics[scale=0.204]{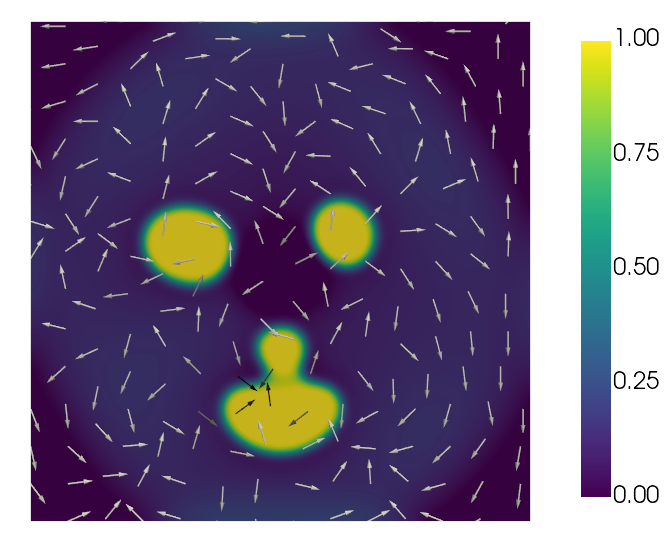}} \\
			\hdashline\vspace*{-0.5cm}\\
			\multirow{3}{*}{\vspace*{-6.2cm}\rotatebox[origin=c]{90}{\textbf{Non-symmetric \eqref{nonsymmetric_functions}}}} & \rotatebox[origin=c]{90}{$K=0.1$} &
			\raisebox{-0.47\height}{\includegraphics[scale=0.204]{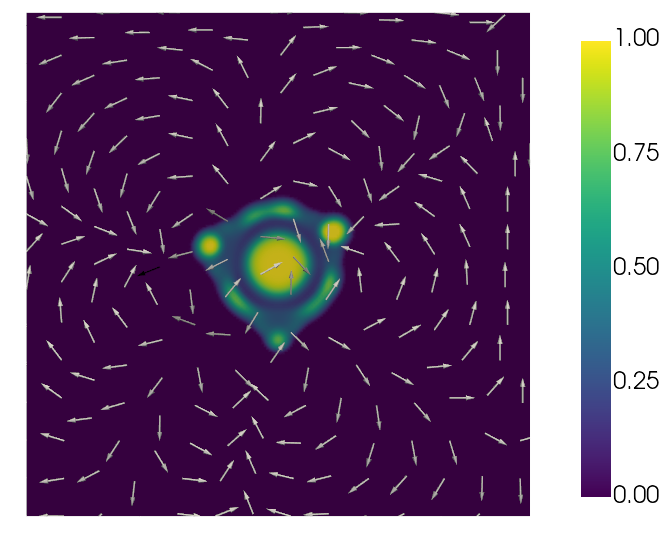}} &
			\raisebox{-0.47\height}{\includegraphics[scale=0.204]{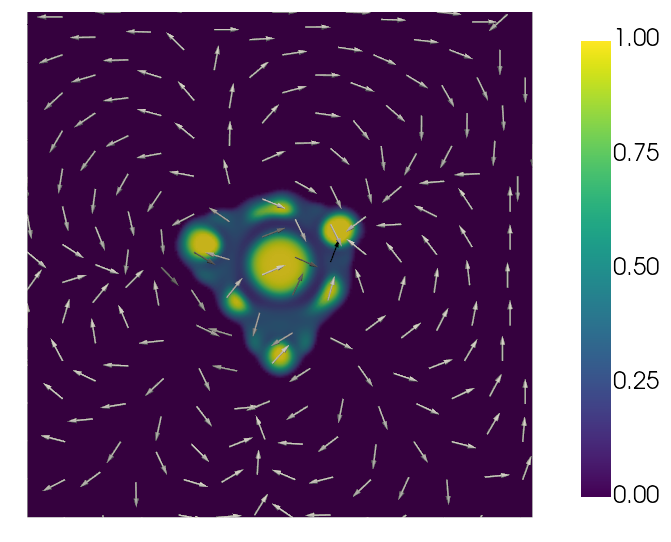}} &
			\raisebox{-0.47\height}{\includegraphics[scale=0.204]{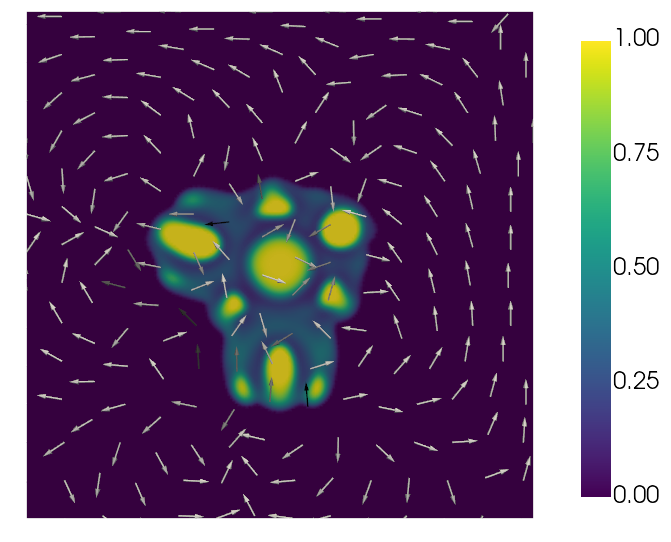}} \\
			& \rotatebox[origin=c]{90}{$K=1$} &
			\raisebox{-0.47\height}{\includegraphics[scale=0.204]{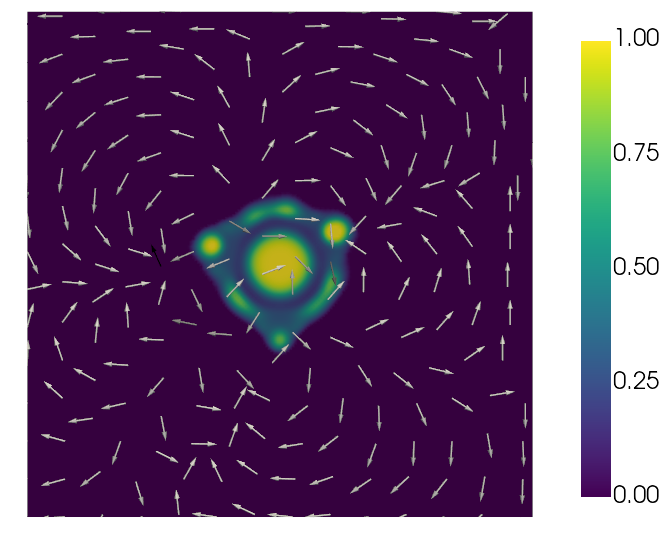}} &
			\raisebox{-0.47\height}{\includegraphics[scale=0.204]{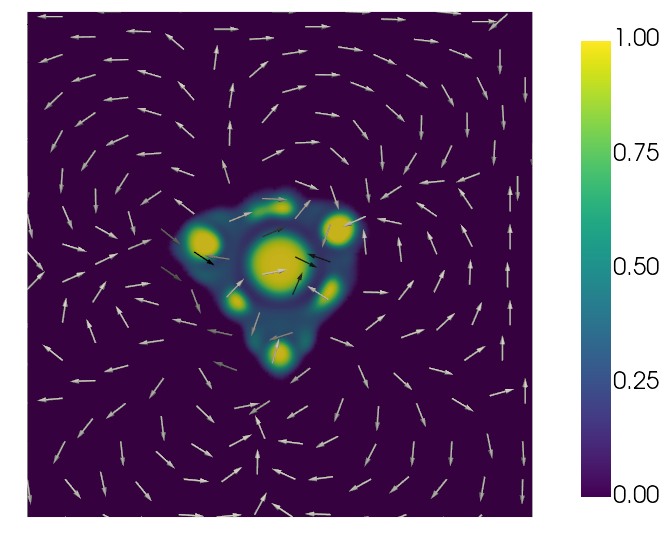}} &
			\raisebox{-0.47\height}{\includegraphics[scale=0.204]{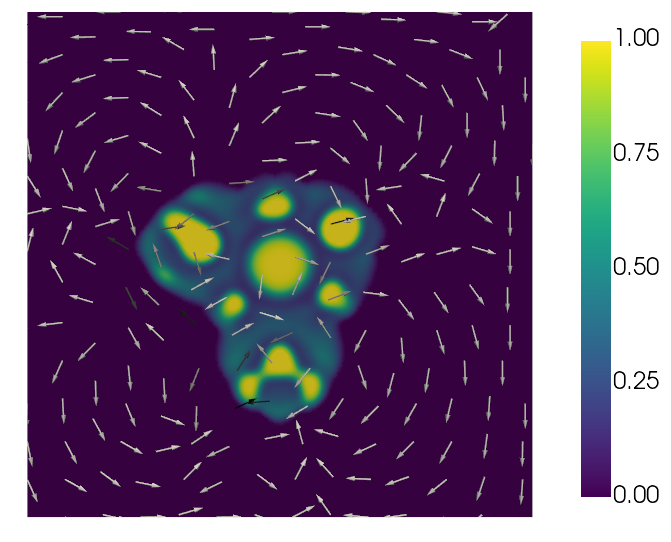}} \\
			& \rotatebox[origin=c]{90}{$K=10$} &
			\raisebox{-0.47\height}{\includegraphics[scale=0.204]{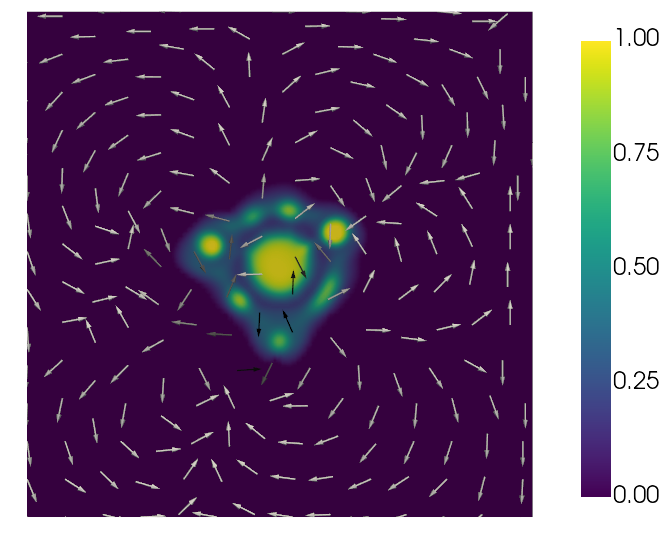}} &
			\raisebox{-0.47\height}{\includegraphics[scale=0.204]{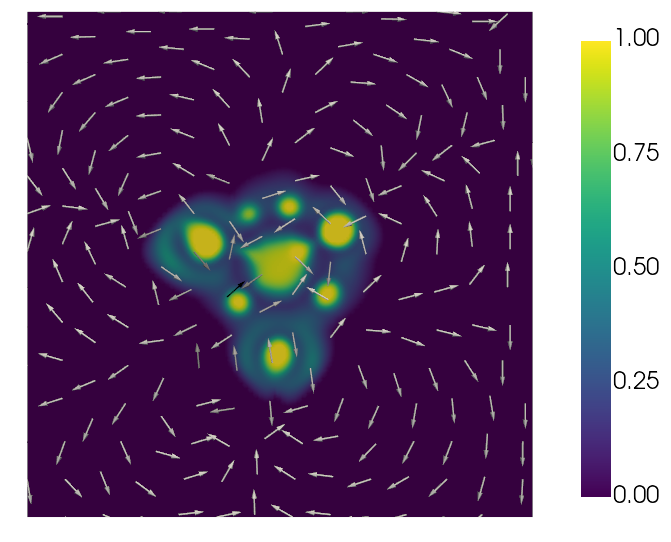}} &
			\raisebox{-0.47\height}{\includegraphics[scale=0.204]{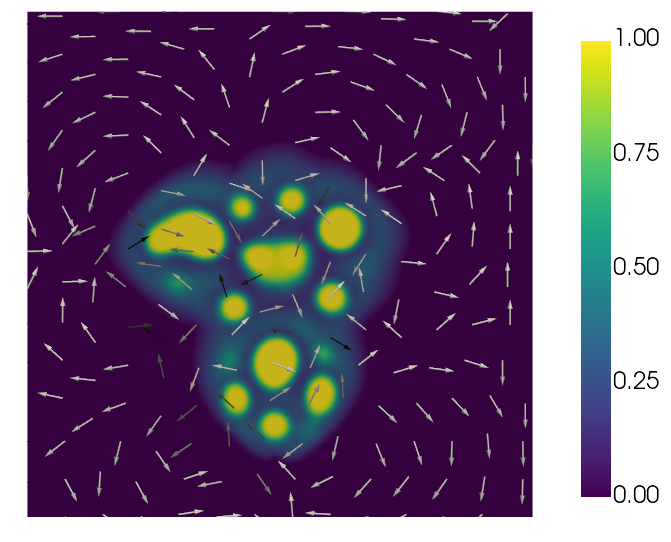}}
	\end{tabular}
	\caption{Tumor for test with $P_0=0.5$, $\chi_0=1$, $\Delta t=0.01$ at different time steps.}
	\label{fig:test-2_chi-1_u}
\end{figure}

\section*{Acknowledgements}

D. Acosta-Soba and F. Guill\'en-Gonz\'alez have been partially supported by Grant
I+D+I PID2023-149182NB-I00 funded by MICIU/AEI/10.13039/501100011033.

Also, F. Guill\'en-Gonz\'alez has been partially supported by an ERDF/EU and IMUS-Maria de Maeztu grant CEX2024-001517-M - Apoyo a Unidades de Excelencia María de Maeztu, funded by MICIU/AEI/10.13039/501100011033.

Finally, D. Acosta-Soba and R. Rodríguez-Galván have been partially supported by grants PR2024-011 and PR2024-039 funded by the Universidad de Cádiz.

\printbibliography

\end{document}